%% file: DaKrSo20-10-12.tex
\documentclass[11pt,a4paper]{amsart}
\usepackage{amscd,amssymb,amsmath,latexsym,url,mathrsfs}
\usepackage[all]{xy}
\usepackage{color,graphicx}
\usepackage{hyperref}
\def\nred{}

\def\ov#1{{\overline{#1}}}

\def\wt#1{{\widetilde{#1}}}

\newcommand{\aff}{{}}

\newcommand{\e}{{\operatorname{e}}}
\newcommand{\m}{{\operatorname{m}}}

\newcommand{\Ann}{{\operatorname{Ann}}}

\newcommand{\Res}{{\operatorname{Res}}}
\newcommand{\irr}{{\operatorname{irr}}}

\newcommand{\ord}{{\operatorname{ord}}}
\newcommand{\Spec}{{\operatorname{Spec}}}
\newcommand{\Supp}{{\operatorname{supp}}}
\newcommand{\Elim}{{\operatorname{Elim}}}

\renewcommand{\sup}{{\operatorname{sup}}}
\renewcommand{\max}{{\operatorname{max}}}
\renewcommand{\div}{{\operatorname{div}}}
\newcommand{\Div}{{\operatorname{Div}}}

\renewcommand{\Im}{{\operatorname{Im}}}
\renewcommand{\char}{\mbox{char}}

\newcommand{\Conv}{{\operatorname{conv}}}

\newcommand{\Stab}{{\operatorname{stab}}}

\newcommand{\newton}{{\operatorname{N}}}

\newcommand{\Id}{{\operatorname{Id}}}
\newcommand{\init}{{\operatorname{init}}}

\newcommand{\h}{{\operatorname{h}}}
\newcommand{\ph}{{\operatorname{ph}}}
\newcommand{\wh}{{\operatorname{\widehat{h}}}}

\newcommand{\mult}{{\operatorname{mult}}}

\newcommand{\coeff}{{\operatorname{coeff}}}
\newcommand{\Bl}{{\operatorname{Bl}}}

\newcommand{\cz}{{]_{_\Z}}}
\newcommand{\ck}{{]_{_{k[\bft]}}}}
\newcommand{\pr}{{\operatorname{pr}}}
\newcommand{\ppr}{{\operatorname{ppr}}}

\newcommand{\bfdeg}{{\operatorname{\bf deg}}}
\newcommand{\hooklongrightarrow}{\lhook\joinrel\longrightarrow}

\newcommand{\A}{{\mathbb{A}}}
\newcommand{\C}{{\mathbb{C}}}
\newcommand{\D}{{\mathbb{D}}}
\newcommand{\F}{{\mathbb{F}}}

\newcommand{\N}{{\mathbb{N}}}
\renewcommand{\P}{{\mathbb{P}}}
\newcommand{\Q}{{\mathbb{Q}}}
\newcommand{\R}{{\mathbb{R}}}

\newcommand{\Z}{{\mathbb{Z}}}

\newcommand{\cI}{{\mathcal{I}}}
\newcommand{\cJ}{{\mathcal{J}}}

\newcommand{\cM}{{\mathcal{M}}}

\newcommand{\cV}{{\mathcal{V}}}
\newcommand{\cW}{{\mathcal{W}}}
\newcommand{\cX}{{\mathcal{X}}}
\newcommand{\cY}{{\mathcal{Y}}}

\newcommand{\bfa}{{\boldsymbol{a}}}
\newcommand{\bfb}{{\boldsymbol{b}}}
\newcommand{\bfc}{{\boldsymbol{c}}}
\newcommand{\bfd}{{\boldsymbol{d}}}
\newcommand{\bfe}{{\boldsymbol{e}}}

\newcommand{\bfh}{{\boldsymbol{h}}}

\newcommand{\bfl}{{\boldsymbol{l}}}
\newcommand{\bfell}{{\boldsymbol{\ell}}}

\newcommand{\bfn}{{\boldsymbol{n}}}
\newcommand{\bfp}{{\boldsymbol{p}}}
\newcommand{\bfq}{{\boldsymbol{q}}}
\newcommand{\bfs}{{\boldsymbol{s}}}

\newcommand{\bft}{{\boldsymbol{t}}}
\newcommand{\bfu}{{\boldsymbol{u}}}
\newcommand{\bfv}{{\boldsymbol{v}}}
\newcommand{\bfw}{{\boldsymbol{w}}}
\newcommand{\bfx}{{\boldsymbol{x}}}
\newcommand{\bfy}{{\boldsymbol{y}}}

\newcommand{\bfL}{{\boldsymbol{L}}}

\newcommand{\bfN}{{\boldsymbol{N}}}

\newcommand{\bfU}{{\boldsymbol{u}}}

\newcommand{\bfgamma}{{\boldsymbol{\gamma}}}
\newcommand{\bfdelta}{{\boldsymbol{\delta}}}

\newcommand{\bftheta}{{\boldsymbol{\theta}}}

\newcommand{\bftau}{{\boldsymbol{\tau}}}
\newcommand{\bfxi}{{\boldsymbol{\xi}}}
\newcommand{\bfpsi}{{\boldsymbol{\psi}}}
\newcommand{\bfomega}{{\boldsymbol{\omega}}}

\newcommand{\Qbarra}{\ov{\Q}}
\newcommand{\bfzero}{{\boldsymbol{0}}}
\newcommand{\bfone}{{\boldsymbol{1}}}
\numberwithin{equation}{section}
\theoremstyle{definition}
\newtheorem{definition}[equation]{Definition}

\newtheorem{remark}[equation]{Remark}
\newtheorem{example}[equation]{Example}

\theoremstyle{plain}
\newtheorem{lemma}[equation]{Lemma}
\newtheorem{proposition}[equation]{Proposition}
\newtheorem{definitionproposition}[equation]{Proposition-Definition}
\newtheorem{theorem}[equation]{Theorem}
\newtheorem{corollary}[equation]{Corollary}

\begin{document}
\title[Heights of varieties and arithmetic Nullstellens\"atze]{Heights
  of varieties in multiprojective spaces \\[1mm]
and arithmetic Nullstellens\"atze}

\author[Carlos D'Andrea]{Carlos D'Andrea}
\address{Departament d'{\`A}lgebra i Geometria, Universitat de Barcelona.
Gran Via 585, 08007 Barcelona, Spain}
\email{cdandrea@ub.edu}
\urladdr{\url{http://atlas.mat.ub.es/personals/dandrea/}}

\author[Teresa Krick]{Teresa Krick}
\address{Departamento de Matem\'atica, Facultad de Ciencias Exactas y
Naturales, Universidad de Buenos Aires and IMAS, CONICET. Ciudad
Universitaria, 1428 Buenos Aires, Argentina  }
\email{krick@dm.uba.ar} \urladdr{\url{http://mate.dm.uba.ar/~krick/}}

\author[Mart{\'\i}n~Sombra]{Mart{\'\i}n~Sombra}
\address{ICREA and
Departament d'{\`A}lgebra i Geometria, Universitat de Barcelona.
Gran Via~585, 08007 Barcelona, Spain}
\email{sombra@ub.edu}
\urladdr{\url{http://atlas.mat.ub.es/personals/sombra/}}

\date{\today}

\thanks{D'Andrea was partially supported
by the research project MTM2010-20279 (Spain). Krick was partially
supported by the research projects ANPCyT 3671/05, UBACyT X-113
2008-2010 and CONICET PIP 2010-2012   (Argentina). Sombra was
partially supported by the research project MTM2009-14163-C02-01
(Spain) and by a MinCyT Milstein fellowship (Argentina).}

\begin{abstract}
    We present bounds for the degree and the height of the polynomials
  arising in some problems in effective algebraic geometry
  including the implicitization of rational maps and the effective
  Nullstellensatz over a variety.  Our treatment is based on
  arithmetic intersection theory in products of projective spaces and
  extends to the arithmetic setting constructions and results due to
  Jelonek.  A~key role is played by the notion of {canonical mixed
    height} of a multiprojective. We study this notion from
  the point of view of resultant theory and establish some of its
  basic properties, including its behavior with respect to
  intersections, projections and products. We obtain analogous results
  for the function field case, including a parametric Nullstellensatz.
\end{abstract}
\maketitle

\setcounter{tocdepth}{3}
\typeout{Contents}
\tableofcontents

\vfill\pagebreak
\section*{Introduction}
\label{Introduction and statement of results}
\setcounter{section}{0}

In 1983, Serge Lang wrote in the preface to his book \cite{Lang83}:

\medskip
\begin{quotation}
{\it     It is legitimate, and to many people an interesting point of view,
  to ask that the theorems of algebraic geometry from the Hilbert
  Nullstellensatz to the more advanced results should carry with them
  estimates on the coefficients occurring in these theorems. Although
  some of the estimates are routine, serious and interesting
  problems arise in this context.
}
\end{quotation}

\medskip Indeed, the main purpose of the present text is to give
bounds for the degree and the size of the coefficients of the
polynomials in the Nullstellensatz.

\medskip
Let $f_{1},\dots,f_{s}\in \Z[x_{1},\dots,x_{n}]$ be polynomials
without common zeros in the affine space $\A^n(\Qbarra)$. The Nullstellensatz says
then that there exist $\alpha\in \Z\setminus \{0\}$ and
$g_{1},\dots,g_{s}\in \Z[x_{1},\dots,x_{n}]$ satisfying a B\'ezout
identity
\begin{displaymath}
  \alpha=g_{1}f_{1}+\dots+g_{s}f_{s}.
\end{displaymath}
As for many central results in commutative algebra and in algebraic
geometry, it is a non-effective statement. By the end of the 1980s,
the estimation of the degree and the height of polynomials satisfying
such an identity became a widely considered question in connection
with problems in computer algebra and Diophantine approximation.  The
results in this direction are generically known as \emph{arithmetic
  Nullstellens\"atze} and they play an important role in number theory
and in theoretical computer science. In particular, they apply to problems in
complexity and computability~\cite{Koiran96, Asc04, DKS10}, to
counting problems over finite fields or over the
rationals~\cite{BoBoKo09, Rem08}, and to effectivity in existence
results in arithmetic geometry~\cite{KreTch08, BiSt09}.

The first non-trivial result on this problem was obtained by
Philippon, who got a bound on the minimal size of the denominator
$\alpha$ in a B\'ezout identity as above \cite{Philippon90}. Berenstein and Yger achieved the next
big progress, producing  height estimates for the
polynomials $g_{i}$'s with techniques from
complex analysis (integral formulae for residues of currents)
\cite{BeYg91}. Later on, Krick, Pardo and Sombra \cite{KrPaSo01} exhibited
sharp bounds by combining arithmetic intersection
theory with the algebraic approach in \cite{KrPa96} based on duality
theory for Gorenstein algebras. Recall that the height of a polynomial $f\in
\Z[x_1,\dots,x_n]$, denoted by $\h(f)$, is defined as the logarithm of the maximum of the absolute
value of its coefficients. Then, Theorem~1
in~\cite{KrPaSo01} reads as follows:
if $d=\max_j \deg(f_j)$ and
$h=\max_j \h_\aff(f_j)$, there is a B\'ezout identity as above
satisfying
$$ \deg (g_i) \leq 4\,n \,d^{n},  \quad \h(\alpha), \h_\aff(g_i) \leq
4\, n\,(n+1) \, d^n \, \Big(h + \log s + (n+7) \, \log (n+1) \, d
\Big).
$$
We refer the reader to the surveys~\cite{Tei90,Brownawell01} for
further information on the history of the effective
Nullstellensatz, main results and open questions.

\medskip
One of the main results of this text is the arithmetic
Nullstellensatz over a variety below, which is a particular case of
Theorem~\ref{arithparam}.
For an affine equidimensional variety $V\subset
\A^n(\Qbarra) $, we denote by $\deg(V)$ and by $\wh(V)$ the degree and
the canonical height of the closure of $V$ with respect to the
standard inclusion $\A^n\hookrightarrow \P^n$.  The degree and the
height of a variety are measures of its geometric and arithmetic
complexity, see \S\ref{The height in the number field case} and the
references therein for details.  We say that a polynomial relation
holds on a variety if it holds for every point in it.

\begin{theorem}\label{mt} Let $V\subset \A^n(\overline{\Q})$ be a variety
  defined over $\Q$  of pure dimension
$r$ and $f_1,\dots,f_s\in \Z[x_1,\dots,x_n]\setminus \Z$   a family
of $s\le r+1$ polynomials without common zeros in~$V$. Set
$d_j=\deg(f_j)$ and $h_j=\h(f_j)$
 for $1\le j\le s$.
 Then there exist $\alpha\in \Z\setminus \{0\}$ and $g_1,\dots, g_s\in
 \Z[x_1,\dots,x_n] $ such that
$$\alpha=g_1f_1+\cdots+g_sf_s \quad \mbox{on } V
$$ with
\vspace{-3mm}
$$\aligned
\bullet &\  \deg\Big(g_if_i\Big) \leq \bigg(\prod_{j=1}^{s}d_j\bigg)\deg(V),\\[-2mm]
\bullet & \ \h(\alpha), \h(g_i)+\h(f_i) \leq
\bigg(\prod_{j=1}^{s}d_j \bigg)\bigg(\wh(V)+\deg(V)
\bigg(\sum_{\ell=1}^{s} \frac{h_\ell }{d_\ell}+ (4r+8)\log(n+3)
 \bigg)\bigg).
\endaligned $$
\end{theorem}

For $V=\A^{n}$,  this result gives the bounds
\begin{equation*}
\deg\Big(g_if_i\Big) \leq \prod_{j=1}^{s}d_j \quad, \quad
\h(\alpha), \h(g_i)+\h(f_i) \leq \sum_{\ell=1}^{s} \bigg(\prod_{j\ne
\ell}d_j \bigg) {h_\ell }+ (4n+8)\log(n+3)\prod_{j=1}^{s}d_j.
\end{equation*}
These bounds are substantially sharper than the previously
known. Moreover, they are close to optimal in many situations.  For
instance, let $d_1,\dots, d_{n+1}$, $H\ge 1$ and set
$$
f_1= x_1-H,\ f_2= x_2-x_1^{d_2},\ \dots, \ f_n=
x_{n}-x_{n-1}^{d_n},\ f_{n+1}=x_n^{d_{n+1}}.
$$
This is a system of polynomials without common zeros. Hence, the
above result implies that there is a B\'ezout identity
$\alpha=g_1f_1+\cdots+g_{n+1}f_{n+1}$ which satisfies $\h(\alpha)
\le d_2\cdots d_{n+1}(\log(H) + (4n+8)\log(n+3))$.  On the other
hand, specializing any such identity at the point
$(H,H^{d_2},\dots, H^{d_2\cdots d_{n}})$, we get
$$
\alpha=g_{n+1}(H,H^{d_2},\dots,  H^{d_2\cdots d_{n}})
 H^{ d_2\cdots d_{n+1}}.
$$
This implies the lower bound $\h( \alpha)\ge d_2\cdots d_{n+1}\log(H)$
and   shows
that the height bound in Theorem \ref{mt} is sharp in this case.
More examples can be found in \S\ref{NssArithZ}.

It is important to mention that all previous results in the
literature are limited to the case when $V$ is a complete intersection
and cannot properly distinguish the influence of each individual
$f_j$, due to the limitations of the methods applied. Hence,
Theorem~\ref{mt} is a big progress as it holds for an
arbitrary variety and gives bounds
depending on the degree and height of each $f_j$. This last point is
more important than it might seem at first. As it is well-known, by
using Rabinowicz' trick one can show that the weak Nullstellensatz
implies its strong version. However, this reduction yields good bounds
for the strong Nullstellensatz only if the corresponding weak version
can correctly differentiate the influence of each $f_j$, {see}
Remark~\ref{rabinowicz}. Using this observation, we obtain in
\S\ref{NssArithZ} the following arithmetic version of the strong
Nullstellensatz over a variety.

\begin{theorem} \label{mt_strong}
Let $V\subset \A^n(\overline{\Q})$ be a variety defined over  $\Q$
of pure dimension $r$  and $g,f_1,\dots, f_s\in
\Z[x_1,\dots,x_n]$ such that $g$ vanishes on the
common zeros of $f_{1},\dots, f_{s}$ in $V$. Set $d_j=\deg(f_j)$
and $h=\max_{j}\h_\aff (f_j)$ for $1\le j\le s$. Assume  that $d_1\ge
\cdots \ge d_{s}\ge 1$ and
set $D=\prod_{j=1}^{\min\{s,r+1\}}d_j$.
Set also $d_0=\max\{1,\deg(g)\}$ and
$h_0=\h_\aff(g)$. Then there exist $\mu\in \N$,  $\alpha\in
\Z\setminus \{0\}$ and $g_1,\dots,g_s\in \Z[x_1,\dots,x_n]$ such that
$$\alpha \,g^\mu=g_1f_1+\cdots+g_sf_s \quad \mbox{on } V$$
 with
\begin{align*} \bullet \ & \mu \le  2 D \deg(V),
\\[1mm]
\bullet \ & \deg(g_if_i)\le  4 d_{0}D\deg(V),\\[-3mm]
 \bullet \ & \h(\alpha),\h(g_i)+\h(f_i) \le
 2d_{0}D  \bigg(\wh(V) +\deg(V)\bigg(\frac{3h_0}{2d_0}+\sum_{\ell=1}^{\min\{s,r+1\}}
\frac{h}{d_\ell} + c(n, r,s)\bigg)\bigg),
\end{align*}
where $c(n,r,s)\le
 (6r+17)\log(n+4) + 3(r+1)\log(\max\{1,s-r\})$.
\end{theorem}

\medskip Our treatment of this problem is the arithmetic
counterpart of Jelonek's approach to produce bounds for the
degrees in the Nullstellensatz over a variety~\cite{Jelonek05}.  To
this end, we develop a number of tools in arithmetic intersection
and elimination theory in {products} of projective spaces. A key
role is played by the notion of {canonical mixed
heights} of multiprojective varieties, which we study from the point
of view of resultants. Our presentation of mixed resultants of
cycles in multiprojective spaces is mostly a reformulation of the
theory developed by R\'emond in~\cite{Remond01a,Remond01b} as an
extension of Philippon's theory of eliminants  of homogeneous ideals
\cite{Philippon86}. We also establish  new properties of them,
including their behavior under projections
(Proposition~\ref{formainicialproy})  and products
(Proposition~\ref{resultantesproductosciclos}).

Let $\bfn=(n_1,\dots,n_m)\in \N^{m}$ and set
$\P^\bfn=\P^{n_1}(\Qbarra)\times\ldots\times\P^{n_m}(\Qbarra)$ for
the corresponding multiprojective space. For a cycle $X$ of
$\P^{\bfn}$ of pure dimension $r$ and a multi-index
$\bfc=(c_{1},\dots, c_{m})\in \N^{m}$ of length $r+1$, the mixed
Fubini-Study height $\h_{\bfc}(X)$ is defined as an alternative
Mahler measure of the corresponding mixed resultant
(Definition~\ref{def:4}). The canonical mixed height is then defined
by a limit process as
\begin{displaymath}
  \wh_{\bfc}(X):=\lim_{\ell\to\infty} \ell^{-r-1} \h_{\bfc}([\ell]_{*}X),
\end{displaymath}
where $[\ell]$ denotes the $\ell$-power map of  $\P^{\bfn}$
(Proposition-Definition \ref{mixed height and sum}).

To  handle mixed degrees and heights, we introduce a notion of
{extended Chow ring of $\P^{\bfn}$}  (Definition~\ref{def:8}). It is
an arithmetic analogue of the Chow ring of $\P^{\bfn}$ and can be
identified with the quotient ring
\begin{math}
\R[\eta,\theta_1,\dots,\theta_m]/(\eta^2,\theta_1^{n_1+1},\dots,\theta_m^{n_m+1}).
\end{math}
We associate to the cycle~$X$  an element in this ring, denoted
$[X]_{_{\Z}}$, corresponding  under this identification to
\begin{displaymath} \sum_{\bfc}\wh_{\bfc}(X)\, \eta\,
\theta_{1}^{n_{1}-c_{1}} \cdots \theta_{m}^{n_{m}-c_{m}} +
\sum_{\bfb}\deg_{\bfb}(X) \, \theta_{1}^{n_{1}-b_{1}} \cdots
\theta_{m}^{n_{m}-b_{m}},
\end{displaymath}
the sums being indexed by all $\bfb, \bfc\in \N^{m}$ of
respective lengths $r$ and $r+1$ such that $\bfb,\bfc\le \bfn$.  Here,
$\deg_{\bfb}(X)$ denotes the mixed degree of $X$ of index $\bfb$.
This element contains the information of all non-trivial mixed degrees
and canonical mixed heights of $X$, since $\deg_{\bfb}(X)$ and
$\wh_{\bfc}(X)$ are zero for any other $\bfb$ and $\bfc$.

The extended Chow ring of $\P^{\bfn}$ turns out to be a quite useful
object which allows to translate geometric operations on
multiprojective cycles into  algebraic operations on rings and
classes. In particular, we obtain the following multiprojective
arithmetic B\'ezout's inequality, see also Theorem \ref{multihom}.
For a multihomogeneous polynomial $f\in
\Z[\bfx_{1},\dots,\bfx_{m}]$, where $\bfx_{i}$ is a group of $n_{i}+1$ variables, we denote by
 $||f||_\sup$ its sup-norm  (Definition
\ref{def:1}) and consider the element
$[f]_{\sup}$ in the extended Chow ring corresponding to the element $\sum_{i=1}^m
\deg_{\bfx_i}(f)\theta_i+ \log||f||_\sup \, \eta.$

\begin{theorem}\label{thm:1}
  Let $X$ be an effective equidimensional cycle of $\P^\bfn$ defined over $\Q$ and $f\in \Z[\bfx_{1},\dots,\bfx_{m}]$  a
multihomogeneous polynomial  such that $X$ and   $\div(f)$ intersect
properly. Then
$$
[ X\cdot\div(f)]_{_{\Z}} \le
[ X]_{_{\Z}}\cdot [ f]_{\sup}.
$$
\end{theorem}

Statements on classes in the extended Chow ring can  easily be
translated into statements on mixed degrees and heights. In this
direction, the above result implies that, for any $\bfb\in \N^{m}$
of length equal to $\dim(X)$,
\begin{displaymath}
\wh_\bfb(X\cdot \div(f)) \le
\sum_{i=1}^m\deg_{\bfx_i}(f)  \wh_{\bfb+\bfe_i}(X) +\log||f||_{\sup}\deg_\bfb(X)
 \end{displaymath}
 where $\bfe_{i}$ denotes the $i$-th vector of the standard basis of
 $\R^{m}$. In a similar way, we also study the behavior of arithmetic
 classes (and {\it a fortiori}, of canonical mixed heights) under
 projections (Proposition~\ref{113alturasrat}) and products
 (Proposition \ref{altrodgenrat}), among other results.

\medskip

Jelonek's approach consists in
producing a B\'ezout identity from an implicit equation of a
specific regular map. In general, the implicitization problem
consists in computing equations for an algebraic variety~$W$ from a
given rational parameterization of it. The typical case is when~$W$
is a hypersurface: the variety is {then defined by a single}
equation and the problem consists in computing this ``implicit
equation''. We consider here  the problem of estimating the height
of the implicit equation of a hypersurface parameterized by a
regular map $V\to W$ whose domain is an affine variety $V$, in terms
of the degree and {the height of $V$
  and of the polynomials defining the map.}  To this end, we prove the
following arithmetic version of Perron's theorem over a
variety~\cite[Thm.~3.3]{Jelonek05}.  It is obtained as a consequence
of Theorem~\ref{implicitmultiparZ}.

\begin{theorem} \label{PerronIntro}
Let $V\subset \A^n(\overline{\Q})$ be a variety defined over $\Q$ of
pure dimension $r$. Let $q_1,\dots, q_{r+1}\in
\Z[x_1,\dots,x_n]\setminus \Z$  such that the   closure of
the image of the map
$$
V\longrightarrow \A^{r+1} (\overline{\Q})\quad, \quad
\bfx\longmapsto\left(q_1(\bfx),\dots,q_{r+1}(\bfx)\right)
$$
is a hypersurface. Let $E=\sum_{\bfa\in\N^{r+1}} \alpha_\bfa
\bfy^\bfa \in \Z[y_1,\dots,y_{r+1}]$ be a primitive and squarefree
polynomial defining this hypersurface. Set $d_j=\deg(q_j)$, $h_j=\h(q_j)$
for $1\le j\le r+1$. Then, for all $\bfa=(a_1,\dots,a_{r+1})$ such
  that $\alpha_{\bfa}\ne 0$,
$$\aligned
\bullet & \
\sum_{i=1}^{r+1}a_id_i \le\bigg(\prod_{j=1}^{r+1}d_j\bigg) \deg(V),\\[-1mm]
\bullet & \ \h(\alpha_\bfa) + \sum_{i=1}^{r+1} a_ih_i  \, \le
\bigg(\prod_{j=1}^{r+1}d_j\bigg)\bigg(\wh(V)+\deg(V)\bigg(
\sum_{\ell=1}^{r+1}\frac{h_\ell}{d_\ell} +
(r+2)\log(n+3)\bigg)\bigg).\endaligned$$
\end{theorem}

For  $V=\A^n$
we have $r=n$, $\deg(V)=1$ and
$\wh(V)=0$.
Hence, the above result extends the classical Perron's theorem~\cite[Satz~57]{Perron27}, which
amounts to the weighted
degree bound for the
implicit equation $\sum_{i}a_id_i \le
\prod_{j}d_j $, by adding the bound for the height
$$
\h(\alpha_\bfa) + \sum_{i=1}^{n+1} a_ih_i  \le \sum_{\ell=1}^{n+1}
\bigg(\prod_{j\ne \ell}d_j\bigg) h_\ell +
(n+2)\log(n+3)\prod_{j=1}^{n+1}d_j.
$$

Our results on the implicitization problem as well as
those on mixed resultants and multiprojective arithmetic intersection
theory should be of independent interest, besides of their applications
to the arithmetic Nullstellensatz.

\medskip The method is not exclusive of $\Z$ but can also be
  carried over to other rings equipped with a suitable height
function. In this direction, we apply it to
$k[t_{1},\dots,t_{p}]$, the ring of polynomials over an arbitrary
field $k$ in $p$ variables: if we set $\bft=\{t_1,\dots, t_p\}$, the
{height} of a polynomial with coefficients in $ k[\bft]$ is its degree
in the variables $\bft$.  For this case, we also develop the
corresponding arithmetic intersection theory, including the behavior
of classes in the extended Chow ring with respect to intersections
(Theorem~\ref{bezaritff}), projections
(Proposition~\ref{113alturasff}), products
(Proposition~\ref{altrodgenff}) and ruled joins
(Proposition~\ref{prop:12}).  As a consequence, we obtain a parametric
analogue of Perron's theorem (Theorem~\ref{perronpar}) and then the
parametric Nullstellensatz below, which is a particular case of
Theorem~\ref{multiparametric}. For an affine equidimensional variety
$V\subset \A^{n}(\ov{k(\bft)})$, we denote by $\h(V)$ the
$\bft$-degree of the Chow form of its closure in
$\P^{n}(\ov{k(\bft)})$, see \S\ref{The height in the function field
  case} for details.
\begin{theorem}\label{mtp}
  Let $V\subset \A^n(\overline{k(\bft)})$ be a variety defined over
  $k(\bft)$ of pure dimension~$r$ and $f_1,\dots,f_s\in
  k[\bft][x_1,\dots,x_n]\setminus k[\bft]$ a family of $s\le r+1$
  polynomials without common zeros in $V$. Set $d_j=\deg_\bfx(f_j)$
  and $h_j=\deg_\bft(f_j)$~for $1\le j\le s$.  Then there exist
  $\alpha\in k[\bft]\setminus \{ 0\} $ and $g_1,\dots, g_s\in
  k[\bft][x_1,\dots,x_n] $ such that
$$\alpha=g_1f_1+\cdots+g_sf_s \quad \mbox{on } V$$
with \vspace{-3mm}
$$\aligned
\bullet &\  \deg_\bfx\left(g_if_i\right) \leq \bigg(\prod_{j=1}^{s}d_j\bigg)\deg(V),\\[-2mm]
\bullet & \ \deg(\alpha), \deg_\bft(g_if_i) \leq
\bigg(\prod_{j=1}^{s}d_j \bigg)\bigg( \h(V )+\deg(V) \sum_{\ell=1}^{s}
\frac{h_\ell }{d_\ell}\bigg).
\endaligned $$
\end{theorem}

For $V=\A^n(\ov{{k(\bft)}})$  we have $r=n$, $\deg(V)=1$ and
$\h(V)=0$. Hence, this result gives the following bounds for the
partial degrees of the polynomials in a B\'ezout identity:
$$
\deg_\bfx\left(g_if_i\right) \leq \prod_{j=1}^{s}d_j \quad, \quad
\deg(\alpha), \deg_\bft(g_if_i) \leq
\sum_{\ell=1}^{s}{}\bigg(\prod_{j\ne \ell}d_j \bigg)h_\ell.
$$

In {Theorem~\ref{strongNSSpar}}, we give a strong version of the
parametric Nullstellensatz over a variety, which also contains the
case of an arbitrary number of input polynomials. Up to our knowledge,
the only previous results on the parametric Nullstellensatz are due to
Smietanski~\cite{Smietanski93}, who considers the case when the number
of parameters $p$ is at most two and $V=\A^n(\ov{{k(\bft)}})$, see
Remark~\ref{rem:5}.

\medskip
To prove both the arithmetic and parametric versions of the effective
Nullstellensatz, we need to consider a more general version of these statements
where the input polynomials  depend on
\emph{groups} of parameters, see Theorem \ref{arithparam}.
The latter has {further} interesting applications. For instance,
consider the family  $F_{1},\dots, F_{n+1}$  of general $n$-variate
polynomials of degree $d_1, \dots, d_{n+1}$, respectively. For each
$j$, write
$$
F_j=\sum_{\bfa}u_{j,\bfa}\bfx^\bfa
$$
{where each $u_{j,\bfa}$ is a variable. Let
$\bfu_{j}=\{u_{j,\bfa}\}_{\bfa}$ be the group of {variables}
corresponding to the coefficients of $F_{j}$ and set $\bfU=\{\bfU_{1},\dots,
\bfU_{n+1}\}$.  The corresponding Macaulay resultant $R\in \Z[\bfu]$
lies in the ideal $(F_{1},\dots, F_{n+1})\subset \Q[\bfu,\bfx]$ and
Theorem \ref{arithparam} gives bounds for a representation of $R$ in
this ideal.  Indeed, we obtain that there are $\lambda\in \Z\setminus
\{0\}$ and $g_j\in \Z[\bfU,\bfx]$ such that $\lambda\, R= g_1
F_1+\cdots + g_{n+1}F_{n+1}$ with
\begin{displaymath}
\deg_{\bfU_j}(g_iF_{i})\le \prod_{\ell\ne
j}d_\ell\quad, \quad \h(\lambda R), \h(g_i)\le
(6n+10)\log(n+3)\bigg(\prod_{\ell=1}^{n+1 }d_\ell\bigg),
\end{displaymath}
{see examples~\ref{ejemploresultante} and \ref{resultantez}.  The
obtained bound for the height of the $g_{i}$'s is of the same order as
the sharpest known bounds for the height of $R$ \cite{Sombra04}.}

\medskip This text is divided in four sections. In the first one, we
recall the basic properties of mixed  resultants and degrees   of
cycles in multiprojective spaces over an arbitrary field~$K$.
The second section focuses on  the
 mixed heights of cycles for the case when~$K$ is a function field and on the
 canonical mixed heights of cycles  for $K=\Q$.
In the third section, we apply this machinery to the study of the
height of the implicit equation, including generalizations and variants
of~Theorem~\ref{PerronIntro}. We conclude in  the  fourth section by deriving
the different arithmetic Nullstellens\"atze.

\medskip {\bf Acknowledgments.} \nred{We thank Jos\'e Ignacio Burgos
  for the many discussions we had and, in particular, for the
  statement and the proof of Lemma \ref{lemm:5}.  We  thank Teresa
  Cortadellas, Santiago Laplagne and Juan Carlos Naranjo for
  helpful discussions} and  Matilde Lal\'\i n for
pointing us some references {on Mahler measures}. We also thank the referee for
his/her comments and, especially, for a simplification of our argument
in the proof of Lemma \ref{field_extension}. D'Andrea thanks the
University of Bordeaux 1 for inviting him in February 2009, Krick
thanks the University of Buenos Aires for her sabbatical during 2009
and the Universities of Bordeaux 1, Caen, Barcelona and Nice-Sophia
Antipolis for hosting her during that time, and Sombra thanks the
University of Buenos Aires for inviting him during October-December
2007 and in November 2010. The three authors also thank the Fields
Institute, where they met during the Fall 2009 FoCM thematic program.

\section{Degrees and resultants of multiprojective cycles}
\label{Some resultant  theory}

Throughout this text, we denote by $\N=\Z_{\ge0}$ and by $\Z_{>0}$ the
sets of non-negative and positive integers, respectively.  Bold
letters denote finite sets or sequences of objects, where the type and
number should be clear from the context: for instance, $\bfx$ might
denote $\{x_1,\dots,x_n\}$ so that if $A$ is a ring,
$A[\bfx]=A[x_1,\dots,x_n]$.  For a polynomial $f\in A[\bfx]$ we adopt
the usual notation
$$
f=\sum_\bfa \alpha_\bfa\bfx^\bfa$$
where, for each index $\bfa=(a_1,\dots,a_n)\in \N^n$,
$\alpha_{\bfa}$ denotes an element of $ A$ and $\bfx^\bfa$  the
monomial $x_1^{a_1}\cdots x_n^{a_n}$.
For $\bfa\in \N^{n}$, we denote by $|\bfa|=a_1+\cdots +a_n$ its
{length} and by
$\coeff_{\bfa}(f)=\alpha_{\bfa}$
the coefficient of $\bfx^{\bfa}$.
We also set $\bfa!=a_{1}!\cdots a_{n}!$.
The {\em support}
of  $f$ is the set of exponents corresponding to its
non-zero terms, that is, $\Supp(f)=\{\bfa: \,
\coeff_\bfa(f)\ne 0\}\subset \N^n$.
For $\bfa, \bfb\in \R^{n}$, we set $\langle \bfa, \bfb \rangle =\sum_{i=1}^{n}a_{i}b_{i}$. We say that $\bfa\le \bfb$
whenever the inequality holds coefficient wise.

For a factorial ring $A$, we denote by $A^\times$ its group of units.
A polynomial with coefficients in $A$ is \emph{primitive} if its
coefficients have no common factor in $A\setminus A^{\times}$.

\subsection{Preliminaries on multiprojective geometry}
\label{sec:cycl-mult-spac}

Let  $A$ be a factorial ring with field of fractions $K$ and $\ov K$
 the algebraic closure of $K$. For $m\in
\Z_{>0}$ and $\bfn=(n_1,\dots,n_m)\in \N^{m}$ we consider the
multiprojective space over $\ov K$
$$
\P^\bfn(\overline{K})=\P^{n_1}(\ov K)\times\cdots \times \P^{n_m}(\overline K).
$$
We also write $\P^\bfn=\P^\bfn(\overline{K})$ for short. For $1\le i\le m$, let $\bfx_i= \{ x_{i, 0}, \dots, x_{i,n_i}\} $
be a group of $n_{i}+1$  variables and set
\begin{displaymath}
  \bfx=\{\bfx_{1},\dots, \bfx_{m}\}.
\end{displaymath}
The multihomogeneous coordinate ring of $\P^{\bfn}$ is $\overline
K[\bfx]=\overline K[\bfx_1, \dots, \bfx_m]$. It is multigraded by
declaring $\bfdeg(x_{i,j}) = \bfe_i\in \N^{m}$, the $i$-th vector
of the standard basis of $\R^{m}$. For  $\bfd=(d_1,\dots,d_m)\in
\N^{m}$, we denote by $\overline K[\bfx]_{\bfd}$  its part of
multidegree  $\bfd$. Set
\begin{displaymath}
\N^{n_{i}+1}_{d_{i}}=\{\bfa_{i}\in \N^{n_{i}+1}:
|\bfa_{i}|=d_{i}\}\quad , \quad
\N^{\bfn+\bfone}_{\bfd}= \prod_{1\le i\le m} \N^{n_i+1}_{d_{i}}.
\end{displaymath}
A multihomogeneous polynomial $f\in \ov K[\bfx]_{\bfd}$ can then be written down  as
\begin{displaymath}
f =  \sum_{\bfa\in \N^{\bfn+\bfone}_{\bfd}} \alpha_{\bfa} \,
\bfx^\bfa.
\end{displaymath}

Let $K\subset E$ be an extension of fields and   $f\in
E[\bfx]_{\bfd}$.
For a point $\bfxi\in\P^{\bfn}$, the value
$f(\bfxi)$ is only defined up to a non-zero scalar in $\ov
K^{\times}$ which depends on a choice of multihomogeneous
coordinates for $\bfxi$.

\medskip An ideal $I\subset \ov K[\bfx]$ is \emph{multihomogeneous} if
it is generated by a family of multihomogeneous polynomials. For any
such ideal, we denote by $V(I)$ the subvariety of~$\P^{\bfn}$ defined
as its set of zeros. {Along this text,} a variety is neither necessarily
irreducible nor equidimensional.  Reciprocally, given a variety
$V\subset\P^{\bfn}$, we denote by $I(V)$ the multihomogeneous ideal of
$\ov K[\bfx]$ of polynomials vanishing on $V$. A variety~$V$ is {\em
  defined over $K$} if its defining ideal $I(V)$ is generated by
polynomials in $K[\bfx]$.

Let $\mathfrak{M}_{\bfn}=\{x_{1,j_1}\cdots x_{m,j_m}: 0\le j_i\le
n_i\}$ be the set of monomials of multidegree
$(1,\dots,1)\in\N^{m}$. A multihomogeneous ideal $I\subset \ov
K[\bfx]$ defines the empty variety of~$\P^{\bfn}$ if and only if
$\mathfrak{M}_{\bfn} \subset \sqrt{I}$, {see} for
instance~\cite[Lem.~2.9]{Remond01a}. The assignment $V\mapsto I(V)$ is
a one-to-one correspondence between non-empty subvarieties of
$\P^{\bfn}$ and radical multihomogeneous ideals of $\overline K[\bfx]$
not containing $\mathfrak{M}_{\bfn}$.

More generally, we denote by $\P_K^\bfn$ the multiprojective space
over $K$ corresponding to~$\bfn$.  The reduced subschemes of
$\P_K^\bfn$ will be alternatively called \emph{subvarieties of
  $\P_K^\bfn$} or {\em $K$-varieties}.  There is a one-to-one
correspondence $V\mapsto I(V)$ between non-empty subvarieties of
$\P^{\bfn}_{K}$ and radical multihomogeneous ideals of $K[\bfx]$ not
containing~$\mathfrak{M}_{\bfn}$.  For a multihomogeneous ideal
$I\subset K[\bfx]$ not containing $\mathfrak{M}_{\bfn}$, we denote by
$V(I)$ its corresponding $K$-variety.  A $K$-variety $V$ is
\emph{irreducible} if it is an integral subscheme of $\P^\bfn_K$ or,
equivalently, if the ideal $I(V)$ is prime. The dimension of $V $
coincides with the Krull dimension of the algebra $ K[\bfx_1,\dots,
\bfx_m]/I(V)$ minus $m$.

\begin{remark} \label{rem:4} In the algebraically closed case, the
  scheme $\P^\bfn_{\ov K}$ can be identified with the set of points
  $\P^\bfn(\ov K)$ and a {subvariety $V\subset \P^\bfn_{\ov K}$}
  can be identified with its set of points {$V(\ov K) \subset
    \P^\bfn(\ov K)$}. Under this identification, a subvariety of
  $\P^\bfn(\ov K)$ defined over $K$ corresponds to a
  $K$-variety. However, a $K$-variety does not necessarily correspond
  to a subvariety of $\P^\bfn(\ov K)$ defined over $K$, as the
  following example shows.  Let $t$ be a {variable} and set
  $K=\F_{p}(t)$, {where $p$ is a prime number and $\F_{p}$ is the
    field with~$p$ elements}. The ideal $(x_{1}^{p}-tx_{0}^{p})\subset
  K[x_{0},x_{1}]$ is prime and hence gives a subvariety
  of~$\P^{1}_{K}$. Its set of zeros in $\P^{1}(\ov K)$ consists in the
  point $\{(1:t^{1/p})\}$, which is not a variety defined over
  $K$. When the field $K$ is perfect (for instance, if $\char(K)=0$),
  the notion of $K$-variety does coincide, under this identification,
  with the notion of subvariety of $\P^\bfn(\ov K)$ defined over~$K$.
\end{remark}

A {\em $K$-cycle} of $\P^\bfn_{K}$ is
a finite $\Z$-linear combination
\begin{equation}
\label{eq:8}
 X=\sum_{V}m_V V
\end{equation}
of irreducible subvarieties of $\P^\bfn_{K}$.  The subvarieties $V$
such that $m_{V}\ne 0$ are the \emph{irreducible components} of
$X$. A $K$-cycle is \emph{of pure dimension} or \emph{equidimensional}
if its components are all of the same dimension. It is {\em effective}
(respectively, \emph{reduced}) if it can be written as in~\eqref{eq:8}
with $m_V\ge0$ (respectively, $m_V=1$). Given two $K$-cycles $X_{1}$
and $X_{2}$, we say that $X_{1}\ge X_{2}$ whenever $X_{1}-X_{2}$ is
effective. The \emph{support} of $X$, denoted $|X|$, is the
$K$-variety defined as the union of its components. Reciprocally, a
$K$-variety is a union of irreducible $K$-varieties of $\P^{\bfn}$ and
we identify it with the reduced $K$-cycle given as the sum of these
irreducible $K$-varieties.

For $0\le r\le |\bfn|$, we denote by $Z_r(\P^\bfn_K)$ the group of
$K$-cycles of $\P^{\bfn}$ of pure dimension~$r$ and by
$Z_r^+(\P^\bfn_K)$ the semigroup of those which are effective. For
shorthand, a $\ov K$-cycle is called a \emph{cycle} and we
denote the sets of $\ov K$-cycles and of effective $\ov K$-cycles of
pure dimension $r$ as $Z_r(\P^\bfn)$  and as $Z_r^+(\P^\bfn)$,
respectively.

\medskip
Let $I\subset K[\bfx]$ be a multihomogeneous ideal. For each
minimal prime ideal $P$ of~$I$, we denote by $m_{P}$ the
\emph{multiplicity of $P$ in $I$}, defined as the length of the
$K[\bfx]_{P}$-module $(K[\bfx]/I)_{P}$.
We  associate to $I$ the {$K$-cycle}
\begin{displaymath}
  X(I):=\sum_{P} m_{P} \, V(P) .
\end{displaymath}
If $V(I)$ is of pure dimension $r$, then $X(I)\in Z_{r}^{+}(\P^{\bfn}_K)$.
Let $K\subset E$ be an extension of fields and $V$ an
irreducible $K$-variety. We define the \emph{scalar extension of $V$
by $E$}  as the $E$-cycle $V_{E}=X(I(V)\otimes_K E)$. This notion
extends to $K$-cycles by linearity and induces an inclusion of groups
$Z_{r}(\P^{\bfn}_{K})\hookrightarrow Z_{r}(\P^{\bfn}_{E})$.

\medskip
Each Weil  or Cartier divisor of $\P^{\bfn}_{K}$ is
globally defined by a single rational multihomogeneous function in $K(\bfx)$ because
the ring $K[\bfx]$ is factorial~\cite[Prop.~II.6.2
and~II.6.11]{Hartshorne77}. Hence, we will not make distinctions
between {them}. We write $\Div(\P^{\bfn}_{K})=
Z_{|\bfn|-1}(\P^{\bfn}_{K})$ for the group of divisors of
$\P^{\bfn}_{K}$ and by $\Div^{+}(\P^{\bfn}_{K})=
Z_{|\bfn|-1}^{+}(\P^{\bfn}_{K})$ for the semigroup of those which
are effective.

Each effective divisor $D$ of $\P^{\bfn}_K$ is defined by a
multihomogeneous primitive polynomial in $A[\bfx]\setminus \{0\}$,
unique up to a unit of $A$. We denote this polynomial by $f_{D}$. If
we write $D=\sum_{H}m_{H}H$ where {$H$ is a}
$K$-hypersurface of $\P^{\bfn}$ and  $m_{H}\in \N$, then there
exists~$\lambda\in A^\times$ such that
\begin{displaymath}
  f_{D}=\lambda \,\prod_{H}f_H^{m_{H}}.
\end{displaymath}
Conversely, given a multihomogeneous
polynomial $f\in  A[\bfx]\setminus\{0\}$, we denote by $\div(f)\in
\Div^{+}(\P^{\bfn}_K)$ the associated  divisor.

\medskip
We introduce some basic operations on cycles and divisors.
\begin{definition}
  \label{def:11}
Let $V$ be an irreducible subvariety of $\P^{\bfn}_{K}$
and $H$ an irreducible hypersurface not containing $V$.
Let $Y$ be an irreducible component of $V\cap H$.
The \emph{intersection multiplicity of $V$ and $H$ along $Y$}, denoted
 $\mult(Y |V , H )$, is
the length of
the $K[\bfx]_{I(Y)}$-module
$\Big(K[\bfx]/(I(V)+I(H))\Big)_{I(Y)}$, see~\cite[\S
I.1.7]{Hartshorne77}.
The \emph{intersection product
of $V$ and $H$} is defined as
\begin{equation}\label{eq:71}
V\cdot H=\sum_{Y}\mult(Y|V,H) \, Y,
\end{equation}
{the sum being over the irreducible components of $V\cap H$}.
It is a cycle of pure dimension~$\dim(V)-1$.

Let $X$ be an equidimensional cycle and $D$ a divisor. We say that
$X$ and $D$ \emph{intersect properly} if no irreducible component of $X$ is
contained in $|D|$.
By bilinearity, the intersection product {in \eqref{eq:71}} extends to a pairing
$$
Z_{r}(\P^{\bfn}_{K})\times
\Div(\P^{n}_{K}) \dashrightarrow  Z_{r-1}(\P^{\bfn}_{K})\quad , \quad
(X,D)\longmapsto X\cdot D,
$$
well-defined  whenever $X$ and $D$ intersect properly.
\end{definition}

Let $X\in Z_{r}(\P^{\bfn}_{K})$ and  $D_{1},\dots, D_{\ell}\in
\Div(\P^{\bfn}_{K})$.
Then $X\cdot \prod_{j=1}^{\ell}D_{j}$ does not depend on the order
of the divisors, provided that
all the intermediate  products are
well-defined~\cite[Cor.~2.4.2 and Example~7.1.10(a)]{Fulton84}.

\begin{definition}
  \label{def:9}
Let  $m_1,m_{2}\in \Z_{>0} $ and $\bfn_{i}\in \N^{m_{i}}$ for $i=1,2$. Let $\varphi:\P^{\bfn_{1}}_{K}\dashrightarrow \P^{\bfn_{2}}_{K}$ be a
rational map  and $V$ an
irreducible subvariety of $\P^{\bfn_{1}}_{K}$. The \emph{degree of
$\varphi$ on $V$} is
\begin{displaymath}
  \deg(\varphi|_{V})=\left\{
    \begin{array}{cl}
      [K(V):K(\ov{\varphi(V)})]& \text{ if }
      \dim(\ov{\varphi(V)})= \dim(V), \\[2mm]
      0 & \text{ if }
      \dim(\ov{\varphi(V)})< \dim(V).
    \end{array}\right.
\end{displaymath}
 The \emph{direct image of $V$ under $\varphi$} is defined as
the cycle $\varphi_{*}V= \deg(\varphi|_{V}) \,\ov{\varphi(V)}$. It
is a cycle of the same dimension as $V$. This notion extends by
linearity to equidimensional cycles and induces a  $\Z$-linear map
\begin{displaymath}
  \varphi_{*}: Z_{r}(\P^{\bfn_{1}}_{K}) \longrightarrow  Z_{r}(\P^{\bfn_{2}}_{K}) .
\end{displaymath}
If $\psi :\P^{\bfn_{2}}_{K}\dashrightarrow \P^{\bfn_{3}}_{K}$ is a further
rational map, then $(\psi\circ \varphi)_{*}=\psi_{*}\circ
\varphi_{*}$ because of the multiplicativity of the degree of field
extensions.

Let $H$ be a hypersuface of  $\P^{\bfn_{2}}_{K}$ not containing the
image of $\varphi$. The \emph{inverse image of~$H$ under $\varphi$}
is defined as the hypersurface $\varphi^{*}H=
\ov{\varphi^{-1}(H)}$. This notion extends  to a $\Z$-linear map
\begin{displaymath}
\varphi^{*}:\Div(\P^{\bfn_{2}}_{K})\dashrightarrow \Div(\P^{\bfn_{1}}_{K}),
\end{displaymath}
well-defined for divisors whose support does not contain the image of $\varphi$.
\end{definition}
Direct images of cycles, inverse images of divisors and intersection
products are related by the  \emph{projection formula} \cite[Prop.~2.3(c)]{Fulton84}: let
$\varphi:\P^{\bfn_{1}}_{K}\to\P^{\bfn_{2}}_{K}$ be a proper map, $X$ a cycle
of $\P^{\bfn_{1}}_{K}$ and $D$ a divisor of $\P^{\bfn_{2}}_{K}$
containing no component of $\varphi(|X|)$. Then
\begin{equation}
  \label{eq:38}
  \varphi_{*}(X\cdot \varphi^{*}D)= \varphi_{*}X\cdot D.
\end{equation}

\subsection{Mixed degrees}\label{Degrees}
We recall the basic properties of mixed degrees of multiprojective
cycles. {We also study the behavior of this notion under linear projections.}

\begin{definition}
  \label{def:13}
  Let $V \subset \P^{\bfn}_K$ be an irreducible $K$-variety. The
  \emph{Hilbert-Samuel function of $V$} is the numerical function
  defined as
\begin{displaymath}
H_{V}:  \N^{m}\longrightarrow  \N \quad, \quad \bfdelta\longmapsto
\dim_{K}\Big(\big(K[\bfx]/I(V)\big)_{\bfdelta}\Big).
\end{displaymath}
\end{definition}

\begin{proposition}
  \label{prop:2}
Let $V\subset \P^{\bfn}_K$ be an irreducible $K$-variety of
dimension $r$.
\begin{enumerate}
\item \label{item:8}
There is a unique polynomial \begin{math} P_{V}\in
\Q[z_{1},\dots,z_{m}]
\end{math}
such that $P_{V}(\bfdelta)=H_{V}(\bfdelta) $ for all $\bfdelta \ge
\bfdelta_0 $ for \nred{some} $\bfdelta_{0}\in \N^{m}$. In addition
$\deg(P_V)=r$.  \smallskip
\item \label{item:38}
Let $\bfb=(b_1,\dots,b_m)\in \N^{m}_{r}$.  Then
 $ \bfb! \, \coeff_{\bfb}(P_{V}) \in \N$. Moreover,
if $b_{i}>n_{i}$ for some $i$,  then $\coeff_{\bfb}(P_{V}) =0$.
\end{enumerate}
\end{proposition}

\begin{proof}
\eqref{item:8} and the second part of \eqref{item:38} follow from~\cite[Thm.~2.10(1)]{Remond01a}.
The first part of \eqref{item:38} follows from \cite[Thm.~2.10(2)]{Remond01a} and its proof.
\end{proof}

The polynomial  $P_V$ in Proposition~\ref{prop:2} is called the \emph{Hilbert-Samuel polynomial
of $V$}.

\begin{definition}
\label{chowgeom} Let $V\subset \P^{\bfn}_K$ be an irreducible
$K$-variety of dimension $r$ and
 $\bfb\in \N^{m}_{r}$.
The \emph{(mixed) degree of $V$ of index $\bfb$} is defined as
\begin{displaymath}
  \deg_{\bfb}(V)= \bfb! \, \coeff_{\bfb}(P_{V}).
\end{displaymath}
It is a non-negative integer, thanks to Proposition \ref{prop:2}\eqref{item:38}.
This  notion extends by linearity to equidimensional $K$-cycles
and induces a map
\begin{math}
\deg_{\bfb}:Z_{r}(\P^{\bfn}_{K})\to \Z.
\end{math}
\end{definition}

Recall that the {Chow ring  of $\P^{\bfn}_K$} is the graded ring
$$
A^{*}(\P^{\bfn}_K)=\Z[\theta_1,\dots,\theta_m]/(\theta_1^{n_1+1},\dots,\theta_m^{n_m+1}),
$$
where each $\theta_{i}$ denotes the class of the inverse image of a
hyperplane of $\P^{n_{i}}_K$ under the projection $\P^{\bfn}_K\to
\P^{n_{i}}_K$ \cite[Example~8.3.7]{Fulton84}.
Given a cycle $X\in  Z_r(\P^\bfn_K)$, its class in the Chow
ring is
\begin{displaymath}
[X]= \sum_{\bfb} \deg_{\bfb}(X)\, \theta_1^{n_1 -b_1}\cdots
 \theta_m^{n_m -b_m} \in A^{*}(\P^{\bfn}_K),
\end{displaymath}
{the sum being over all $\bfb\in \N_r^{m}$ such that $\bfb\le
\bfn$. }It is a homogeneous element of degree $|\bfn|-r$. By
Proposition~\ref{prop:2}\eqref{item:38},  $\deg_{\bfb}(X)=0$
whenever  $b_i>n_i$ for some $i$. Hence,  $[X]$ contains the
information of all the mixed degrees of $X$, since
$\{\bftheta^{\bfb}\}_{\bfb\le \bfn}$ is a $\Z$-basis of
$\A^*(\P^\bfn_K)$. For $X_1,X_2\in Z_r(\P^\bfn_K)$, we say that
$[X_1]\ge [X_2]$ whenever the inequality holds coefficient wise in
terms of this basis.

Given a $K$-cycle $X$, its class in the Chow ring is invariant under field extensions. In
particular,  $[X_{\ov K}]=[X]$ and
$\deg_\bfb(X)=\deg_\bfb(X_{\ov K})$ for all $\bfb\in \N^{m}_{r}$.
If $\dim(X)=0$, its degree is defined as
the number of points in $X_{\ov K}$, counted with multiplicity.

\begin{proposition} \
\label{prop:5}
  \begin{enumerate}
\item \label{item:48} Let $X\in Z_{r}^{+}(\P^{\bfn}_K)$. Then $[X] \ge 0$.
\smallskip
\item\label{item:11} We have $[\P^{\bfn}_K]=1$.
Equivalently, $\deg_{\bfn}(\P^{\bfn}_K)=1$  and
$\deg_{\bfb}(\P^{\bfn}_K)=0$ for all $\bfb\in \N_{|\bfn|}^{m}$ such
that $\bfb\neq\bfn$.
\smallskip
\item \label{item:12} Let $X\in Z_0(\P^\bfn_K)$. Then $[X]= \deg(X)
  \bftheta^{\bfn}$. Equivalently, $\deg_{\bf0}(X)=\deg(X)$.
  \smallskip
\item \label{gradosparciales} Let $D\in \Div^{+}(\P^{\bfn}_K) $ and
 $f_{D}$ its defining polynomial. Then
$$
[D]=\sum_{i=1}^{m}\deg_{\bfx_i}(f_{D}) \theta_{i}.
$$
Equivalently, $\deg_{\bfn-\bfe_i}(D)= \deg_{\bfx_i}(f_{D}) $ for
$1\le i\le m$ and $\deg_\bfb(D)=0$ for all $\bfb\in \N_{|\bfn|-1}^{m}$
such that $\bfb\ne \bfn-\bfe_i$ for all $i$.
\smallskip
\item \label{item:13} Let $n\in \N$ and $V\subset \P^n_K$ a
  $K$-variety  of pure  dimension $r$. Then
  \begin{displaymath}
[V]= \deg(V) \theta^{n-r},
  \end{displaymath}
where $\deg(V)$ denotes the  degree of the projective variety
$V$. Equivalently, $\deg_r(V)=\deg(V)$.
  \end{enumerate}
\end{proposition}

\begin{proof}\

\eqref{item:48} This follows from the definition of $[X]$ and
Proposition \ref{prop:2}\eqref{item:38}.

\smallskip
\eqref{item:11} For
$\bfdelta=(\delta_1,\dots,\delta_m) \in \N^{m}$,
\begin{displaymath}
  H_{\P^{\bfn}_K}(\bfdelta)= \dim_{K}\Big(  K[\bfx]_{\bfdelta}\Big) =
\prod_{i=1}^{m}{n_{i}+\delta_{i}\choose n_{i}} =
\frac{1}{n_{1}!\cdots n_{m}!} \bfdelta^{\bfn} + O(||\bfdelta||^{|\bfn|-1}),
\end{displaymath}
where $||\cdot||$ denotes any fixed norm on $\R^{m}$. This implies
that $\deg_{\bfn}(\P^{\bfn}_K)=1$ and thus
$[\P^\bfn_K]=\deg_\bfn(\P^\bfn_K)=1$, as stated.

\smallskip
\eqref{item:12}
Let $\bfxi=(\bfxi_{1},\dots, \bfxi_{m})\in \P^{\bfn}$.
We have
$ \ov K[\bfx]/I(\bfxi)= \bigotimes_{i=1}^{m}\ov
K[\bfx_{i}]/I(\bfxi_{i})$. Hence, for $\bfdelta\in \N^{m}$,
$$
  H_{\bfxi}(\bfdelta)= \dim_{\ov K}\Big( \Big( \ov
  K[\bfx]/I(\bfxi)\Big)_{\bfdelta}\Big) = \prod_{i=1}^{m}
  \dim_{\ov K}\Big( \ov
  K[\bfx_{i}]/I(\bfxi_{i})\Big)_{\delta_{i}}=1.
$$
This implies that  $\deg_{\bf0}(\bfxi)=1$ and so
$[\bfxi]=\bftheta^\bfn$.
For a general zero-dimensional $K$-cycle $X$, write
$X_{\ov K}=\sum_\bfxi m_\bfxi \bfxi$ for some points $\bfxi\in
\P^{\bfn}$ and $m_{\bfxi}\in \Z$.
Hence, $\deg_{\bf0}(X)=\sum_\bfxi m_\bfxi \deg_{\bf0}(\bfxi)=
\sum_{\bfxi}m_{\bfxi}=\deg(X)$ and so $[X]= \deg(X)\bftheta^\bfn$.

\smallskip
\eqref{gradosparciales}
Write $\bfdeg(f)=(\deg_{\bfx_{1}}(f),  \dots,  \deg_{\bfx_{m}}(f))$. For  $\bfdelta \ge \bfdeg(f)$,
there is an exact sequence
$$
  0 \longrightarrow  K[\bfx]_{\bfdelta -\bfdeg(f)} \mathop{\longrightarrow}^{\times f}
  K[\bfx]_{\bfdelta} \longrightarrow  \big( K[\bfx]/(f)\big)_{\bfdelta} \longrightarrow  0.
$$
 Hence, $  H_{D}(\bfdelta)=H_{\P^{\bfn}_K}(\bfdelta)- H_{\P^{\bfn}_K}(\bfdelta-\bfdeg(f_{D}))$
and therefore
\begin{displaymath}
P_{D}(\bfd)=\sum_{i=1}^{m}
\frac{\deg_{\bfx_{i}}(f)}{(\bfn-\bfe_{i})!}\bfdelta^{\bfn-\bfe_{i}} +
O(||\bfdelta||^{|\bfn|-2}).
\end{displaymath}
This implies that $\deg_{\bfn-\bfe_i}(D)= \deg_{\bfx_i}(f) $ and so
$[D]=\sum_{i=1}^{m}\deg_{\bfx_i}(f_{D}) \theta_{i}$, as stated.

\smallskip
\eqref{item:13} This follows readily from the definition of $\deg(V)$
in terms of Hilbert functions.
\end{proof}

The following is the multiprojective version of B\'ezout's theorem.

\begin{theorem}\label{bezgeom}
  Let $X\in Z_{r}(\P^{\bfn}_K)$ and $f\in  K[\bfx_{1},\dots,\bfx_{m}]$ be a
multihomogeneous polynomial such
that $X$ and  $\div(f)$ intersect properly. Then
$$
[X\cdot\div(f)]=[X]\cdot [\div(f)].
$$
Equivalently,
$\deg_\bfb\Big(X\cdot \div(f)\Big) = \sum_{i=1}^m \deg_{\bfx_i}(f)
\deg_{\bfb+\bfe_i}(X)$
for all $\bfb\in \N_{r-1}^{m}$.
\end{theorem}

\begin{proof}
  The equivalence between the two statements follows from
  Proposition~\ref{prop:5}\eqref{gradosparciales}.  The second
  statement follows for instance from~\cite[Thm.~3.4]{Remond01b}.
\end{proof}

{Next} corollary follows readily from this result together with
Proposition~\ref{prop:5}\eqref{item:12}.

\begin{corollary}\label{cor:4}
  Let $f_1,\dots,f_{|\bfn|}\in  K[\bfx_{1},\dots,\bfx_{m}]$ be multihomogeneous
  polynomials such that $\dim\Big(V(f_{1},\dots, f_{i})\Big)= |\bfn|-i$ for all
  $i$.
Then
\begin{displaymath}
  \deg\left(\prod_{i=1}^{|\bfn|}\div(f_i)\right)= \coeff_{\bftheta^\bfn}\bigg( \prod_{i=1}^{|\bfn|}\Big(
  \deg_{\bfx_1}(f_i)\, \theta_1+\cdots +\deg_{\bfx_m}(f_i)\, \theta_m\Big)\bigg).
  \end{displaymath}
\end{corollary}

\begin{example} \label{exm:5}
How many pairs (eigenvalue, eigenvector) can a generic square
matrix have? Given $M=(m_{i,j})_{i,j}\in  K^{n\times n}$, the
problem of computing these pairs  consists in solving   $M \bfv
=\lambda\,\bfv$ for  $\lambda \in \ov K$ and $\bfv=(v_{1},\dots,
v_{n})\in \ov K^{n}\setminus \{0\}$. Set
\begin{equation*}
f_{i}= s_{1}v_{i}-  s_{0}\sum_{j=1}^{n}m_{i,j}v_{j}  \quad
 , \quad  1\le i \le n.
\end{equation*}
The matrix
equation $M  \bfv=\lambda\,\bfv$ translates into the system of $n$
bilinear scalar equations $f_i =0$, $1\le i\le n$, for
$((s_{0}:s_{1}),\bfv)\in \P^{1}\times \P^{n-1}$ such
that $s_{0}\ne0$.
If $M$ is generic, the hypersurfaces $V(f_{i})$
intersect properly.
By Corollary~\ref{cor:4}, the number of solutions in
$ \P^{1}\times \P^{n-1}$ of this system of equations is
$$
\coeff_{\theta_{1}\theta_{2}^{n-1}}\bigg( \prod_{i=1}^{n}
\deg_{\bfs}( f_{i}) \theta_{1}+\deg_{\bfv}( f_{i})
 \theta_{2}\bigg ) =
\coeff_{\theta_{1}\theta_{2}^{n-1}}\Big( ( \theta_{1}+\theta_{2})^{n}
\Big ) = n.
$$
We deduce that $M$ admits at most $n$ pairs (eigenvalue, eigenvector)
counted with multiplicities.  A straightforward application of the
usual B\'ezout's theorem would have given the much larger bound
$2^{n}$.
\end{example}

The following result shows that mixed degrees can also be defined geometrically.
\begin{corollary}
  \label{cor:5}
Let $X\in Z_r(\P^\bfn_K)$ and  $\bfb\in \N_{r}^{m}$. For $1\le i\le
m$ and $0\le j\le b_{i}$ we denote by $H_{i,j}\subset \P^{\bfn}_K$ the
inverse image with respect to the projection $ \P^\bfn_K\to
\P^{n_i}_K$
 of a generic hyperplane of $\P^{n_i}_K$. Then
$$
\deg_{\bfb}(X) = \deg\bigg(X\cdot \prod_{i=1}^{m}\prod_{j=1}^{b_{i}} H_{i,j}\bigg).
$$
\end{corollary}

\begin{proof}
The variety $X$ and the divisors $ H_{i,j}$ intersect properly and
$[H_{i,j}]=\theta_i$. Theorem~\ref{bezgeom} implies that, for
$Z=X\cdot \prod_{i=1}^{m}\prod_{j=1}^{b_{i}}
  H_{i,j}\in Z_{0}(\P^{\bfn}_K)$,
\begin{displaymath}
\deg(Z)\bftheta^{\bfn}= [Z]= [X]\cdot
\prod_{i=1}^{m}\prod_{j=1}^{b_{i}} [H_{i,j}] =
[X] \bftheta^{\bfb}= \bigg(\sum_{\bfc\in \N_r^m}\deg_\bfc(X)\bftheta^{\bfn-\bfc}\bigg)\bftheta^\bfb
= \deg_{\bfb}(X) \bftheta^{\bfn},
\end{displaymath}
which proves the statement.
\end{proof}

Next we show that mixed degrees are monotonic with
respect to linear projections.
For $1\le i\le m$, let  $0\le l_i\le n_i $ and set $\bfl=(l_{1},\dots, l_{m})\in \N^m$.
Consider the linear projection which forgets the last $n_i-l_i$ coordinates in each factor of
$\P^\bfn_K$:
\begin{equation}\label{eq:58}
\pi: \P^\bfn_K\dashrightarrow\P^\bfl_K \quad, \quad (x_{i,j})_{1\le
  i\le m \atop 0\le j\le n_{i}} \longmapsto (x_{i,j})_{1\le
  i\le m \atop 0\le j\le l_{i}}
\end{equation}
This is a rational map, well-defined outside the union
of  linear subspaces $L:=\bigcup_{i=1}^m
V(x_{i,0},\ldots,x_{i,l_i})\subset\P^\bfn_K$.
It induces an injective
$\Z$-linear map
\begin{displaymath}
 \jmath: A^*(\P^\bfl_K)\hooklongrightarrow A^*(\P^\bfn_K) \quad, \quad
 P\longmapsto \bftheta^{\bfn-\bfl}P.
 \end{displaymath}

 \begin{proposition}\label{113}
 Let $\pi: \P^\bfn_K\dashrightarrow\P^\bfl_K$ be the linear projection
as  above and  $X\in Z_r^+(\P^\bfn_K)$. Then
     $$
     \jmath \Big([\pi_{*}X]\Big)\le [X].
     $$
Equivalently, $ \deg_\bfb\Big(\pi_*X\Big)\le \deg_\bfb(X) $ for all $\bfb\in \N^m_r$.
    \end{proposition}

  The proof of this result relies on the {technical}  Lemma \ref{lemm:5}
  below, which was suggested to us by Jos\'e Ignacio Burgos.
Consider the blow up of
$\P^{\bfn}$ along the subvariety $L$, denoted $\Bl_{L}(\P^{\bfn})$ and
defined as the   closure in $\P^{\bfn}\times\P^{\bfl}$  of the graph of
$\pi$. It is an irreducible variety of dimension $|\bfn|$. Set
$\bfx$ and $\bfy$ for the  multihomogeneous coordinates of
$\P^{\bfn}$ and $\P^{\bfl}$, respectively. The ideal of this variety is
\begin{equation}\label{eq:1}
  I(\Bl_{L}(\P^{\bfn}))=\Big( \{x_{i,j_{1}}y_{i,j_{2}}- x_{i,j_{2}}y_{i,j_{1}} : {1 \le
    i\le m, 0 \le j_{1}<j_{2}\le l_{i}}\}\Big) \subset \ov K[\bfx,\bfy].
\end{equation}
Consider the projections
\begin{displaymath}
  \pr_{1}:\P^{\bfn}\times\P^{\bfl}\longrightarrow \P^{\bfn} \quad, \quad
  \pr_{2}:\P^{\bfn}\times\P^{\bfl}\longrightarrow \P^{\bfl}.
\end{displaymath}
The exceptional divisor of the blow up is supported in the hypersurface $E=\pr_{1}^{-1}(L)$.
Let $V\subset \P^{\bfn}$ be an irreducible variety such that
$V\not\subset L$ and $W$ its strict
transform, which is the closure of the set $ {\pr_{1}^{-1}(V\setminus
  L)\cap \Bl_{L}(\P^{\bfn})}$. Then
\begin{displaymath}
  \pr_{1*}  W = V\quad, \quad
  \pr_{2*} W = \pi_{*}V.
\end{displaymath}
For a multihomogeneous polynomial $f\in \ov K[\bfy]\setminus \{0\}$,
we write $\div_{\P^{\bfl}}(f)$ for the divisor of
 $\P^{\bfl}$  defined by $f(\bfy)$ and $\div_{\P^{\bfn}}(f)$ for the divisor of $\P^{\bfn}$ defined by $f(\bfx)$.

\begin{lemma}
  \label{lemm:5}
Let $V\subset \P^{\bfn}$ be an irreducible variety of dimension $r$
such that $V\not\subset L$ and $f\in \ov K[\bfy]\setminus\{0\}$
a multihomogeneous polynomial. Assume that  $\pr_{2}^{*}
\,\div_{\P^{\bfl}}(f)$ intersects  $W$     properly and that no component of
$W\cdot \pr_{2}^{*}\, \div_{\P^{\bfl}}(f)$  is contained in $E$.
Then  $\div_{\P^{\bfl}}(f)$ (respectively,  $\div_{\P^{\bfn}}(f)$)
intersects $\pi_{*}V$ (respectively, $V$) properly and
\begin{displaymath}
\pi_{*}(V\cdot\div_{\P^{\bfn}}(f))=\pi_{*}V\cdot \div_{\P^{\bfl}}(f).
\end{displaymath}
\end{lemma}

\begin{proof}
Consider the following divisors of $\P^{\bfn}\times \P^{\bfl}$:
\begin{displaymath}
D_{1}=\pr_{1}^{*} \,\div_{\P^{\bfn}}(f) \quad,\quad
D_{2}=\pr_{2}^{*} \,\div_{\P^{\bfl}}(f)  .
\end{displaymath}
Write for short $B=\Bl_{L}(\P^{\bfn})$. Since $W$ is irreducible, the hypothesis that  $D_{2}$ intersects  $W$ properly is equivalent to
the fact that $W$ is not contained in $|D_{2}|$. This implies
that neither $V$ is contained in $V(f(\bfx))$
nor $\pi(V)$ is contained in $V(f(\bfy))$. Hence all intersection
products are well-defined.
We claim that
  \begin{equation*}
B\cdot (D_{2}-D_{1})
  \end{equation*}
is a cycle of pure dimension $|\bfn|-1$ with support contained in the
hypersurface $E$.
To prove this, write
$d_{i}=\deg_{\bfx_{i}}(f)$ and, for each $1\le i\le m$, choose an index $0\le
j_{i}\le l_{i}$. Using~\eqref{eq:1} we verify that
\begin{displaymath}
 \bigg( \prod_{i=1}^{m}y_{i,j_{i}}^{d_{i}}\bigg)f(\bfx) \equiv  \bigg(
 \prod_{i=1}^{m}x_{i,j_{i}}^{d_{i}}\bigg)f(\bfy)
\pmod{I(B)}.
\end{displaymath}
Observe that the ideal of $E$ in $\ov K[\bfx,\bfy]/I(B)$ is
generated by the set of monomials
\begin{displaymath}
 \bigg( \prod_{i=1}^{m}x_{i,j_{i}}\bigg)_{1\le i\le m \atop 0\le
j_{i}\le l_{i}}
\end{displaymath}
Let $(\bfxi,\bfxi')\in B\setminus E$. We have that
$\prod_{i=1}^{m}\xi_{i,j_{i}}\ne 0$ for a choice of $j_{i}$'s.
From here, we can verify that $\prod_{i=1}^{m}\xi'_{i,j_{i}}\ne 0$.
This implies that $f(\bfx)$ and $f(\bfy)$ generate the same ideal in
the localization $(\ov K[\bfx,\bfy]/I(B))_{(\bfxi,\bfxi')}$ for all
such $(\bfxi,\bfxi')$ and proves the claim.
Therefore, there exists a cycle  $Z\in Z_{r-1}(\P^{\bfn}\times \P^{\bfl})$  supported on $ E$ such that
$$
  W\cdot D_{1} = W\cdot D_{2}-Z.
$$
Since the map $\pr_{1}$ is proper, the  projection formula \eqref{eq:38}
implies that
\begin{displaymath}
V\cdot \div_{\P^{\bfn}}(f)=
\pr_{1*}(  W\cdot D_{1})
= \pr_{1*}(  W\cdot D_{2}) -
\pr_{1*}   Z .
\end{displaymath}
By hypothesis, no component of
$W\cdot D_{2}$
is contained in $E$. Since  $\pr_1: B\setminus E \to \P^\bfn \setminus L$ is an isomorphism,
no component of $\pr_{1*}(W\cdot D_{2})$ is
contained in $L$.
Hence,
\begin{displaymath}
\pi_{*}(V\cdot \div_{\P^{\bfn}}(f)))
= \pi_{*}\pr_{1*}(  W\cdot D_{2}) -
\pi_{*}\pr_{1*}   Z =   (\pi \circ \pr_1)_*(W\cdot D_2)=
\pr_{2*}(  W\cdot D_{2}),
\end{displaymath}
because $\pi_{*}\pr_{1*}   Z = 0$ as this is a cycle supported on
$L$. Again by the projection formula,
\begin{displaymath}
\pr_{2*}(  W\cdot D_{2})=  \pr_{2*} W \cdot     \div_{\P^{\bfl}}(f) =
\pi_{*} V \cdot \div_{\P^{\bfl}}(f),
\end{displaymath}
which proves the statement.
\end{proof}

\begin{proof}[Proof of Proposition~\ref{113}.]
The equivalence between the two formulations is a direct consequence
of the definitions. We reduce without loss of generality to the case of an irreducible
variety $V\subset \P^{\bfn}$ such that $\dim(\ov{\pi(V)})=r$.

We proceed  by induction on the dimension.
For $r=0$, the statement is
obvious and so we assume  $r\ge1$.
Let $1\le i\le m$ such that
$b_{i}\ne0$ and $\ell\in \ov K[\bfy_{i}]$ a linear form.
For each component $C$ of $W\cap E$ we
pick a point $\bfxi_{C}\in C$ and we impose that
$\ell(\bfxi_{C})\ne0$,
which holds for a generic choice of $\ell$.
This implies that $\pr_{2}^{*}(\div_{\P^{\bfl}}(\ell))$
intersects  $W$  properly and that their intersection has no
component contained in $E$.
Hence
\begin{align*}
\deg_{\bfb}(\pi_{*} V )&= \deg_{\bfb-\bfe_{i}}(\pi_{*} V \cdot
\div_{\P^{\bfl}}(\ell))  & \mbox{by B\'ezout's theorem \ref{bezgeom}, }\\[1mm]
 &= \deg_{\bfb-\bfe_{i}}(\pi_{*}(V\cdot
\div_{\P^{\bfn}}(\ell))) & \mbox{by Lemma \ref{lemm:5},}
\\[1mm]
& \le \deg_{\bfb-\bfe_{i}}(V\cdot \div_{\P^{\bfn}}(\ell)))& \mbox{by the inductive hypothesis,}\\[1mm]
& = \deg_{\bfb}(V) & \mbox{by B\'ezout's theorem
\ref{bezgeom}},
\end{align*}
which completes the proof.
\end{proof}

Next result gives the behavior of Chow rings and classes with respect
to products.

\begin{proposition}
\label{prop:3}
Let $m_i\in \Z_{>0} $ and $\bfn_{i}\in \N^{m_{i}}$ for $i=1,2$. Then
\begin{enumerate}
\item \label{item:10} $A^{*}(\P^{\bfn_{1}}_K\times \P^{\bfn_{2}}_K) \simeq A^{*}(\P^{\bfn_{1}}_K)
 \otimes_{\Z} A^{*}(\P^{\bfn_{2}}_K)$.
\smallskip
\item \label{item:14} Let $X_{i}\in Z_{r_{i}}(\P^{\bfn_{i}}_K) $ for
  $i=1,2$. The above {isomorphism identifies $[X_{1}\times X_{2}]$
    with $[X_{1}]\otimes_{\Z} [X_{2}]$.}  Equivalently, for all
  $\bfb_{i}\in \N^{m_{i}}$ such that
  $|\bfb_{1}|+|\bfb_{2}|=r_{1}+r_{2}$,
  \begin{equation*}
\deg_{\bfb_1,\bfb_2}(X_{1}\times X_{2})= \left\{
\begin{array}{cl}
  \deg_{\bfb_1}(X_{1})\deg_{\bfb_2}(X_{2}) & \text{ if }
  |\bfb_{1}|=r_{1}, |\bfb_{2}|= r_{2}, \\[2mm]
0 & \text{ otherwise.}
\end{array}
\right.
  \end{equation*}
\end{enumerate}
\end{proposition}

\begin{proof}\

\eqref{item:10} This is immediate from the definition of the Chow
ring.

\smallskip
\eqref{item:14} We reduce without loss of generality to the case of
irreducible $K$-varieties $V_i\subset \P^{\bfn_i}_K$, $i=1,2$.
Let $\bfx_{i}$ denote the
multihomogeneous coordinates of $\P^{\bfn_{i}}_K$.
For $\bfdelta_{i}\in\N^{m_{i}}$,
$$
\Big(  K[\bfx_{1},\bfx_{2}]/I(V_1\times
V_2)\Big)_{\bfdelta_{1},\bfdelta_{2}} \simeq \Big(
K[\bfx_{1}]/I(V_1)\Big) _{\bfdelta_{1}}\otimes\Big(
K[\bfx_{2}]/I(V_2)\Big)_{\bfdelta_{2}}.
$$
Hence $ H_{V_1\times V_2}(\bfdelta_{1},\bfdelta_{2})=
H_{V_1}(\bfdelta_{1})H_{V_2}(\bfdelta_{2}) $ and therefore
{$P_{V_1\times V_2}=P_{V_1}P_{V_2}$}. This implies the equality of
mixed degrees, which in turn implies that $[V_1\times V_2]=
[V_1]\otimes [V_2] $ under the identification in~\eqref{item:10}.
\end{proof}

We end  this section with the notion of ruled join of projective
varieties. Let  $n_i\in \N$  and consider an irreducible $K$-variety
$V_i \subset\P^{n_i}_K$ for $i=1,2$. Let $K[\bfx_{i}]$ denote the
homogeneous coordinate ring of $\P^{n_{i}}_{K}$ and $I(V_{i})\subset
K[\bfx_{i}]$ the ideal of~$V_{i}$. The  {\em ruled join of $V_1$ and
$V_2$}, denoted $V_1\# V_2$, is the irreducible subvariety of
$\P^{n_1+n_2+1}_K$ defined by the homogeneous ideal generated by
$I(V_{1})\cup I(V_{2})$ in $K[\bfx_{1},\bfx_{2}]$. In case~$K$ is
algebraically closed, identifying $\P^{n_{1}}$ and $\P^{n_{2}}$ with
the linear subspaces of $\P^{n_1+n_2+1}$ where the last $n_{2}+1$
(respectively, the first $n_{1}+1$) coordinates vanish, $V_{1}\#
V_{2}$ coincides with the union of the lines of $\P^{n_1+n_2+1}$
joining points of $V_{1}$ with points of $V_{2}$.

The notion of ruled join extends to equidimensional $K$-cycles by
linearity. Given  cycles $X_{i} \in Z_{r_{i}}(\P^{n_{i}}_{K})$, $i=1,2$, the ruled join $X_{1}\# X_{2}$ is a cycle of
$\P^{n_1+n_2+1}_K$ of pure dimension $r_{1}+r_{2}+1$ and degree
\begin{equation}\label{degjoin} \deg(X_1\# X_2)=\deg(X_{1})\, \deg(X_{2}),\end{equation}
 see for instance
\cite[Example 8.4.5]{Fulton84}. For $i=1,2$ consider the
injective  $\Z$-linear map $\jmath_{i}:A^{*}(\P^{n_{i}}_{K})\hookrightarrow
A^{*}(\P^{n_{1}+n_{2}+1}_{K})$ defined by $\theta^{l}\mapsto
\theta^{l}$ for $0\le l\le n_{i}$. Then \eqref{degjoin} is equivalent to
the equality of classes
\begin{displaymath}
[X_{1}\#X_{2}]=  \jmath_{1}([X_{1}])\cdot \jmath_{2}([X_{2}]).
\end{displaymath}

\subsection{Eliminants and resultants}
\label{sec:elim-result-mult}
In this section, we introduce the notions and basic properties of
eliminants of varieties and of resultants of cycles in multiprojective
spaces. This is mostly a reformulation of the  theory of eliminants
and resultants of
multihomogeneous ideals  developed by
R\'emond in~\cite{Remond01a,Remond01b} as an extension of Philippon's theory
of eliminants  of homogeneous ideals~\cite{Philippon86}.
We refer the reader to these articles for a complementary presentation of the
subject.

\medskip
We keep the notation of \S\ref{sec:cycl-mult-spac}. In particular,
we denote by $A$ a factorial ring with field of fractions $K$. Let
$V\subset \P^{\bfn}_K$ be an irreducible $K$-variety of dimension
$r$. Let $ \bfd_0, \dots, \bfd_r \in \N^{m}\setminus \{\bf0\}$ and
set $\bfd=(\bfd_0, \dots, \bfd_r )$. For each $0\le i\le r$, we
introduce a group of variables $\bfU_i= \big\{u_{i,\bfa} : \bfa\in
\N^{\bfn+\bfone}_{\bfd_i}\big\}$ and consider the {general
form} $F_i$ of multidegree $\bfd_i$ in the variables $\bfx$:
\begin{equation}
  \label{eq:15}
F_i = \sum_{\bfa\in \N^{\bfn+\bfone}_{\bfd_{i}}} u_{i,\bfa} \,
\bfx^\bfa \in K[\bfU_{i}][\bfx].
\end{equation}

Set $\bfU=\{\bfU_0, \dots, \bfU_r\}$ and consider the $K[\bfU]$-module
$$
\cM_{\bfd}(V)= K[\bfU][\bfx_1, \dots,
\bfx_m] / \Big( I(V)+(F_0, \dots, F_r) \Big).
$$
This module inherits a multigraded structure from
$K[\bfx]$.
For $\bfdelta \in \N^{m}$, we denote by $\cM_{\bfd}(V)_\bfdelta$  its part of
multidegree $\bfdelta$ in the variables $\bfx$.
It is a  $K[\bfU]$-module multigraded
by setting $\bfdeg(u_{i,\bfa})=\bfe_{i}\in \N^{r+1}$, the $(i+1)$-th vector of
the standard basis of $\R^{r+1}$.

For the sequel, we fix a set of representatives of the irreducible
elements of $K[\bfU]$ made out of primitive polynomials in $A[\bfU]$
and we denote it
 by $\irr(K[\bfU])$.
We recall that the {annihilator} of a $K[\bfU]$-module $M$ is the ideal of $K[\bfU]$ defined as
\begin{displaymath}
  \Ann(M)=\Ann_{K[\bfU]}(M)=\{ f\in K[\bfU]: \, fM=0\}.
\end{displaymath}

 \begin{definition}
   \label{def:10}     Let  $M$ be a
finitely generated $K[\bfU]$-module.
If  $\Ann(M)  \ne 0$, we~set
\begin{equation}
  \label{eq:9}
\chi(M)=
\chi_{K[\bfU]}(M)
 =\prod_{f\in \, \irr(K[\bfU])}  f^{\ell(M_{(f)})},
\end{equation}
where
$\ell(M_{(f)})$ denotes the length of the $K[\bfU]_{(f)}$-module
$M_{(f)}$.
In case $\Ann(M)  =0$, we set $\chi(M)=0$.
\end{definition}

We have that $\ell(M_{(f)})\ge 1$ if and only if $\Ann(M) \subset
(f)$, see \cite[\S 3.1]{Remond01a}. Hence, the product in \eqref{eq:9}
involves a finite number of factors and $\chi(M)$ is well-defined.

\begin{lemma}
  \label{lemm:6}
Let $V\subset \P^{\bfn}_K$ be an irreducible $K$-variety of
dimension $r$ and $ \bfd \in (\N^{m}\setminus \{{\bf0}\})^{r+1}$.
Then there exists  $\bfdelta_{0} \in \N^{m}$  such that
\begin{displaymath}
  \Ann(\cM_{\bfd}(V)_\bfdelta) =
\Ann(\cM_{\bfd}(V)_{\bfdelta_{0}}) \quad, \quad
\chi(\cM_{\bfd}(V)_{\bfdelta})=\chi(\cM_{\bfd}(V)_{\bfdelta_{0}})
\end{displaymath}
for all $\bfdelta\in \N^{m}$ such that $\bfdelta\ge \bfdelta_{0}$.
\end{lemma}

\begin{proof}
Let $\bfdelta_{\max}\in \N^{m}$ be the maximum of the multidegrees of
a set of generators of $\cM_{\bfd}(V)$ over $K[\bfU]$.
For $\bfdelta'\ge \bfdelta\ge\bfdelta_{\max}$ we have that
$\cM_{\bfd}(V)_{\bfdelta'}= K[\bfU][\bfx]_{\bfdelta'-\bfdelta}
\, \cM_{\bfd}(V)_\bfdelta$ and so
$\Ann(\cM_{\bfd}(V)_{\bfdelta'})\supset \Ann(\cM_{\bfd}(V)_\bfdelta)$.
Hence the annihilators of the parts of multidegree $\ge
\bfdelta_{\max}$  form an ascending chain of ideals with respect to
the  order $\le $ on $\N^{m}$.
Eventually, this chain stabilizes because $K[\bfU]$ is Noetherian,
which proves the first statement.
The second statement is~\cite[Lem.~3.2]{Remond01a}.
\end{proof}

We define eliminants and resultants following \cite[Def.~2.14 and \S 3.2]{Remond01a}.
 The  \emph{principal part}
$\ppr(\mathfrak{I})$ of an ideal $\mathfrak{I}\subset K[\bfU]$ is defined as
any primitive polynomial in $A[\bfU]$ which is a greatest common divisor of the elements
in $\mathfrak{I}\cap A[\bfU]$.
If $\mathfrak{I}$ is principal, this polynomial can be equivalently defined as any primitive polynomial in
$A[\bfU]$ which is a generator of $\mathfrak{I}$.
The principal part of an ideal is unique up to a unit of $A$ and we fix its choice by
supposing that it is a product of elements of  $\irr(K[\bfU])$.

\begin{definition}
\label{def:5} Let $V\subset \P^{\bfn}_K$ be an irreducible
$K$-variety of dimension $r\ge 0$ and
$ \bfd \in (\N^{m}\setminus \{{\bf0}\})^{r+1}$.
The \emph{eliminant ideal of $V$ of index $\bfd$} is defined as
\begin{displaymath}
\mathfrak{E}_{\bfd}(V)=\Ann(\cM_{\bfd}(V)_\bfdelta)
\end{displaymath}
 for any $\bfdelta \gg \bf0$.
The \emph{eliminant of $V$ of index $\bfd$} is defined as
\begin{displaymath}
\Elim_{
  \bfd}(V)=\ppr(\mathfrak{E}_{\bfd}(V)).
\end{displaymath}
\end{definition}

The eliminant ideal is a non-zero  multihomogeneous prime ideal in
$K[\bfU]$~\cite[Lem.~2.4(2) and Thm.~2.13(1)]{Remond01a}, see also
Proposition \ref{prop:7}\eqref{item:20} below.
In particular, the eliminant is a primitive irreducible
multihomogeneous polynomial in
$A[\bfU]\setminus \{0\}$.

\begin{definition}\label{def:2}
 Let $V\subset \P^{\bfn}_K$ be an irreducible
$K$-variety of dimension $r\ge 0$ and
$ \bfd \in (\N^{m}\setminus \{{\bf0}\})^{r+1}$.
The {\em resultant of $V$ of index $\bfd$} is defined as
$$
\Res_{\bfd} (V)=\chi\Big(\cM_{\bfd}(V)_\bfdelta
\Big)
$$
for any $\bfdelta\gg\bf0$. It is a non-zero primitive multihomogeneous
polynomial in $A[\bfU]$, because of Definition \ref{def:10} and the
fact that $\mathfrak{E}_{\bfd}(V)$ is non-zero.

Let  $X\in Z_r(\P^\bfn_K)$ and write $X=\sum_{V}m_{V}V$. The {\em
  resultant of $X$ of index $\bfd$}
is defined~as
\begin{displaymath}
\Res_{\bfd} (X)= \prod_V
\Res_{\bfd} (V)^{m_V} \in K(\bfU)^{\times}.
\end{displaymath}
When $X$ is effective, $\Res_{\bfd} (X)$ is a
primitive multihomogeneous polynomial in $A[\bfU]$.
\end{definition}

Eliminants and resultants are
invariant under index permutations.
Next result follows easily from the definitions:

\begin{proposition}\label{prop:6-item:27} Let $X\in
  Z_{r}(\P^{\bfn}_{K})$ and $V\subset \P^{\bfn}_K$ an irreducible
 $K$-variety of dimension~$r$. Let $ \bfd=(\bfd_{0},\dots, \bfd_{r}) \in (\N^{m}\setminus
 \{{\bf0}\})^{r+1}$ and $\bfU=(\bfU_{0},\dots,\bfU_{r})$ the group of variables
corresponding to $\bfd$. Let $\sigma$ be a permutation of
the set $\{0,\dots, r\}$ and write $\sigma \bfd=
(\bfd_{\sigma({0})},\dots, \bfd_{\sigma({r})})$, $\sigma
\bfU=(\bfU_{\sigma({0})},\dots, \bfU_{\sigma({r})})$.
Then
 $\Res_{\sigma \bfd}(X)(\sigma \bfU)
=\Res_{\bfd}(X)(\bfU)$ and
$\Elim_{\sigma \bfd}(V)(\sigma \bfU)
=\Elim_{\bfd}(V)(\bfU)$.
\end{proposition}

Eliminants and resultants are also invariant under field extensions.

\begin{proposition}
 \label{prop:6-1} Let $X\in Z_{r}(\P^{\bfn}_{K})$, $V\subset \P^{\bfn}_K$ an irreducible
 $K$-variety of dimension~$r$ and $ \bfd=(\bfd_{0},\dots, \bfd_{r}) \in (\N^{m}\setminus
 \{{\bf0}\})^{r+1}$.
Let $K\subset E$ be a field extension. Then there exists
$\lambda_{1}\in E^{\times}$ such that $\Res_{\bfd}(X_{E})=  \lambda_{1}\,
\Res_{\bfd}(X)$. Furthermore, if $V_{E}$ is an
 irreducible $E$-variety, then there exists  $\lambda_{2}\in E^{\times}$ such
 that $ \Elim_{\bfd}(V_{E})=  \lambda_{2}\, \Elim_{\bfd}(V)$.
 \end{proposition}

To prove this, we need the following lemma. 

\begin{lemma} \label{field_extension}
Let $M$ be a finitely generated $K[\bfu]$-module and $K\subset E$ a
field extension. Then $ \Ann_{E[\bfu]}(M\otimes_{K}E ) =
\Ann_{K[\bfu]}(M)\otimes_{K}E $ and $ \chi_E(M\otimes_{K}E)
=\lambda\, \chi_K(M)$ with $ \lambda\in E^{\times}$.
\end{lemma}


\begin{proof} The first statement is a consequence of the fact that $E$ is a flat
$K$-module: for each $m\in M$, the exact sequence
$$0 \to \Ann_{K[\bfu]}(m) \to K[\bfu] \to K[\bfu]m \to 0$$
yields the tensored exact sequence
$$0 \to \Ann_{K[\bfu]}(m)\otimes_KE \to E[\bfu]\to E[\bfu](m\otimes_{K}1) \to 0.$$
Hence, 
$\Ann_{E[\bfu]}(m\otimes_{K}1)=\Ann_{K[\bfu]}(m)\otimes_KE$.
Now, if $M=(m_1,\dots,m_\ell)$, then
\begin{multline*}\Ann_{E[\bfu]}(M\otimes_KE)= \bigcap_i
\Ann_{E[\bfu]}(m_{i}\otimes_{K}1)\\ = \bigcap_i\Ann_{K[\bfu]}(m_{i})\otimes_KE = 
\Ann_{K[\bfu]}(M)\otimes_KE.
\end{multline*}

For the second statement, we can reduce to the case when $\Ann(M)\ne
0$ because otherwise it is trivial. Let $f\in \irr(K[\bfU]) $. The
localization $K[\bfU]_{(f)}$ is a principal local  domain and so $
M_{(f)}\simeq \bigoplus_{i=1}^N K[\bfU]_{(f)}/(f^{\nu_i}) $ for some
$\nu_{i}\ge 1$. In particular, $\ell(M_{(f)} )=\sum_i \nu_i$. Let
$f=\lambda_{f}\, \prod_{g} g^{\mu_g}$ be  the factorization of $f$
into elements $g\in \irr(E[\bfU])$ and a non-zero constant
$\lambda_{f}\in E^{\times}$. On the one hand, for each  $g$ in this
factorization,
 $$ \Big(M\otimes _{K}E\Big)_{(g)}\simeq M_{(f)}\otimes_{K[\bfU]}
E[\bfU]_{(g)}\simeq \bigoplus_{i=1}^N E[\bfU]_{(g)}/(g^{\mu_g
\nu_i}).
$$
Hence $\ell((M\otimes _{K}E)_{(g)})=(\sum \nu_i)\mu_g = \ell(M_{(f)}
) \mu_{g}$. On the other hand, let $g\in \irr(E[\bfU])$ be an
irreducible polynomial which does not divide any $f\in
\irr(K[\bfU])$ and suppose that $\ell((M\otimes _{K}E)_{(g)})\ge 1$.
If this were the case, we would have $ \Ann_{E[\bfU]}(M\otimes_{K}E
)\subset (g)$. By the (already proved) first part of this
proposition,
$$ \Ann_{E[\bfU]}(M\otimes_{K}E) = \Ann_{K[\bfU]}(M)\otimes_{K}E .
$$
This implies $\Ann(M) \subset (g)\cap K[\bfU]=0$, which contradicts
the assumption that $\Ann(M)\ne 0$. Therefore, $\ell((M\otimes
_{K}E)_{(g)})=0$. We deduce
\begin{displaymath}
\prod_{g\in \,\irr(E[\bfU])} g^{\ell((M\otimes E)_{(g)} ) }=
\prod_{f\in \,\irr(K[\bfU])}\bigg(\prod_{g|f} g^{\ell(M_{(f)} )
\mu_g}\bigg) =  \lambda \prod_f f^{\ell(M_{(f)} ) }
\end{displaymath}
for $\lambda= \prod_f \lambda_{f}^{-\ell(M_{(f)} ) }\in E^{\times}$.
Hence $\chi_{_{E[\bfU]}}(M\otimes _{K}E) = \lambda\,
\chi_{_{K[\bfU]}}(M)$, as stated.

\end{proof}

 \begin{proof}[Proof of Proposition \ref{prop:6-1}]
To prove the  first part, it is enough to consider the case of an
irreducible $K$-variety $V$.
By definition,~$V_E$ is the effective $E$-cycle defined by
 the extended ideal $I(V)\otimes_{K}E$. Hence,
 $$\cM_{\bfd}(V_{E})=E[\bfU][\bfx_1, \dots, \bfx_m] / \Big(
 I(V)\otimes_{K}E +(F_0, \dots, F_r) \Big)=\cM_\bfd(V)\otimes_K E .$$
 By~\cite[Thm.~3.3]{Remond01a} and
 Lemma~\ref{field_extension},
 $$
 \Res_{\bfd} (V_{E}) =
 \chi_{_{E[\bfU]}}\Big(\cM_{\bfd}(V_{E})_{\bfdelta})=\lambda
 \,
 \chi_{_{K[\bfU]}}\Big(\cM_{\bfd}(V)_{\bfdelta})=\lambda_{1}\,\Res_{\bfd}
 (V)
 $$
 for any $\bfdelta\gg 0$ and a $\lambda_{1}\in E^{\times}$, which
 proves the first part of the statement.
The second part  follows similarly from the definition of
 eliminants and
 Lemma~\ref{field_extension}.
 \end{proof}

\begin{proposition}
  \label{prop:1}
Let $V\subset \P^{\bfn}_K$ be an irreducible $K$-variety of
dimension $r$ and $ \bfd \in (\N^{m}\setminus \{\bfzero\})^{r+1}$. Then
there exists  $\nu \ge 1$ such that
\begin{displaymath}
 \Res_{\bfd}(V) =  \Elim_{\bfd}(V)^{\nu}.
\end{displaymath}
\end{proposition}

\begin{proof}
 Let $\bfdelta\in \N^{m}$ and $f\in \irr(K[\bfU])$. We have that
 $(\cM_{\bfd}(V)_{\bfdelta})_{(f)}\ne 0$ if and only if
$ \Ann(\cM_{\bfd}(V)_{\bfdelta})\subset (f)$. Therefore, $$f\mid
 \Res_{\bfd}(V)\iff f\mid \Elim_{\bfd}(V).$$
Thus  $ \Elim_{\bfd}(V)$ is the only irreducible factor of $
\Res_{\bfd}(V)$ and the statement follows.
\end{proof}

\begin{example}
  \label{exm:2} The resultant of an irreducible variety
    is not necessarily an irreducible polynomial: consider the curve $C=V(x_{1,0}^{2}x_{2,1}-x_{1,1}^{2}x_{2,0})
  \subset \P^{1}\times\P^{1}$ and the indexes $\bfd_{0}=\bfd_{1}=(0,1)$ with associated linear forms
  $F_i=u_{i,0}x_{2,0}+u_{i,1}x_{2,1}$ for $i=0,1$.
We can verify that the corresponding  resultant is
$$
\Res_{\bfd_{0},\bfd_{1}}(C)= (u_{0,0}u_{1,1}-u_{0,1}u_{1,0})^{2}.
$$
\end{example}

The partial degrees of resultants can be expressed in terms of mixed
degrees.

\begin{proposition}
  \label{prop:4}
Let $X\in Z_{r}^{+}(\P^{\bfn}_K)$ and $ \bfd \in (\N^{m}\setminus
\{{\bf0}\})^{r+1}$. Then, for $0\le i\le r$,
\begin{displaymath}
  \deg_{\bfU_{i}}(\Res_{\bfd}(X)) = \coeff_{\bftheta^{\bfn}}\bigg(
[X] \, \prod_{j\ne i} \Big(\sum_{\ell=1}^{m} d_{j,\ell}\, \theta_{\ell}\Big) \bigg) .
\end{displaymath}
\end{proposition}

\begin{proof}
This follows from~\cite[Prop.~3.4]{Remond01a}.
\end{proof}

For projective varieties, eliminants and resultants
coincide:

\begin{corollary}
\label{cor:3} For $n\in \N$ let $V\subset \P^{n}_K$  be
an irreducible $K$-variety of dimension $r$ and $\bfd\in (\Z_{>0})^{r+1}$. Then $ \Res_{\bfd}(V) =
\Elim_{\bfd}(V)$.
\end{corollary}

\begin{proof}
By Proposition~\ref{prop:1}, $\Res_{\bfd}(V) = \Elim_{\bfd}(V)^{\nu}$ for
some $\nu\ge1$.
On the one hand,
$\deg_{\bfu_{i}}(\Elim_{\bfd}(V))= \Big( \prod_{j\ne i}d_{j}\Big)
\deg(V) $ for all $i$~\cite[Remark to Lem.~1.8]{Philippon86} while on the
other hand,
Proposition~\ref{prop:4} implies that
$$\deg_{\bfu_{i}}(\Res_{\bfd}(V))=\coeff_{\theta^n}\bigg(
\deg(V)\theta^{n-r}\prod_{j\ne i}d_j\theta\bigg)
= \bigg( \prod_{j\ne i}d_{j}\bigg)
\deg(V) .$$
Thus $ \Res_{\bfd}(V)$ and $\Elim_{\bfd}(V)$ have the same total
degree. Hence,  $\nu=1$ and the statement follows.
\end{proof}

Given a subset $J\subset \{1,\dots,m\}$ we  set $\pi_{J}:\P^{\bfn}_K\to
 \prod_{j\in J}\P^{n_{j}}_K$ for the natural projection and
 $\bfx_{J}= (\bfx_{j})_{j\in J}$.

\begin{lemma}
\label{prop:8} Let  $\bfd\in (\N^{m}\setminus \{\bfzero\})^{r+1}$
and $F_{i}$ the associated general form of multidegree~$\bfd_{i}$ for $0\le i\le r$. Let $V\subset \P^{\bfn}_K$ be an
irreducible $K$-variety of dimension $r\ge0$. Then
$\mathfrak{E}_{\bfd}(V)$ is a principal ideal if and only if
\begin{equation}\label{eq:17}
  \dim(\pi_{J}(V)) \ge \#\{ i: 0\le i\le r, F_{i}\in
  K[\bfU_i][\bfx_{J}]\}-1 \quad \text{ for all }J\subset \{1,\dots,m\}.
\end{equation}
If this is the case,  $\Elim_{\bfd}(V)\in K[\bfU]\setminus K$.
Otherwise, $\mathfrak{E}_{\bfd}(V)$ is not
principal and $\Elim_{\bfd}(V)=1$.
\end{lemma}

\begin{proof}
Assume for the moment that the field $K$ is infinite.

\smallskip
($\Leftarrow$) If~\eqref{eq:17} holds,
\cite[Cor.~2.15(2)]{Remond01a} implies that $\Elim_{\bfd}(V)$
generates~$\mathfrak{E}_{\bfd}(V)$. Applying
\cite[Cor.~2.15(1)]{Remond01a}, it follows $\Elim_{\bfd}(V)\ne 0$
since, for $J=\{1,\dots,m\}$,
$$
\dim({\pi_{J}(V)}) = \dim(V)= r  = \#\{ i\in \{0,\dots,r\} : F_{i}\in K[\bfU_i][\bfx]\} -1.
$$
Now suppose that $\Elim_{\bfd}(V)=1$. This is equivalent to the
fact that $
  \cM_{\bfd}(V)_{\bfdelta}=0
$ for $\bfdelta\gg0$, which implies  $I(V)
\supset(\mathfrak{M}_{\bfn})_{\bfdelta}$ and so $V=\emptyset$, which
is a contradiction.
Therefore,  $\Elim_{\bfd}(V)\in K[\bfU]\setminus K$.

\smallskip
($\Rightarrow$) Suppose that~\eqref{eq:17} does not hold.
By~\cite[Cor.~2.15(3)]{Remond01a}, $\Elim_{\bfd}(V)=1$.
Hence $\mathfrak{E}_{\bfd}(V)$ is necessarily not principal, because
otherwise we would have that $V=\emptyset$.

\smallskip
 The case when $K$  is a finite
field reduces to the previous case, by considering any
transcendental extension $E$ of $K$ and applying
Proposition~\ref{prop:6-1}.
\end{proof}

Given $\bfd\in (\N^{m}\setminus \{0\})^{r+1}$, the space of coefficients of a family of multihomogeneous
polynomials in $\ov K[\bfx]\setminus \{0\}$ of
multidegrees $\bfd_0, \dots, \bfd_{r}$  can be identified with
\begin{displaymath}\P^\bfN:=
\prod_{i=0}^{r}\P^{\prod_{j=1}^{m}
  {d_{i,j}+n_{j}\choose n_{j}}-1}.
\end{displaymath}
For $V\subset \P^{\bfn}$ consider the following subset of $\P^\bfN$:
\begin{equation}\label{nabla}
\nabla_{\bfd}(V)=\{ (\bfu_{0},\dots, \bfu_{r}) \in \P^\bfN: \ V \cap
V\big(F_0(\bfu_0,\bfx),\dots, F_{r}(\bfu_r,\bfx)\big)\ne \emptyset\}.
\end{equation}

The following results gives a geometric interpretation of eliminant
ideals.

\begin{proposition}
\label{prop:7} Let $V\subset \P^{\bfn}$ be an irreducible variety of
dimension $r\ge 0$ and $ \bfd \in (\N^{m}\setminus
\{{\bf0}\})^{r+1}$. Then
\begin{enumerate}
\item \label{item:20}  $I(\nabla_{\bfd}(V))= \mathfrak{E}_{\bfd}(V)$;
\smallskip
\item \label{item:21} the variety $\nabla_{\bfd}(V)$ is a hypersurface
  if and only if~\eqref{eq:17} holds. If this is the case, $\nabla_{\bfd}(V)= V(\Elim_{\bfd}(V))= V(\Res_{\bfd}(V))$.
\end{enumerate}
\end{proposition}

\begin{proof}
\eqref{item:20} follows from~\cite[Thm.~2.2]{Remond01a}, while
\eqref{item:21} follows from \eqref{item:20} together with
Lemma~\ref{prop:8} and Proposition \ref{prop:1}.
\end{proof}

The following corollary gives a formula {\it \`a la} Poisson for the resultants of a cycle of dimension
0. Recall that the evaluation of a multihomogeneous polynomial at a
point of $\P^{\bfn}$ is only defined up to a non-zero constant in $\ov
K^{\times}$ which depends of a choice of a representative of the
given point.

\begin{corollary} \label{cor:2}
Let $X\in Z_{0}(\P^{\bfn}_{K})$ and $\bfd_{0}\in
\N^{m}\setminus\{\bfzero\}$.
Write $X_{\ov
  K}=\sum_{{\bfxi}}m_{\bfxi}\bfxi$ with $\bfxi\in
\P^{\bfn}$  and $m_{\xi}\in \Z$  and let $F_{0}$ be the general form
of multidegree $\bfd_{0}$. Then there exists~$\lambda \in \ov
K^\times$ such that
\begin{equation*}
\Res_{\bfd_{0}}(X)=\lambda \,\prod_{\bfxi}F_0(\bfxi)^{m_{\bfxi}}.
\end{equation*}
\end{corollary}

\begin{proof}
Let $\bfxi\in \P^{\bfn}$. Observe
that $F_{0}(\bfxi)=0$ is the irreducible equation of the
hypersurface $\nabla_{\bfd}(\bfxi)$. By Proposition~\ref{prop:7},
there exists $\lambda\in \ov
 K^\times$ such that
 $\Elim_{\bfd_{0}}(\bfxi)=\lambda \,F_{0}(\bfxi)$.
Proposition~\ref{prop:4} together with Proposition
\ref{prop:5}\eqref{item:12} imply that
$\deg_{\bfU_{0}}(\Res_{\bfd}(\bfxi))= 1$. Applying
Proposition~\ref{prop:1}, we get
$\Res_{\bfd_{0}}(\bfxi)=\Elim_{\bfd_{0}}(\bfxi)$. The general case
follows  readily from the definition of the resultant and its
invariance  under field extensions.
\end{proof}

\begin{remark} \label{rem:3}
The notions of eliminant and resultant of multiprojective cycles
include several of the classical notions of resultant.
  \begin{enumerate}
\item {\it The Macaulay resultant \cite{Mac02}.} The classical resultant
of $n+1$ homogeneous polynomials
of degrees $d_0,\dots,d_n$ coincides both with
$\Elim_{(d_0,\dots,d_n )}(\P^{n}) $ and with $\Res_{(d_0,\dots,d_n
  )}(\P^{n}) $. This is a consequence of  Proposition
\ref{prop:7}\eqref{item:21} and Corollary \ref{cor:3}.
\smallskip
\item \label{segundo}
{\it Chow forms \cite{CB37}.} The Chow form of an irreducible variety
$V\subset \P^{n}$ of dimension $r$ coincides both with
$\Elim_{(1,\dots, 1)}(V) $ and with $\Res_{(1,\dots, 1)}(V) $.
This follows from Proposition \ref{prop:7} and Corollary \ref{cor:3}.

\smallskip
  \item \emph{The GKZ mixed resultant \cite[\S 3.3]{GeKaZe94}.}
Let $V$ be a proper irreducible variety over $\C$ of dimension
    $r$  equipped with a family of very ample line bundles
$L_{0},\dots,L_{r}$ and $R_{L_{0},\dots,L_{r}}$ the
{$(L_{0},\dots,L_{r})$-resultant of~$V$} in the sense of
I. Gelfand, M. Kapranov and A. Zelevinski.
Each $L_{i}$ defines an embedding
$\psi_{i}:V\hookrightarrow \P^{n_{i}}$.
We consider then the map
\begin{displaymath}
  \bfpsi: V\hooklongrightarrow  \P^{\bfn} \quad ,\quad  \bfxi \longmapsto
(\psi_{0}(\bfxi), \dots, \psi_{r}(\bfxi)).
\end{displaymath}
Using Proposition
\ref{prop:7}\eqref{item:21}, it can be shown that
$R_{L_{0},\dots,L_{r}}$ coincides with the eliminant form
$\Elim_{\bfe_{0},\dots,
  \bfe_{r}}(\bfpsi(V))$.
Using the  formula for the degree of the GKZ mixed resultant in
\cite[Thm. 3.3]{GeKaZe94}, we can show
that it also coincides with the resultant
$\Res_{\bfe_{0},\dots,  \bfe_{r}}(\bfpsi(V))$.
  \end{enumerate}
\end{remark}

\subsection{Operations on resultants}
\label{sec:oper-result-}
We will now study the behavior of resultants with respect to
basic geometric operations, including intersections, linear
projections and products of cycles.

\medskip An important feature of resultants is that they transform the
intersection product of a cycle with a divisor into an evaluation.

\begin{proposition}
\label{especializacion} Let $X\in Z^r(\P^\bfn_K)$  and $ \bfd_0,
\dots, \bfd_r \in \N^{m}\setminus \{\bf0\}$. Let $f\in
K[\bfx]_{\bfd_{r}}$ such that $\div(f)$ intersects  $X$ properly.
Then there exists $\lambda \in K^\times$ such that
$$
\Res_{\bfd_0,\bfd_1\dots,\bfd_r}(X)(\bfU_0,\dots,\bfU_{r-1},f)
=\lambda \, \Res_{\bfd_0,\dots,\bfd_{r-1}}(X\cdot
\div(f))(\bfU_0,\dots,\bfU_{r-1}),$$ {where the left-hand side
  denotes the specialization of the last of group of variables of
  $\Res_{\bfd_0,\bfd_1\dots,\bfd_r}(X)$ at the coefficients of~$f$}.
\end{proposition}

\begin{proof}
This is \cite[Prop.~3.6]{Remond01a}.
\end{proof}

Next we consider the behavior of resultants with respect to standard
projections.
Consider the
linear projection $ \pi: \P^\bfn_K\dashrightarrow\P^\bfl_K $
in \eqref{eq:58} and let
 $\bfx$ and $\bfy$ denote the multihomogeneous coordinates of
 $\P^{\bfn}_K$ and of
$\P^{\bfl}_K$, respectively. Let  $\bfd\in (\N^m\setminus
\{{\bf0}\})^{r+1}$. The general forms of multidegree $\bfd_i$ in the variables
$\bfx$ and $\bfy$ are, respectively,
$$
F_i = \sum_{\bfa\in\N^{\bfn+\bfone}_{\bfd_{i}}}  u_{i,\bfa} \,
\bfx^\bfa  \in K[\bfU_i][ \bfx]  \quad, \quad
F'_i = \sum_{\bfa\in\N^{\bfl+\bfone}_{\bfd_{i}}}  u_{i,\bfa} \, \bfy^\bfa  \in K[\bfU_i][\bfy].
$$
Write$$
\bfU'_{i}=\{u_{i,\bfa}: \ \bfa\in\N^{\bfl+\bfone}_{\bfd_{i}}\}\quad, \quad
\bfU''_{i}=\{u_{i,\bfa}: \ \bfa\in\N^{\bfn+\bfone}_{\bfd_{i}}\setminus\N^{\bfl+\bfone}_{\bfd_{i}} \}.
$$
Let $\prec$ be the partial monomial order on $
K[\bfU_0,\ldots,\bfU_r]$ defined as
$$\{\bfU'_i\}_{0\leq i\leq r}\prec \bfU''_0\prec\cdots\prec \bfU''_r.
$$
By this, we mean that the variables in each set have the same weight
and that those in $\bfU''_r$ have
the maximal weight, then come those in $\bfU''_{r-1}$, etcetera.
Observe that this order can be alternatively defined as the lexicographic order
associated to the sequence of vectors $\bfw_{0}, \dots, \bfw_{r}$
defined as
$$
\bfw_{i}=\Big(\overbrace{(\bfzero,\bfzero), \dots,
  (\bfzero,\bfzero)}^{i-1},(\bfzero,\bfone), \overbrace{(\bfzero,\bfzero), \dots,
  (\bfzero,\bfzero)}^{r-i}\Big).
$$
Given a polynomial $F\in  K[\bfU] \setminus \{0\}$, we denote by
$\init_\prec(F)\in K[\bfU] \setminus \{0\}$ its \emph{initial part}
with respect to this
order. It consists in the sum of the terms in $F$ whose monomials are
minimal with respect to $\prec$.
This order is multiplicative, in the sense that
$\init_\prec(FG)=\init_\prec(F)\init_\prec(G)$ for all $F,G\in K[\bfU]$.

\begin{proposition}\label{formainicialproy}
Let $\pi: \P^\bfn_K\dashrightarrow\P^\bfl_K$ be the linear
projection and $\prec$ the partial monomial order on $K[\bfU]$
considered above. Let $X\in Z_r^+(\P^\bfn_K)$ and $\bfd\in
(\N^m\setminus \{{\bf0}\})^{r+1}$. Then
$$
\Res_{\bfd}(\pi_* X ) \,\Big|\,
\init_\prec(\Res_{\bfd}(X))\quad \mbox{in } A[\bfU].
$$
\end{proposition}

\begin{proof}
It is enough to consider the case $\pi_*(X) \ne 0$, otherwise
$\Res_{\bfd}(\pi_* X ) =1$. In addition, we only need to  prove that
the division holds in $K[\bfU]$, since both polynomials belong to
$A[\bfU]$ and $\Res_{\bfd}(\pi_*  X)$ is primitive.

\medskip
We proceed by induction on the dimension $r$. For $r=0$, write
$X_{\ov K}=\sum_{\bfxi}m_{\bfxi}\bfxi$ with $\bfxi\in
\P^{\bfn}$ and $m_{\bfxi}\in \Z$. Hence, $$\pi_{*}X_{\ov K}= \sum_{\bfxi \notin L
}m_{\bfxi}\pi_{*}\bfxi$$ since $\pi_{*}\bfxi=0$ whenever $\bfxi\in
L$, with $L$ defined in  \eqref{eq:58}. Observe that
$F_{0}'(\pi(\bfxi))= \init_{\prec}(F_{0}(\bfxi))$ for each $\bfxi \notin L$. Using Corollary
\ref{cor:2} and the multiplicativity of the order~$\prec$, we deduce
that there exist $\lambda,\lambda'\in \ov K^\times$ such that
\begin{displaymath}
\Res_{\bfd_0}(\pi_* X )=
\lambda'\prod_{\bfxi\notin L}\init_{\prec}(F_{0}(\bfxi))^{m_{\bfxi}}
\, \bigg|\, \init_{\prec}\bigg(\prod_{\bfxi}F_{0}(\bfxi)^{m_{\bfxi}}\bigg)
= \lambda\,  \init_{\prec}(\Res_{\bfd_0}(\bfxi)).
\end{displaymath}
Now let $r\ge 1$ and suppose that we have proved the statement for all cycles of pure
dimension up to $r-1$ and any base field. Let $V$ be an irreducible
$K$-variety of dimension~$r$. We suppose that
$\Res_{\bfd}(\pi_{*}V)\ne1$ because otherwise the statement is
trivial. Thus, the degree of this resultant in some group of variables is
$\ge1$ and, up to a reordering, we can suppose that this holds for
the group $\bfU_{0}'$. Consider the scalar extension  $V_{{
K(\bfU_{r}')}}$ of the $K$-variety  $V$ by the field $K(\bfU'_r)$.
By Proposition~\ref{prop:6-1}, there exists
$\lambda_{1}\in  K(\bfU_{r}')^{\times}  $
such that
\begin{equation*}
  \label{eq:22}
\Res_{\bfd}(\pi_* V )=\lambda_{1}\,
\Res_{\bfd}(\pi_* V_{ K(\bfU_{r}') }).
\end{equation*}
In turn, Proposition~\ref{especializacion} implies that there exists $\lambda_{2}\in
K(\bfU_{r}')^{\times}$ such that
\begin{equation*}
  \label{eq:23}
\Res_{\bfd}(\pi_*(V_{ K(\bfU_{r}')}))= \lambda_{2}\,
\Res_{\bfd_0,\ldots,\bfd_{r-1}}\big(\pi_* (V_{ K(\bfU_{r}')})\cdot
{\div}_{\P^{\bfl}}(F'_r)\big)(\bfU_{0},\dots, \bfU_{r-1}).
\end{equation*}
We can verify that
the form $F_{r}'$ satisfies the hypothesis of Lemma~\ref{lemm:5}.
Hence, this lemma implies the  equality of cycles $ \pi_*\big(V_{
K(\bfU_{r}')}\cdot {\div}_{\P^{\bfn}}(F'_r)\big)= \pi_*(V_{
K(\bfU_{r}')})\cdot {\div}_{\P^{\bfl}}(F'_r) $. In particular,
\begin{equation*}\label{eq:24}
\Res_{\bfd_0,\ldots,\bfd_{r-1}}\big(\pi_* (V_{ K(\bfU_{r}')})\cdot
\div_{\P^{\bfl}}(F'_r)\big)
=
\Res_{\bfd_0,\ldots,\bfd_{r-1}}\big(\pi_*(V_{ K(\bfU_{r}')}\cdot
\div_{\P^{\bfn}}(F_{r}'))\big).
\end{equation*}
Applying the inductive hypothesis to the $(r-1)$-dimensional
cycle $V_{ K(\bfU_{r}')}\cdot
\div_{\P^{\bfn}}(F_{r}')$,
\begin{equation*}
  \label{eq:25}
  \Res_{\bfd_0,\ldots,\bfd_{r-1}}\big(\pi_*(V_{ K(\bfU_{r}')}\cdot
\div_{\P^{\bfn}}(F_{r}'))\big) \, \Big|\,
 \init_\prec\big(\Res_{\bfd_0,\ldots,\bfd_{r-1}}(V_{ K(\bfU_{r}')}\cdot
 \div_{\P^{\bfn}}(F'_r)\big)\big).
\end{equation*}
The divisor ${\div}_{\P^\bfn}(F'_r)$ intersects $V_{K(\bfU_{r}')}$
properly. Proposition~\ref{especializacion} then implies
\begin{equation*}\label{eq:26}
\Res_{\bfd_0,\dots,\bfd_{r-1}}(V_{K(\bfU_{r}')}\cdot
{\div}_{\P^\bfn}(F'_r))(\bfU_0,\dots,\bfU_{r-1}) =\lambda_{3} \,
\Res_{\bfd}(V)(\bfU_0,\dots,\bfU_{r-1},F'_r)
\end{equation*}
for some $\lambda_{3} \in K(\bfU'_r)^{\times}$.
This
last polynomial is not zero and it satisfies
$$
\init_\prec\big(\Res_{\bfd}(V)(\bfU_0,\dots,\bfU_{r-1},F'_r) \big)=
\init_\prec\big(\Res_{\bfd}(V)\big),
$$
due to the definition of $\prec$. We conclude
$$\Res_{\bfd}(\pi_* V )\, \Big| \,
\init_\prec\big(\Res_{\bfd}(V)\big)\quad \text{ in }
K(\bfU_{r}')[\bfU_{0},\dots, \bfU_{r-1}].$$
This readily implies that
$\Res_{\bfd}(\pi_* V )\, |\,
\init_\prec\big(\Res_{\bfd}(V)\big)$ in $K[\bfU]$,
because  $\Res_{\bfd}(\pi_* V )$ is a power of an irreducible
polynomial of positive degree in $\bfU_0$ and
$r>0$.

For a general $K$-cycle of pure dimension $r$, the statement follows
by applying this result to its irreducible components and using the
multiplicativity of the order $\prec$. This concludes the inductive
step.
\end{proof}

\begin{remark}
  \label{rem:9}
In the projective case ($m=1$),  this result can be alternatively
derived  from  \cite[Prop.~4.1]{PedStu93}, see also \cite[Lem.~2.6]{KrPaSo01}.
\end{remark}

Resultants corresponding to general linear forms play an important
role in the definition and study of mixed heights of cycles. We
introduce a convenient notation for handling this particular case.
Given $\bfc=(c_1,\dots,c_m)\in \N^{m}$ we set
\begin{equation}\label{d(a)}
\bfe({\bfc})=(\overbrace{\bfe_1,\dots,\bfe_1}^{c_1},\dots,\overbrace{\bfe_m,\dots,\bfe_m}^{c_m})\
\in  (\N^{m}\setminus \{0\})^{|\bfc|},
\end{equation}
where $\bfe_i$ denotes $i$-th vector of the standard basis of $\R^{m}$.
For $\bfc\in \N_{r+1}^m$, both $\Elim_{\bfe({\bfc})}(V)$ and
$\Res_{\bfe({\bfc})}(V)$ are polynomials in the coefficients of the
$r+1$ general linear forms $L_0,\dots,L_r$ corresponding to the
index $\bfe(\bfc)$.
In this case, Proposition~\ref{prop:4}
implies that, for $0\leq i\leq r$,
\begin{equation}\label{gradosparciales2}\deg_{\bfU_i}(\Res_{\bfe(\bfc)}(X))=\deg_{\bfc-\bfe_{j(i)}}(X),
\end{equation}
where $j(i)$ is the index $j$ such that
$c_1+\ldots+c_{j-1}<i+1\leq c_1+\ldots+c_{j}$.

\begin{proposition}\label{resultantesproductosciclos}
Let  $m_i\in \Z_{>0} $,  $\bfn_{i}\in \N^{m_{i}}$ and
$X_{i}\in Z_{r_{i}}(\P^{\bfn_i}_K) $  for $i=1,2$. Let
 $\bfc_{i}\in \N^{m_{i}}$ such that $|\bfc_{1}|+|\bfc_{2}|=
{r_{1}+r_{2}+1}$. Then {there exists $\lambda\in A^{\times}$
such that}
\begin{displaymath}
 \Res_{\bfe({\bfc_{1},\bfc_{2}})}(X_{1}\times X_{2})=
\left\{
\begin{array}{cl}
\lambda\, \Res_{\bfe({\bfc_1})}(X_{1})^{\deg_{\bfc_2}(X_{2})} & \mbox{ if }
    |\bfc_{1}|=r_{1}+1, |\bfc_{2}|=r_{2}, \\[2mm]
\lambda\, \Res_{\bfe({\bfc_2})}(X_{2})^{\deg_{\bfc_1}(X_{1})} & \mbox{ if } |\bfc_{1}|=r_{1}, |\bfc_{2}|=r_{2}+1, \\[2mm]
 1 & \mbox{ otherwise}.
  \end{array}
\right.
\end{displaymath}
\end{proposition}

\begin{proof}
We first prove the statement for the case when $K$ is algebraically
closed and $A=K$. Let $\bfU=\{\bfU_{0},\dots, \bfU_{r_{1}+r_{2}}\}$ be the
group of variables associated to the index $\bfe(\bfc_1,\bfc_2)$.
By~\eqref{gradosparciales2}, for each  $0\le i\le r_1+r_2$ there
is $j$, $1\le j\le m_1+m_2$, such that
\begin{displaymath}
  \deg_{\bfU_{i}}(\Res_{\bfe(\bfc_1,\bfc_2)}(X_{1}\times X_{2}))  =
  \deg_{(\bfc_1,\bfc_2)-\bfe_{j}}(X_{1}\times X_{2}).
\end{displaymath}
If either $|\bfc_{1}|\ge r_{1}+2$ or $|\bfc_{2}|\ge r_{2}+2$, then $
\deg_{(\bfc_1,\bfc_2)-\bfe_{j}}(X_{1}\times X_{2})=0$ for all $j$,
thanks to Proposition~\ref{prop:3}. Hence all  partial degrees are 0
and  $\Res_{\bfe(\bfc_1,\bfc_2)}(X_{1}\times X_{2})=1$.

\medskip
Consider then the case when
$|\bfc_1|=r_{1}+1$, $|\bfc_2|=r_{2}$.
Again by~\eqref{gradosparciales2}, for $1\le i\le r_2$,
there exists    $j> m_1$ such that
\begin{displaymath}
  \deg_{\bfU_{r_1+i}}(\Res_{\bfe(\bfc_1,\bfc_2)}(X_{1}\times X_{2}))  =
  \deg_{(\bfc_1,\bfc_2)-\bfe_{j}}(X_{1}\times X_{2}).
\end{displaymath}
By Proposition~\ref{prop:3}, this mixed degree
also vanishes, because $|\bfc_1|>r_1$. Therefore
\begin{equation}\label{dependenciavariables}
\Res_{\bfe(\bfc_1,\bfc_2)}(X_{1}\times X_{2})\in
K[\bfU_0,\dots,\bfU_{r_1}].\end{equation} Furthermore,  suppose that
$\deg_{\bfc_2}(X_2)=0$. In this case, for each $0\le i\le r_1$ there
exists  $j$, $1\le j\le m_1$, such that
$$\deg_{\bfU_i}(\Res_{\bfe(\bfc_1,\bfc_2)}(X_{1}\times X_{2}))=
 \deg_{(\bfc_1,\bfc_2)-\bfe_j}(X_1\times X_2)= \deg_{\bfc_1-\bfe_j}(X_1)\deg_{\bfc_2}( X_2)=0.$$
Hence,
$\Res_{\bfe(\bfc_1,\bfc_2)}(X_{1}\times X_{2})=1= \Res_{\bfe(\bfc_1)}(X_1)^{\deg_{\bfc_2}(X_2)}$
and the statement holds in this case.
Therefore, we assume that $\deg_{\bfc_2}(X_2)\ne 0$. By linearity, it
suffices to prove the statement for two irreducible varieties
$V_i\subset \P^{\bfn_i}$ of dimension $r_i$, $i=1,2$. We consider
first the case when $r_{2}=0$, that is, when  $V_2= \{\bfxi\}$ is a
point. Then
\begin{displaymath}
  \cM_{\bfe(\bfc_{1})}(V_1\times V_2)
= \cM_{\bfe(\bfc_{1})}(V_1)\otimes_{K}  K[\bfx_{2}]/I(\bfxi).
\end{displaymath}
Hence, for $\bfdelta_{i}\in \N^{m_{i}}$,
\begin{displaymath}
  \cM_{\bfe(\bfc_{1})}(V_1\times V_2)_{\bfdelta_{1},\bfdelta_{2}}
= \cM_{\bfe(\bfc_{1})}(V_1)_{\bfdelta_{1}}\otimes_{ K} (
K[\bfx_{2}]/I(\bfxi))_{ \bfdelta_{2}} \simeq
\cM_{\bfe(\bfc_{1})}(V_1)_{\bfdelta_{1}}.
\end{displaymath}
By the  definition of the resultant, there exists $\lambda \in
K^\times$ such that
\begin{displaymath}
  \Res_{\bfe(\bfc_1)}(V_1\times V_2)=
  \lambda\,\Res_{\bfe(\bfc_1)}(V_1)=\lambda\,
  \Res_{\bfe(\bfc_1)}(V_1)^{\deg_{{\bf0}}(V_2)},
\end{displaymath}
which proves the statement in this case.

Now let $V_2\subset \P^{\bfn_2}$ be an irreducible variety of
dimension $r_{2}\ge 1$.
Write $\bfU'=(\bfU_0,\dots,\bfU_{r_1})$ and
let $\ell_{i}\in K[\bfx_{2}]$, $1\le
i \le r_{2}$,  be generic linear forms
associated to~$\bfe(\bfc_{2})$ and $\pi_2$ the projection
$\P^{\bfn_1}\times \P^{\bfn_2}\to \P^{\bfn_2}$.
By \eqref{dependenciavariables}, $\Res_{\bfe(\bfc_1,\bfc_2)}(V_1\times
V_2)$ does not depend on the groups of variables $\bfU_{r_{1}+i}$ for
$1\le i\le r_{2}$.
Hence,
 \begin{align*}
\Res_{\bfe(\bfc_1,\bfc_2)}(V_1\times V_2)(\bfu) &=
  \Res_{\bfe(\bfc_1,\bfc_2)}(V_1\times
  V_2)(\bfU',\ell_{1},\dots,\ell_{r_2})\\
&= \lambda\, \Res_{\bfe(\bfc_1)}\bigg((V_1\times V_2)\cdot
\prod_{i=1}^{r_2}\pi_{2}^{*}\,\div(\ell_{i})  \bigg)(\bfU')\\
& =\lambda\, \Res_{\bfe(\bfc_1)}\bigg(V_1\times(V_2\cdot
\prod_{i=1}^{r_2}\div(\ell_{i})) \bigg)(\bfU')
\end{align*}
with $\lambda\in K^{\times}$, thanks to Proposition~\ref{especializacion}.
The cycle   $V_2\cdot \prod_{i=1}^{r_2}\div(\ell_{i})$ is of
dimension 0 and so we are in the hypothesis of the previous case. By
Corollary~\ref{cor:5}, it is a cycle of degree
$\deg_{\bfc_2}(V_2)$. Therefore,
$\Res_{\bfe(\bfc_1)}\Big(V_1\times(V_2\cdot
\prod_{i=1}^{r_2}\div(\ell_{i})) \Big)=\lambda'\,
\Res_{\bfe(\bfc_1)}(V_1)^{\deg_{\bfc_2}(V_2)}$ with $\lambda'\in
K^\times$, which completes the proof for the case when $K$ is
algebraically closed and $A=K$.

The case of an arbitrary field $K$ which is the field of
fractions of a factorial ring $A$ follows from  Proposition
\ref{prop:6-1} and the fact that the resultants of $V_1$, $V_2$ and
$V_{1}\times V_{2}$ are primitive polynomials in $A[\bfu]$.
\end{proof}

\section{Heights of cycles of multiprojective spaces}
\label{Mixed degrees and heights in multiprojective geometry}

\subsection{Mixed heights of cycles over function fields}\label{The height in the function field case}
Throughout this section, we denote by $k$ a field and
$\bft=\{t_{1},\dots,t_{p}\}$ a group of variables.  The {\em height}
of a polynomial {$f= \sum_{\bfa} \alpha_\bfa \, x_{1}^{a_{1}}\cdots x_{n}^{a_{n}}\in
k[\bft][x_{1},\dots, x_{n}]\setminus \{0\}$} is defined as
\begin{displaymath}
\h(f) =\deg_\bft(f)=\max_\bfa \deg(\alpha_\bfa).
\end{displaymath}
For $f=0$, we set $ \h(f)=0$. The following lemma estimates the
behavior of the height of polynomials with respect to addition,
multiplication and composition. Its proof follows directly from the
definitions.

\begin{lemma}\label{lemma:2-1} Let $f_1,\dots,f_s\in k[\bft][x_{1},\dots, x_{n}]$
and $g\in k[\bft][y_1,\dots,y_s]$. Then
\begin{enumerate}
\item $\h(\sum_i f_i)\le \max_i\,\h(f_i)$;\smallskip
\item $\h(\prod_i f_i)=  \sum_i \h(f_i)$;\smallskip
\item $\h(g(f_1,\dots,f_s))\le
\h(g)+\deg_\bfy(g)\, \max_i\,\h(f_i)$.
\end{enumerate}
\end{lemma}

In the sequel,
we extend this notion to
cycles of $\P^{\bfn}_{_{k(\bft)}}$ and study its basic properties.
To this end, we  specialize the theory  in
\S\ref{Some resultant  theory} to the case when the factorial
ring $A$ is the polynomial ring $k[\bft]$ with field of fractions
$K=k(\bft)$. In particular, the resultant of an effective
equidimensional $k(\bft)$-cycle  is a primitive polynomial in
$k[\bft][\bfU]$.

\begin{definition}
  \label{def:15}
Let  $V\subset \P^\bfn_{_{k(\bft)}}$ be an irreducible
$k(\bft)$-variety of dimension $r$, $\bfc\in \N^m_{r+1}$ and
$\bfe(\bfc)$ as in~\eqref{d(a)}. The {\it (mixed) height of
$V$ of index $\bfc$} is defined as
$$
\h_\bfc(V)=\h(\Res_{\bfe(\bfc)}(V))
=\deg_\bft(\Res_{\bfe(\bfc)}(V)).
$$
This definition extends by linearity to cycles in $
Z_r(\P^\bfn_{_{k(\bft)}})$.

For $n\in \N$ and $X\in Z_{r}(\P^n_{_{k(\bft)}})$, the
\emph{height of $X$} is defined as  $\h(X)=\h_{r+1}(X)$.
\end{definition}

\begin{definition}\label{arith chow pol}
Let $\eta$ be an indeterminate.
The \emph{extended Chow ring of $\P^{\bfn}_{_{k(\bft)}}$}
is the graded ring
\begin{displaymath}
A^{*}\Big(\P^{\bfn}_{_{k(\bft)}};{{k[\bft]}}\Big)
= A^{*}\Big(\P^{\bfn}_{_{k(\bft)}}\Big)\otimes_\Z\Z[\eta]/(\eta^{2})\simeq
\Z[\eta,\theta_1,\dots,\theta_m]/(\eta^2,\theta_1^{n_1+1},\dots,\theta_m^{n_m+1}),
\end{displaymath}
where $\theta_{i}$ denotes the class in
$A^{*}\Big(\P^{\bfn}_{_{k(\bft)}}\Big)$ of the inverse image of a
hyperplane of $\P^{n_i}_{_{k(\bft)}}$ under the projection $\P^{\bfn}_{_{k(\bft)}}\to
\P^{n_i}_{_{k(\bft)}}$.
For short, we alternatively denote this ring as
$A^{*}(\P^{\bfn};{{k[\bft]}})$.
To a cycle  $X\in Z_r\Big(\P^\bfn_{_{k(\bft)}}\Big)$ we associate an
element of this ring, namely
\begin{displaymath}
[X]_{_{k[\bft]}}=\sum_{\bfc\in \N^{m}_{r+1}, \ \bfc\le \bfn}\h_{\bfc}(X)\, \eta\,
\theta_{1}^{n_{1}-c_{1}} \cdots \theta_{m}^{n_{m}-c_{m}} +
\sum_{\bfb\in \N^{m}_{r}, \ \bfb\le \bfn}\deg_{\bfb}(X) \, \theta_{1}^{n_{1}-b_{1}}
\cdots \theta_{m}^{n_{m}-b_{m}}.
\end{displaymath}
This is a homogeneous
element of degree $|\bfn|-r$.
\end{definition}
There is an inclusion of the
Chow ring into the extended Chow ring
$$\imath: A^{*}\Big(\P^{\bfn}_{_{k(\bft)}}\Big)\hookrightarrow
A^{*}\Big(\P^{\bfn}_{_{k(\bft)}};k[\bft]\Big)$$ satisfying
\begin{math}
[X]_{_{k[\bft]}} \equiv \imath([X]) \pmod{\eta}.
\end{math}
In particular,
the class of a cycle in the Chow ring
is determined by its class in the extended Chow ring.

For a cycle $X$ of pure dimension $r$,
{we will see in Theorem \ref{bezaritff}\eqref{item:25}} that
$\h_{\bfc}(X)=0$ for every $\bfc$ such that
$c_{i}>n_{i}$ for some $i$.
Hence $[X\ck$ contains the information about all mixed degrees and heights, since
$\{\bftheta^\bfa, \eta\, \bftheta^\bfa\}_{\bfa\le \bfn}$ is a $\Z$-basis of $A^{*}(\P^{\bfn};{{k[\bft]}})$.

\medskip
The mixed heights of a $k(\bft)$-cycle $X$ can be interpreted
as some mixed degrees of a model of $X$ over $\P^{p}_{k}$. For
simplicity, we will only consider the case of projective
$k(t)$-cycles, where $t$ is a single variable.

\begin{definition}\label{def:12}
  Let $n\in \N$ and $V\subset \P^{n}_{_{k(t)}}$ be an irreducible
  $k(t)$-variety.  Let $\bfs=\{s_{0},s_{1}\}$ and $\bfx=\{x_{0},\dots,
  x_{n}\}$ be groups of variables and $\cI\subset k[\bfs, \bfx]$ the
  bihomogeneous ideal generated by all the polynomials of the form
  $s_{0}^{\deg_{t}(f)} f({s_{1}}/{s_{0}}, \bfx)$ for $f\in I(V)\cap
  k[t,\bfx]$.  The \emph{standard model of $V$ over $\P^1_{k}$} is
  defined as the $k$-variety $V(\cI) \subset \P^1_{k}\times
  \P^n_{k}$. This notion extends by linearity to cycles of
  $\P^{n}_{_{k(t)}}$: the \emph{standard model} of a cycle $ X=\sum_{V}m_V V$
  is defined as $\cX=\sum_{V}m_V \cV$, where $\cV$ denotes the
  standard model of the irreducible $k(t)$-variety $V$.
\end{definition}

\begin{remark}\label{rem:1}
Set $\cV_{0}$ and $\cV_{\infty}$ for the restriction of $\cV$ to the open
subsets $(\P^1_{k}\setminus \{(0:1)\})\times \P^n_{k}$  and
$(\P^1_{k}\setminus \{(1:0)\})\times \P^n_{k}$, respectively. These
are irreducible $k$-varieties which correspond to the prime ideals $\cI_{0}=I(V)\cap
k[t][\bfx]$ and $\cI_{\infty}=I(V)\cap k[t^{-1}][\bfx]$,
respectively, and form a covering of $\cV$. In particular, $\cV$ is an irreducible $k$-variety.
\end{remark}

A variety $\cW\subset \P^{1}_k\times \P^{n}_{k}$ is \emph{vertical} if
its projection to $\P^{1}_{k}$ consists in a single point.
The following lemma shows that the standard model of a cycle of
$\P^{n}_{_{k(t)}}$ of pure dimension $r\ge 0$ is a cycle of
$\P^{1}_{k}\times \P^{n}_{k}$ without vertical components. Moreover, there is a bijection between
$Z_{r}(\P^{n}_{_{k(t)}})$ and the set of cycles in
$Z_{r+1}(\P^{1}_{k}\times \P^{n}_{k})$ without vertical components.

 \begin{lemma} \label{lemm:2} Let $n\in \N$ and $r\ge 0$.
   \begin{enumerate}
   \item \label{item:16} Let $X\in Z_{r}(\P^{n}_{_{k(t)}})$. Then its
     standard model $\cX$ is a cycle of $\P^{1}_k\times \P^{n}_{k}$ of pure
     dimension $r+1$, without vertical fibers, and the generic fiber
     of $\cX\to \P^{1}_{k}$ coincides with $X$ under the natural identification of the generic
     fiber of $\P^1_{k}\times \P^n_{k}\to \P^{1}_{k}$ with $
     \P^{n}_{_{k(t)}}$.
\smallskip
\item \label{item:17} Let $\cY \in Z_{r+1}(\P^1_{k}\times \P^n_{k})$
  be a cycle without vertical components and $Y$ the  generic
  fiber of $\cY\to \P^{1}_{k}$. Then~$Y \in Z_{r}(\P^{n}_{_{k(t)}})$ and $\cY$ is its
  standard model.
   \end{enumerate}
\end{lemma}

 \begin{proof}\

   \eqref{item:16} It suffices to prove the statement for an
   irreducible variety $V\subset \P^{n}_{k(t)}$.  We keep the notation
   in Definition \ref{def:12} and Remark \ref{rem:1} and we denote by
   $\pi$ the projection $\cV\to \P^{1}_{k}$. Suppose that $\cV$ is
   vertical. This would imply that $\pi(\cV_{0})$ is a point. Hence,
   there exists $f\in k[t]\setminus \{0\}$ such that
   $\pi(\cV_{0})\subset V(f)$ or, equivalently, such that $f\in
   \cI_{0}$. But this would imply that $1\in I(V)$ and {\it a fortiori},
   $V=\emptyset$. This is contrary to our assumptions and so we deduce
   that $\cV$ is not vertical.

   Choose $0\le i\le n$ such that $V$ is not contained in $V(x_{i})$.
   For simplicity, we suppose that $i=0$. Set $V_{x_{0}}=V\setminus
   V(x_{0})\subset \P^{n}_{_{k(t)}}\setminus V(x_{0})\simeq
   \A^{n}_{_{k(t)}}$. Set $x_{i}'=x_{i}/x_{0}$, $1\le i\le n$, and let
   $J\subset k(t)[x'_{1},\dots, x_{n}']$ be the ideal of
   $V_{x_{0}}$. Then $J\cap k[t,x'_{1},\dots, x_{n}']$ is the ideal of
   the $k$-variety $(\cV_{0})_{x_{0}}:= \cV_{0}\setminus V(x_{0})
   \subset (\A_{k}^{1}\times \P^{n}_{k})\setminus V(x_{0})\simeq
   \A_{k}^{1}\times \A^{n}_{k}$. Hence,
\begin{displaymath}
 (\cV_{0})_{x_{0}}\times_{\A^{1}_{k}}\Spec(k(t))=
 \Spec(k[t,\bfx']/I((\cV_{0})_{x_{0}}) \otimes_{k} k(t))=
\Spec(k(t)[\bfx']/J)= V_{x_{0}}.
\end{displaymath}
The generic fiber of $\pi$ coincides with the closure in
$\P^{n}_{_{k(t)}}$ of the generic fiber of $(\cV_{0})_{x_{0}} \to
\A^{1}_{k}$, namely, with $\ov{V_{x_{0}}}=V$. The fact that $\pi$ is
surjective with generic fiber $V$ implies that
$\dim(\cV)=\dim(V)+\dim(\P^{1}_{k})=r+1$ by the theorem of
dimension of fibers.

\smallskip
\eqref{item:17} It is sufficient to consider the case of an
irreducible $k$-variety $\cW\subset \P^{1}_{k}\times \P_{k}^{n}$ without vertical components.
   Let $\cJ_{0}\subset k[t][\bfx]$ be the ideal of the
restriction of $\cW$ to the open
subset  $(\P^1_{k}\setminus \{(0:1)\})\times \P^n_{k}\simeq \A^1_{k}\times \P^n_{k}$. It is a prime
ideal of Krull dimension $r+2$ and $\cJ_{0}\cap k[t]= \{0\}$ since
$\pi$ is surjective. Hence,
$J:=k(t)\otimes_{k[t]} \cJ_{0}$ is a homogeneous prime ideal of Krull
dimension $r+1$ which defines the generic fiber of $\pi$. Moreover,
$\cJ_{0}= J\cap k[t][\bfx]$ and so $\cW$ is the
standard model  of $W$.
 \end{proof}

\begin{lemma} \label{lemm:16}
  Let $I\subset k(t)[\bfx]$ be an equidimensional ideal,  $X(I)\in Z(\P^{n}_{_{k(t)}})$ its
  associated cycle and $\cX$ the standard model of $X(I)$.
Let $\cI\subset k[\bfs,\bfx]$ be the ideal generated
by all the polynomials of the form
$s_{0}^{\deg_{t}(f)} f({s_{1}}/{s_{0}}, \bfx)$ for $f\in
I\cap k[t,\bfx]$. Then $\cX= X(\cI)$.
\end{lemma}

\begin{proof}
  This can be verified by going through the ideals: the minimal primes
  of $I$ are in bijection with the minimal primes of $\cI$, and this
  bijection preserves multiplicities.
\end{proof}

There is {an} isomorphism $\phi :A^{*}(\P^{n}_{_{k(t)}};k[t])\to
A^{*}(\P^{1}_{k}\times \P^{n}_{k})$ which sends the generators $\eta,
\theta_{1}\in A^{*}(\P^{n}_{k};k[t])$ to the generators $\theta_{1},
\theta_{2}\in A^{*}(\P^{1}_{k}\times \P^{n}_{k})$, respectively.  Next
result shows that, {\it via} this isomorphism, the class of a
projective $k(t)$-cycle $X$ identifies with the class of its standard
model $\cX$. In particular, the height of $X$ coincides with a mixed
degree of $\cX$.

\begin{proposition}\label{lemm:12}
Let  $X\in Z_{r}(\P^{n}_{_{k(t)}})$ and $\cX$ be the standard model of
$X$. Then
\begin{displaymath}
 \phi \Big([X]_{_{k(t)}}\Big)= [\cX] .
\end{displaymath}
Equivalently, $\deg(X)= \deg_{1,r}(\cX)$ and $\h(X)=
\deg_{0,r+1}(\cX).$
\end{proposition}

\begin{proof}\ It is enough to prove the statement for an irreducible
  $k(t)$-variety $V$ with standard model $\cV$.  Set
  $\bfu=\{\bfu_{0},\dots, \bfu_{r}\}$, $\bfone=(1,\dots,1)\in
  \N^{r+1}$, $\bfe=(1,0)$ and $\wt{\bfe}=((0,1),\dots, (0,1))\in
  (\N^{2})^{r+1}$. We first claim that there exists $\lambda \in
  k^{\times}$ such that
  \begin{equation}\label{item:27}
\Res_{{\bf1}}(V)(\bfu)= \lambda \, \Res_{\bfe,\wt{\bfe}}(\cV)
((-t,1),\bfu).
\end{equation}

Let $I\subset k(t)[\bfx]$ and $\cI\subset k[\bfs, \bfx]$ denote the
homogeneous ideal of $V$ and the bihomogeneous ideal of $\cV$,
respectively. Set
$$
\cJ_{1}=k(t)\otimes _{k}\cI + (s_{1}-ts_{0}), \quad \cJ_{2}= k(t)[\bfs]
\otimes_{_{k(t)}} I
 + (s_{1}-ts_{0}) \quad \subset  k(t)[\bfs, \bfx].
$$
These ideals define the subschemes $\cV_{_{k(t)}}\cap V(s_{1}-ts_{0})$
and $\{(1:t)\}\times V$ of $\P^{1}_{_{k(t)}}\times\P^{n}_{_{k(t)}}$, respectively.
For $f\in k[t,\bfx]$, set
\begin{displaymath}
f^{\hom}=s_{0}^{\deg_{t}(f)} f\bigg(\frac{s_{1}}{s_{0}}, \bfx\bigg)\in k[\bfs,\bfx].
\end{displaymath}
This  polynomial   is homogeneous of degree $\deg_{t}(f)$ with
respect to the variables $\bfs$. Observe that
\begin{equation}
  \label{eq:68}
  f^{\hom} \equiv s_{0}^{\deg_{t}(f)} f \equiv (t^{-1}s_{1})^{\deg_{t}(f)} f \quad \pmod{(s_{1}-ts_{0})}.
\end{equation}
The ideal $\cI $ is generated by $\{f^{\hom}: \, f\in I\cap
k[t,\bfx]\}$. By \eqref{eq:68}, these generators lie in~$\cJ_{2}$;
hence  $\cI \subset \cJ_{2}$ and  so $\cJ_{1}\subset \cJ_{2}$.
The equations \eqref{eq:68} also  imply that $ I $ is contained in the
localization $(\cJ_{2})_{s_{i}}\subset k(t)[\bfs,\bfx]_{s_{i}}$ for $i=0,1$.
Hence, $(\cJ_{1})_{s_{i}}= (\cJ_{2})_{s_{i}}$ and so
both ideals define the same subschemes of
$\P^{1}_{_{k(t)}}\times\P^{n}_{_{k(t)}}$. In particular,
$\div(s_{1}-ts_{0})$ intersects $\cV_{_{k(t)}}$ properly and we have
the equality of cycles
\begin{equation}
  \label{eq:69}
   \cV_{_{k(t)}}\cdot \div(s_{1}-ts_{0})=
\{(1:t)\}\times V \  \in \
Z_{r}(\P^{1}_{_{k(t)}}\times\P^{n}_{_{k(t)}}).
\end{equation}
Applying successively propositions \ref{prop:6-1} and~\ref{especializacion},
\eqref{eq:69} and Proposition \ref{resultantesproductosciclos}, we
obtain that there exists {$\mu_{0}, \mu \in k[t]\setminus \{0\}$}
such that
\begin{multline} \label{eq:37}
\Res_{\bfe,\wt{\bfe}}(\cV) ((-t,1),\bfU)= \mu_{0} \,
\Res_{\wt{\bfe}}(\cV_{_{k(t)}}\cdot \div(s_{1}-ts_{0}))(\bfU)\\=\mu_{0} \,
\Res_{\wt{\bfe}}(\{(1:t)\}\times V)(\bfU)=\mu \, \Res_{\bfone}(V)(\bfU).
\end{multline}

It remains to show that $\mu\in k^{\times}$. For $\tau\in \ov k$,
$\div(s_{1}-\tau s_{0})$ intersects $\cV$ properly since the projection $\pi:\cV\to \P^{1}_{k}$ is
surjective (Lemma~\ref{lemm:2}). Proposition~\ref{especializacion} then implies that
$\Res_{\bfe,\wt{\bfe}}(\cV) ((-\tau,1),\bfU)\ne 0$. Furthermore,
$\Res_{\bfone}(V)(\bfu)\Big|_{t=\tau}\ne 0$ as this resultant is a
primitive polynomial in $k[t][\bfu]$. Specializing \eqref{eq:37} at
$t=\tau$, we deduce that $\mu(\tau)\ne 0$ for all $\tau\in\ov k$.
Hence, $\mu\in k^{\times}$ and \eqref{item:27} follows for
$\lambda=\mu^{-1}$.

\medskip Let $\bfv=\{v_{0},v_{1}\}$ be a group of
variables. By \eqref{item:27}, for any $0\le i\le r$,
\begin{displaymath}
\deg_{\bfu_{i}}\Big( \Res_{\bfone}(V)(\bfu)\Big)=
\deg_{\bfu_{i}}\Big( \Res_{\bfe,\wt{\bfe}}(\cV)((-t,1),\bfU)\Big) =
\deg_{\bfu_{i}}\Big( \Res_{\bfe,\wt{\bfe}}(\cV)(\bfv,\bfu)\Big),
\end{displaymath}
since $ \Res_{\bfe,\wt{\bfe}}(\cV)$ is homogeneous in the
variables $\bfv$. Proposition \ref{prop:4} then implies that $\deg(V)= \deg_{1,r}(\cV)$.

Applying an argument similar to the one above, we verify that  $s_0$ intersects
$\cV$ properly. Hence, $\Res_{\bfe,\wt{\bfe}}(\cV)((1,0),\bfu)=
\lambda \, \Res_{\wt{\bfe}}(\cV\cdot \div(s_0))(\bfu)$ with $\lambda\in
k^{\times}$. In particular,  this specialization is not zero, and so
the degrees of $\Res_{\bfe,\wt{\bfe}}(\cV)$ in $v_0$ and in  $\bfv$
coincide. Therefore,
\begin{equation*}\h(V)=
\deg_{t}\Big( \Res_{\bfone}(V)(\bfu)\Big)= \deg_{t}\Big(
\Res_{\bfe,\wt{\bfe}}(\cV)((-t,1),\bfU)\Big) = \deg_{\bfv}\Big(
\Res_{\bfe,\wt{\bfe}}(\cV)(\bfv,\bfu)\Big).
\end{equation*}
By Proposition \ref{prop:4}, we conclude that $\h(V)=
\deg_{0,r+1}(\cV)$.
\end{proof}

Next proposition collects some basic properties of mixed heights and
classes in the extended Chow ring of multiprojective cycles over
$k(\bft)$ with $\bft= \{t_{1},\dots, t_{p}\}$.

\begin{proposition}\label{prop:9} \
\begin{enumerate}
\item \label{item:50}   Let $X\in
  Z_{r}^{+}\Big(\P_{_{k(\bft)}}^{\bfn}\Big)$. Then $[
  X]_{_{k[\bft]}} \ge0$. In particular,  $\h_{\bfc}(X)\ge
  0$ for all $\bfc\in \N_{r+1}^m$.
\smallskip
\item \label{item:18}
Let $X\in Z_{r}(\P^{\bfn}_k)$. Then $ [ X_{_{k(\bft)}}]_{_{k[\bft]}}
= \imath([X_{_{k(\bft)}}])$ or equivalently, $\h_{\bfc}(X)=0$ for all $\bfc\in \N_{r+1}^m$. In
particular, $[ \P^\bfn_{_{k(\bft)}}]_{_{k[\bft]}} =1$.
\smallskip
\item \label{item:49} Let $\bfxi=(\bfxi_{1},\dots, \bfxi_{m})\in \P^\bfn$ be a
  point with coordinates in $k(\bft)$  and  for each $1\le i\le m$
  write  $\bfxi_{i}=(\xi_{i,j})_{j}$ for
  coprime polynomials $\xi_{i,j}\in k[\bft]$. Then
$$
[\bfxi]_{_{k[\bft]}}  = \sum_{i=1}^{m}\h(\bfxi_{i}) \,
\eta\, \bftheta^{\bfn-\bfe_{i}}  +  \bftheta^{\bfn},
$$
with $\h(\bfxi_i):=\max_j\h(\xi_{i,j})=\max_j\deg(\xi_{i,j})$.
In particular,   \begin{math}    \h_{\bfe_{i}}(\bfxi) = \h(\bfxi_{i}).   \end{math}
\smallskip
\item \label{item:15}
Let $D \in \Div^{+}\Big(\P^{\bfn}_{_{k(\bft)}}\Big)$ and $f_{D}\in
k[\bft][\bfx]$ its primitive defining polynomial. Then
\begin{equation*}
   [D]_{_{k[\bft]}} =  \h(f_{D})\, \eta  + \sum_{i=1}^{m}\deg_{\bfx_i}(f_{D})\,
   \theta_i .\end{equation*}
   In particular, $\h_{\bfn}(D)=\h(f_{D})$.
\end{enumerate}
\end{proposition}

\begin{proof}\

\eqref{item:50} This is immediate from the definition of mixed
heights and classes in the extended Chow ring.

\smallskip
\eqref{item:18} By Proposition~\ref{prop:6-1}, there exists
$\lambda\in k(\bft)^\times$ such that
$\Res_{\bfe(\bfc)}(X_{_{k(\bft)}})=\lambda\,\Res_{\bfe(\bfc)}(X)$.
Since the term in the left-hand side is primitive with respect to
$k[\bft]$, we deduce that $\lambda\in k^\times$. Hence, this
resultant does not depend on $\bft$ and $\h_{\bfc}(X)=0$ for all
$\bfc$. This readily implies that  $ [ X_{_{k(\bft)}}]_{_{k[\bft]}} =
\imath([X_{_{k(\bft)}}])$. The rest of the  statement follows from
Proposition~\ref{prop:5}\eqref{item:11}.

\smallskip
\eqref{item:49} By Proposition~\ref{prop:5}\eqref{item:12}, it is enough to show
that $\h_{\bfe_i}(\bfxi)=\h(\bfxi_i)$. By Corollary~\ref{cor:2},
there exists $\lambda_i\in k(\bft)^\times$ such that
$\Res_{\bfe_i}(\bfxi)=\lambda_i \, L_i(\bfxi_i)$, where
 $L_{i}$ is the general linear form of multidegree~$\bfe_{i}$.
Indeed, $\lambda_i\in k^\times$ since $\Res_{\bfe_i}(\bfxi)$ and $L_i(\bfxi_i)$ are primitive polynomials in $k[\bft]$.
Hence,
 $\h_{\bfe_i}(\bfxi)=\deg_\bft(L_i(\bfxi_i))=\h(\bfxi_i)$.

\smallskip
\eqref{item:15} By Proposition~\ref{prop:5}\eqref{gradosparciales}, it is enough to show that $\h_{\bfn}(D)=\h(f_D)$.
Set $\bfd=\bfdeg(f_{D})$.
Consider the
general form $F$  of multidegree $\bfd$ and the
general linear forms
$\bfL=(L_{0},\dots, L_{|\bfn|-1})$
corresponding to  $\bfe(\bfn)$.
Write for short $R=\Res_{\bfe(\bfn),\bfd}(\P^{\bfn}_{k})$.
Using propositions \ref{especializacion} and \ref{prop:6-1}, we
deduce that there exists
$\lambda\in k(\bft)^{\times}$ such that
\begin{equation*}
\lambda \,   \Res_{\bfe(\bfn)}(D)(\bfL)=  R(\bfL,f_{D}).
\end{equation*}
Observe that $\lambda\in
k[\bft]\setminus \{0\}$ because the resultant in the left-hand side
is a primitive polynomial.
On the other hand, let $H\subset \A^{p}_{k}$ be an irreducible  $k$-hypersurface. The
fact that $f_{D}$ is primitive implies that there exists $\bfxi\in H$
such that $f_{D}(\bfxi,\bfx)\ne 0$.
Applying Proposition \ref{especializacion} to the cycle $\P^{\bfn}_{k}$ and
$f_{D}(\bfxi,\bfx)\in k[\bfx]$ we deduce
\begin{equation*}
R(\bfL,f_{D}(\bfxi,\bfx))\ne0.
\end{equation*}
Hence, $\lambda(\bfxi)\ne 0$. This implies that $V(\lambda)$ contains
no hypersurface of $\A^{p}_{k}$ and so  $\lambda\in k^{\times}$. Hence,
\begin{displaymath}
\h_{\bfn}(D)=\deg_{\bft}\Big( \Res_{\bfe(\bfn)}(D) \Big) =
\deg_{\bft}\Big( R(\bfL,f_{D}) \Big)
=\deg_{\bft}(f_{D}) \deg_{\bfU_{n}}(R)= \h(f_{D}),
\end{displaymath}
since $ \deg_{\bft}(f_{D})=\h(f_{D})$ and, by propositions
\ref{prop:4} and \ref{prop:5}\eqref{item:11}, $ \deg_{\bfU_{n}}(R)= 1$.
\end{proof}

We recall some notions and properties of
valuations of fields. Let $T_{0},\dots, T_{p}$ denote the standard
homogeneous coordinates of $\P^{p}_{k}$. If we identify each
variable $t_{i}$ with the rational function $T_{i}/T_{0}$, we can
regard  $k(\bft)$ as the field of rational functions of this
projective space: given  $\alpha\in k(\bft)$, then
$\alpha(T_1/T_0,\dots,T_p/T_0)$ is homogeneous of degree 0 and
defines a rational function on $\P^{p}_{k}$.

Given an irreducible $k$-hypersurface $H$ of $\P^{p}$ and $\alpha\in
k(\bft)$, we set $\ord_{H}(\alpha)$ for the order of vanishing of
$\alpha$ along $H$. The map $\ord_{H}:k(\bft)\to \Z$ is a valuation
of $k(\bft)$. If $H_{\infty}:=V(T_{0})$ is the  hyperplane at
infinity, then $\ord_{H_\infty}(\alpha)=-\deg(\alpha)$ where
$\deg(\alpha)=\deg(\alpha_{1})-\deg(\alpha_{2})$ for any
$\alpha_{i}\in k[\bft]$ such that
$\alpha=\alpha_1/\alpha_{2}$. If~$H\ne H_{\infty}$, then
$\ord_{H}(\alpha)$  coincides with the order of the polynomial
$f_{H}(1,t_{1},\dots, t_{p})$ in the factorization of $\alpha$,
where $f_{H}\in k[T_{0},\dots, T_{p}]$ is the defining polynomial of
$H$.

Let $K$ be an arbitrary field and $v$ a valuation of $K$.
For a polynomial  $f$ over $K$, we set $v(f)$ for the
minimum of the valuation of its coefficients. Gauss Lemma
says that for any given polynomials $f,g$ over $K$,
\begin{displaymath}
  v(f g)= v(f)+v(g).
\end{displaymath}
Given a finite extension $E$ of $K$, there exists a
(non-necessarily unique) valuation $w$ of $E$ extending $v$.

\medskip
The height of an arbitrary resultant can be expressed in terms of mixed heights:

\begin{lemma}
  \label{lemm:3}
Let $X\in Z_{r}^{+}\Big(\P^\bfn_{_{k(\bft)}}\Big)$ and $\bfd\in
(\N^{m}\setminus \{{\bf0}\})^{r+1}$. Then
\begin{displaymath}
  \deg_{\bft}(\Res_{\bfd}(X))= \coeff_{\eta\bftheta^{\bfn}}\bigg(
  [X]_{_{k[\bft]}}\,
  \prod_{i=0}^{r} \sum_{j=1}^{m}d_{i,j} \theta_{j}\bigg).
\end{displaymath}
In particular, for $\bfb\in \N^{m}_{r}$ and $\bfd_{r}\in
\N^{m}\setminus \{\bf0\}$,
\begin{equation}\label{eq:2}
  \deg_{\bft}(\Res_{\bfe(\bfb),\bfd_{r}}(X))=
\sum_{j=1}^{m}d_{r,j} \h_{\bfb+\bfe_{j}}(X).
\end{equation}
\end{lemma}

\begin{proof}
Write $\bfd=(\bfd_{0},\dots, \bfd_{r})$ for
$\bfd_{i}\in \N^{m}\setminus \{\bf0\}$.
We claim that
\begin{equation}\label{eq:27}
  \deg_{\bft}(\Res_{\bfd_{0},\dots, \bfd_{r}}(X))=
\sum_{j=1}^{m}d_{0,j} \deg_{\bft}(\Res_{\bfe_{j},\bfd_{1},\dots, \bfd_{r}}(X)).
\end{equation}
 We first consider the case $r=0$.
Let $X_{_{\ov{k(\bft)}}}= \sum_{\bfxi} m_{\bfxi}\, \bfxi$ with
 $\bfxi\in \P^{\bfn}$ and  $m_{\bfxi}\ge1$.
For each of these points write $\bfxi=(\xi_{i,j})_{i,j}$ for some $\xi_{i,j}\in \ov{k(\bft)}$.
Let  $F_{0}$  and $L_{j}$ denote the general forms of degree
$\bfd_{0}$ and $\bfe_{j}$ as in \eqref{eq:15}, respectively.
By Proposition~\ref{prop:6-1} and
Corollary~\ref{cor:2},
there exist $\mu, \lambda_j\in \ov{k(\bft)}^\times$ such that
\begin{displaymath}
  \Res_{\bfd_{0}}(X)(F_{0})= \mu\,\prod_{\bfxi}F_{0}(\bfxi)^{m_{\bfxi}}
\quad , \quad
  \Res_{\bfe_{j}}(X)(L_{j})= \lambda_j\,
  \prod_{\bfxi}L_{j}(\bfxi_{j})^{m_{\bfxi}} .
\end{displaymath}
Picking a suitable choice of multihomogeneous coordinates for the
$\bfxi$'s, {it is possible to set} $\lambda_{j}=1$ for all $j$.
We have
\begin{displaymath}
  \Res_{\bfd_{0}}(X)\Bigg(\prod_{j=1}^{m}L_{j}^{d_{0,j}}\Bigg) =
  \mu\,\prod_{\bfxi}\prod_{j=1}^{m}L_{j}(\bfxi)^{d_{0,j}m_{\bfxi}}
=
  \mu\,\prod_{j=1}^{m}\Res_{\bfe_{j}}(X)(L_{j})^{d_{0,j}}.
\end{displaymath}
We deduce that $\mu\in k(\bft)^{\times}$ since all considered
resultants have coefficients in $k[\bft]$. Now let $E$ be a
sufficiently large finite extension of $k(\bft)$ containing all  the
chosen coordinates $\xi_{i,j}$. For any valuation $v$ of $E$,
\begin{equation} \label{eq:16}
  v(F_{0}(\bfxi))
=\min_{|\bfa|={\bfd_{0}}} v(\bfxi^\bfa)
= \sum_{j=1}^{m}d_{0,j}\, \min_{\ell}
  v(\xi_{j,\ell})
= \sum_{j=1}^{m}d_{0,j}v(L_{j}(\bfxi_j)).
\end{equation}
Let  $H$ be an irreducible $k$-hypersurface of $\P^{p}$ different
from $H_{\infty}$, and $v_{H}$ a valuation of $E$ extending $\ord_{H}$.  Since $\Res_{\bfd_{0}}(X)$ and
$\Res_{\bfe_{j}}(X)$ are primitive with respect to $k[\bft]$, we
have that $v_{H}(\Res_{\bfd_{0}}(X))=v_{H}(\Res_{\bfe_{j}}(X))=0$.
By Gauss Lemma,
\begin{align*}
\ord_{H}(\mu)&= -
\ord_{H}\bigg(\prod_{\bfxi}F_{0}(\bfxi)^{m_{\bfxi}}\bigg) =
-\sum_{\bfxi}m_{\bfxi}\, v_{H}(F_{0}(\bfxi)),
\\
0&=\ord_{H}\bigg(\prod_{\bfxi}L_{j}(\bfxi_j)^{m_{\bfxi}}\bigg) =
\sum_{\bfxi}m_{\bfxi}\, v_{H}(L_{j}(\bfxi_j)).
\end{align*}
Applying \eqref{eq:16},
 \begin{displaymath}
\ord_{H}(\mu)=-\sum_{\bfxi}m_{\bfxi}v_{H}(F_{0}(\bfxi))=
-\sum_{j=1}^m d_{0,j}\,\bigg( \sum_{\bfxi}m_{\bfxi}\,
v_{H}(L_{j}(\bfxi_j))\bigg)=0.
\end{displaymath}
Since this holds for every $H\ne
H_{\infty}$, we deduce that $\mu\in k^\times$.
Now let $v_{\infty}$ be a valuation of $E$ extending
$\ord_{H_{\infty}}$. Applying again Gauss Lemma and \eqref{eq:16},
\begin{multline*}
\ord_{H_{\infty}}(\Res_{\bfd_{0}}(X)) =
\sum_{\bfxi}  m_{\bfxi} \, v_{\infty}(F_{0}(\bfxi)) \\
=\sum_{j=1}^{m}d_{0,j} \bigg( \sum_{\bfxi}  m_{\bfxi} \,
v_{\infty}(L_{j}(\bfxi_{j})) \bigg) = \sum_{j=1}^{m} d_{0,j} \,
\ord_{H_{\infty}}(\Res_{\bfe_{j}}(X)).
\end{multline*}
Hence
$\deg_{\bft}(\Res_{\bfd_{0}}(X))=\sum_{j=1}^{m}d_{0,j} \,
\deg_{\bft}(\Res_{\bfe_{j}}(X))$, which concludes the case $r=0$.

\medskip The case $r\ge1$ follows by reduction to the zero-dimensional
case.  For $1\le i\le r$, let~$F_{i}$ be the generic polynomial of
multidegree $\bfd_i$ and $\bfU_{i}$ the {variables} corresponding to
its coefficients. Set $\wt k=k\big(\bfU_1, \dots, \bfU_{r}\big)$.
Observe that $\div(F_j)$ intersects properly the cycle $X_{\wt
  k(\bft)}\cdot \prod_{i=1}^{j-1}\div(F_i)$, $1\le j\le r$. Therefore,
$X_{\wt k(\bft)}\cdot \prod_{i=1}^{r}\div(F_i)$ is a cycle of
dimension~0 and, by Proposition \ref{especializacion}, there exists
$\lambda\in \wt k(\bft)^\times$ such that
\begin{displaymath}
 \Res_{\bfd}(X)=\lambda\,\Res_{\bfd_{0}}\bigg(X_{\wt k} \cdot
 \prod_{i=1}^{r}\div(F_{i})\bigg),
\end{displaymath}
Since both resultants are primitive  with respect to $\wt k[\bft]$,
 $\lambda\in \wt k^{\times}$ and so these resultants have the same $\bft$-degree.
Analogous  relations hold  for
 $\Res_{\bfe_{j}, \bfd_{1},\dots, \bfd_{r}}(X)$, $1\le j\le m$.
Hence, \eqref{eq:27} follows by applying the previously
considered zero-dimensional case.

\medskip From Proposition \ref{prop:6-item:27}, we deduce that the map
\begin{displaymath}
  (\N^{m}\setminus \{{\bf0}\})^{r+1}\longrightarrow \Z\quad, \quad \bfd\longmapsto \deg_{\bft}(\Res_{\bfd}(X))
\end{displaymath}
 is multilinear  with respect
to the variables $\bfd_{0},\dots, \bfd_{r}$.
The same holds for the map $\bfd\mapsto \coeff_{\eta\bftheta^{\bfn}}\Big(
  [X]_{_{k[\bft]}}\,
  \prod_{i=0}^{r} \sum_{j=1}^{m}d_{i,j} \theta_{j}\Big)$.
Both maps coincide when $\bfd= \bfe(\bfc)$ for $\bfc\in \N^{m}_{r+1}$ because
\begin{displaymath}
  \deg_{\bft}(\Res_{\bfe(\bfc)}(X))= \h_{\bfc}(X)= \coeff_{\eta\,
    \bftheta^{\bfn}}(\bftheta^{\bfc} \, [X\ck).
\end{displaymath}
Since the family $\{\bfe(\bfc)\}_{\bfc}$ is
a basis of the semigroup $(\N^{m}\setminus\{{\bf0}\})^{r+1}$,
both maps coincide for all $\bfd$.
This completes the proof of the statement.
\end{proof}

The following is an arithmetic version of B\'ezout's theorem for multiprojective cycles over $k(\bft)$.

\begin{theorem}\label{bezaritff}
  Let $X\in Z_{r}\Big(\P^{\bfn}_{_{k(\bft)}}\Big)$ and $f\in k[\bft][\bfx_{1},\dots,\bfx_{m}]$ {a
polynomial, multihomogeneous in  the variables $\bfx$, such that}
$X$ and   $\div(f)$   intersect properly.
\begin{enumerate} \item \label{item:24} If $X$ is effective,
then  for any $\bfb\in \N_{r}^{m}$
\begin{displaymath}\h_\bfb(X\cdot \div(f)) \le \h(f)\deg_\bfb(X)
 + \sum_{i=1}^m\deg_{\bfx_i}(f)  \h_{\bfb+\bfe_i}(X).
 \end{displaymath}
\smallskip
\item \label{item:25}   $\h_\bfc(X)=0$ for any $\bfc\in \N_{r+1}^m$ such that
 $c_i>n_i$ for some $i$.
 \smallskip
\item \label{item:23} If $X$ is effective, then  $[ X\cdot\div(f)]_{_{k[\bft]}} \le [ X]_{_{k[\bft]}}\cdot [ \div(f)]_{_{k[\bft]}}$.
\end{enumerate}
\end{theorem}

\begin{proof}\

\eqref{item:24}
Set $\bfd=\bfdeg(f)$ and let $\bfU=(\bfU_{0},\dots, \bfU_{r-1},\bfU_{r})$ be the groups of variables
corresponding to $(\bfe(\bfb),\bfd)$.
By Proposition \ref{especializacion},
there exists $\lambda\in k(\bft)^{\times}$ such that
\begin{displaymath}
  \Res_{\bfe(\bfb),\bfd}(X) (\bfU_{0},\dots,\bfU_{{r-1}},f)=
  \lambda \, \Res_{\bfe(\bfb)}(X\cdot \div(f))(\bfU_{0},\dots,\bfU_{{r-1}}).
\end{displaymath}
Indeed,
$\lambda\in k[\bft]\setminus\{0\}$ because the resultant in the
right-hand side is primitive. Hence,
\begin{align*} \deg_{\bft}\Big(\Res_{\bfe(\bfb)}(X\cdot \div(f))\Big)
&\le \deg_{\bft}\Big(  \Res_{\bfe(\bfb),\bfd}(X)
(\bfU_{0},\dots,\bfU_{{r-1}},f)\Big)\\[1mm]
&\le \deg_{\bft}(f)\,
\deg_{\bfU_{r}}(\Res_{\bfe(\bfb),\bfd}(X))+\deg_{\bft}(
\Res_{\bfe(\bfb),\bfd}(X)).
\end{align*}
From the definition of the height and Proposition~\ref{prop:4}, we
deduce that \begin{math}
  \h_\bfb(X\cdot \div(f))\le \h(f)\deg_\bfb(V)+\deg_{\bft}(
\Res_{\bfe(\bfb),\bfd}(X)).
\end{math}
The statement  follows then from \eqref{eq:2}.

\smallskip
\eqref{item:25} It is enough to prove the statement for a
$k(\bft)$-variety $V$. We first consider the case when $V$ is
determined by a regular sequence. We proceed by induction on its
codimension. Let $f_{j}$, $1\le j\le |\bfn|-r$, be a regular
sequence of multihomogeneous polynomials. For $r\le \ell\le |\bfn|$,
set
$$
Y_{\ell}=\prod_{j=1}^{|\bfn|-\ell}\div(f_{j}) \in
Z_{\ell}^{+}\Big(\P^{\bfn}_{_{k(\bft)}}\Big).
$$
For $r=|\bfn|$, we have $Y_{|\bfn|}=\P^\bfn$ and
Proposition~\ref{prop:9}\eqref{item:18} implies that
$\h_\bfc(Y_{|\bfn|})=0$  for all $\bfc\in \N_{|\bfn|+1}$.
Suppose now that $r<|\bfn|$ and that the statement holds for $Y_{r+1}$.
Let $\bfc\in
\N_{r+1}^m$ such that  $c_i>n_i$ for some $i$.
By Proposition~\ref{prop:9}\eqref{item:50} and item~\eqref{item:24} above,
$$
0\le \h_\bfc(Y_{r}) \le \h(f_{|\bfn|-r})\deg_{\bfc}(Y_{{r+1}})+
\sum_{j=1}^m \deg_{\bfx_j}(f_{|\bfn|-r}) \h_{\bfc+\bfe_j}(Y_{{r+1}}).
$$
The inductive
hypothesis together with the fact that $\deg_\bfc(Y_{r+1})=0$ imply that the right-hand side of this inequality
vanishes, and
hence $\h_\bfc(Y_{r})=0$.

For the general case, consider $|\bfn|-r$ generic
linear combinations of a system of generators of  $I(V)$. The obtained
polynomials form a regular sequence and define a variety $Y_{r}$
such that $Y_r-V$ is effective.  By the previous analysis,  $0\le
\h_\bfc(V)\le \h_\bfc(Y_{r})=0$, {hence} $ \h_\bfc(V)=0$ as stated.

\smallskip
\eqref{item:23} This is a direct consequence of \eqref{item:24} and
\eqref{item:25} together with Theorem~\ref{bezgeom} and
Proposition~\ref{prop:9}\eqref{item:15}.
\end{proof}

\begin{corollary} \label{cor:1}
  Let $V\subset \P^{\bfn}_{_{k(\bft)}}$ be a $k(\bft)$-variety of pure
  dimension $r$ and $f\in k[\bft][\bfx_{1},\dots,\bfx_{m}]$ a
multihomogeneous polynomial. Let $W$ denote the union of
the components of dimension $r-1$ of the intersection $V\cap V(f)$.
Then
\begin{displaymath}
[ W]_{_{k[\bft]}} \le [ V]_{_{k[\bft]}}\cdot [ V(f)]_{_{k[\bft]}}.
\end{displaymath}
In particular,
\begin{math}\h_\bfb(W) \le \h(f)\deg_\bfb(V)
 + \sum_{i=1}^m\deg_{\bfx_i}(f)  \h_{\bfb+\bfe_i}(V)
 \end{math} for all $\bfb\in \N_{r}^{m}$.
\end{corollary}

\begin{proof}
  Let $V'\subset \P^{\bfn}_{_{k(\bft)}}$ be the union of the
  components of $V$ not contained in $|\div(f)|$. Then
  $W\subset V'\cap V(f)$ and $V(f)$ intersects $V'$
  properly. By  Theorem \ref{bezaritff}\eqref{item:23},
  \begin{displaymath}
    [W]_{_{k(\bft)}}\le [ V'\cdot V(f)]_{_{k(\bft)}}
\le  [ V']_{_{k(\bft}}\cdot[V(f)]_{_{k(\bft)}}\le [
V]_{_{k(\bft}}\cdot[V(f)]_{_{k(\bft)}}.
  \end{displaymath}
The last statement follows from this inequality  when $\bfb\le \bfn$
and from Theorem \ref{bezaritff}\eqref{item:25} otherwise.
\end{proof}

\begin{corollary}\label{desffcase}
Let $n\in\Z_{>0}$, $X\in Z_r^+\Big(\P^n_{_{k(\bft)}}\Big)$ and
$f_{j}\in k[\bft][x_{0},\dots, x_{n}]\setminus k[\bft]$ a family of
$s\le r$ polynomials homogeneous in the variables
$x_{0},\dots,x_{n}$ such that
 $ X\cdot \prod_{j=1}^{i-1}\div(f_{j})$ and $\div(f_{i})$ intersect
properly  for  $1\le i\le s$. Then
 \begin{equation*}
 \h\bigg(X\cdot \prod_{j=1}^{s}\div(f_{j})\bigg)\leq \bigg(\prod_{j=1}^s\deg_\bfx(f_{j})\bigg)\,\bigg(\h(X)+\deg(X)
\bigg(\sum_{\ell=1}^s\frac{\deg_{\bft}(f_{\ell})}{\deg_\bfx(f_{\ell})}\bigg)\bigg).
\end{equation*}
\end{corollary}

\begin{proof}
Set  $Y=X\cdot \prod_{j=1}^{s}\div(f_{j})$.
We have  $$[Y]_{_{k[\bft]}}= \h(Y)\,
\eta\, \theta^{n-r+s-1}+\deg(Y)\, \theta^{n-r+s} \ , \
[X]_{_{k[\bft]}}=\Big(\h(X)\, \eta\, \theta^{n-r-1} +  \deg(X)\,
\theta^{n-r}\Big).$$
Applying recursively
Theorem~\ref{bezaritff}\eqref{item:23},
$$
[Y]_{_{k[\bft]}} \le  [X]_{_{k[\bft]}}\cdot \prod_{j=1}^s
[\div(f_j)]_{_{k[\bft]}}.
$$
The statement follows by comparing the coefficients  corresponding to
the monomial $\eta \, \theta^{n-r+s-1}$ in the above inequality.
\end{proof}

Next  result shows that, for projective $k(t)$-cycles, the
inequality in Theorem~\ref{bezaritff}\eqref{item:24} is an equality
in the generic case.

\begin{proposition}\label{lemm:15}
Let $X\in Z_{r}(\P^{n}_{_{k(t)}})$ with $k$ an algebraically closed field,
  $t$ a single variable and  $n,r\ge 1$. Let $\cX \subset
  \P_{k}^{1}\times\P_{k}^{n}$ be the standard model of $X$ and $\ell\in k[x_{0},\dots,
  x_{n}]$ a generic linear form. Then
\begin{enumerate}
\item \label{item:44} $\cX\cdot \div(\ell)$ is the standard model of
  $X\cdot \div (\ell)$;
\smallskip
\item \label{item:52} $\deg(X\cdot \div(\ell))= \deg(X)$ and $\h(X\cdot \div(\ell))= \h(X)$.
\end{enumerate}
\end{proposition}

\begin{proof} \

  \eqref{item:44} By Lemma \ref{lemm:16}, it is enough to prove the
  claim for an irreducible $k(t)$-variety $V$ of dimension $r\ge
  1$. Let $\cV\subset \P_{k}^{1}\times \P_{k}^{n}$ be the standard
  model of $V$. By Lemma \ref{lemm:2}\eqref{item:17}, it suffices to
  prove that $V\cdot \div(\ell)$ is the generic fiber of $\cV\cdot
  \div(\ell)$ and that the support of $\cV\cdot \div(\ell)$
  has no vertical components.
Consider the projection $\varpi:\cV\to \P_{k}^{n}$. We have
\begin{displaymath}
  1\le r=\dim(\cV)-\dim(\P_{k}^{1})\le \dim(\varpi(\cV))\le \dim(\cV)=r+1.
\end{displaymath}
If $\dim(\varpi(\cV))=1$, then $\dim(V)=1$ and
$\varpi^{-1}(\bfxi)= \P_{k}^{1}\times \{\bfxi\}$ for all $\bfxi\in
\varpi(\cV)$ because of the theorem of dimension of fibers. Hence,
in this case,
\begin{displaymath}
  \cV= \P^{1}_{k}\times \varpi(\cV).
\end{displaymath}
Therefore, \begin{math} \cV\cdot \div(\ell) = \P^{1}_{k}\times
(\varpi(\cV)\cdot \div(\ell))
\end{math} has no vertical components. Moreover, by considering the generic fiber of $\pi:\cV\to
\P_{k}^{1}$, it follows from Lemma \ref{lemm:2}\eqref{item:16} that
$V= \varpi(\cV)_{_{k(t)}}$. Hence, the generic fiber of
$\cV\cdot \div(\ell)$ over $\P^{1}_{k}$ coincides with $V\cdot
\div(\ell)$, which proves the claim for $\dim(\varpi(\cV))= 1$.

If $\dim(\varpi(\cV))\ge 2$, $\cV\cap V(\ell)$ is an irreducible
$k$-variety of dimension $r$ by \cite[Thm.~6.3(4)]{Jou83}. Moreover,
the induced projection $\pi_{\ell}:\cV\cdot \div(\ell)\to \P^{1}_{k}$
is surjective. Indeed, for $\bftau\in\P^{1}(k)$, we have that $\cV\cap
\pi^{-1}(\bftau)$ is a projective variety of dimension $r\ge 1$ and so
\begin{displaymath}
\Big(\cV\cap V(\ell) \Big)\cap  \pi^{-1}(\bftau)=
\Big(\cV\cap \pi^{-1}(\bftau)\Big) \cap V(\ell) \ne \emptyset.
\end{displaymath}

Hence, the projection $\pi_{\ell}$ has no vertical fibers, as it is
surjective and the support of~$\cV\cdot \div(\ell)$ is irreducible.
By going through the ideals of definition, we can verify that its
generic fiber coincides with $V\cdot \div(\ell)$, which completes the
proof of the claim.

\smallskip
\eqref{item:52}
The statement concerning the degree follows from Theorem
  \ref{bezgeom}.  For the height,
\begin{displaymath}
\h(X)= \deg_{0,r+1}(\cX)= \deg_{0,r}(\cX\cdot \div(\ell))= \h(X\cdot
\div(\ell)).
\end{displaymath}
The first and third equalities follow from Proposition \ref{lemm:12}
and the second one follows from Theorem \ref{bezgeom}.
\end{proof}

We next show that mixed heights are monotonic with
respect to linear projections. We keep the notation from
Proposition~\ref{113}. In particular, we consider the
linear projection $ \pi: \P^\bfn_K\dashrightarrow\P^\bfl_K $
in \eqref{eq:58} and the inclusion $\jmath: A^{*}(\P^\bfl;k[\bft])\hookrightarrow
A^{*}(\P^\bfn;k[\bft])$ defined by $\jmath(P)= \bftheta^{\bfn-\bfl}P$.

\begin{proposition}\label{113alturasff}
  Let $\pi:\P^{\bfn}_{_{k(\bft)}}\dashrightarrow \P^{\bfl}_{_{k(\bft)}}$ be the
  linear projection defined in \eqref{eq:58}  and
  $X\in Z^+_{r}\Big(  \P^\bfn_{_{k(\bft)}}\Big)$. Then
$$
\jmath \Big([\pi_{*}X]_{_{k[\bft]}}\Big)\le [X]_{_{k[\bft]}}.
$$
In particular, $\h_{\bfc}(\pi_{*}X) \le \h_{\bfc}(X)$ for all
$\bfc\in \N^{m}_{r+1}$.
\end{proposition}

\begin{proof}
The statement   is equivalent to the inequalities
$\deg_{\bfb}(\pi_{*}X) \le \deg_{\bfb}(X)$ and  $\h_{\bfc}(\pi_{*}X) \le \h_{\bfc}(X)$ for all
$\bfb,\bfc$.
Because of  Proposition~\ref{113},
 we
only need to  prove the latter.
Let  $\bfc\in \N_{r+1}^m$. By Proposition~\ref{formainicialproy},
$\Res_{\bfe(\bfc)}(\pi_*X)$ divides $\init_\prec(\Res_{\bfe(\bfc)}(X))$ in
$k[\bft][\bfU]$.
We deduce that
\begin{displaymath}
\h_\bfc(\pi_*X)=\deg_\bft(\Res_{\bfe(\bfc)}(\pi_*X))
\le \deg_\bft(\Res_{\bfe(\bfc)}(X)) = \h_\bfc(X),
\end{displaymath}
which proves the statement.
\end{proof}

The following result gives the behavior of extended Chow rings and classes
with respect to products.

\begin{proposition}\label{altrodgenff}
Let $m_i\in \Z_{>0}$ and $\bfn_{i}\in \N^{m_{i}}$ for $i=1,2$. Then
\begin{enumerate}
\item \label{item:19} $A^{*}(\P^{\bfn_{1}}\times \P^{\bfn_{2}};k[\bft]) \simeq
  A^{*}(\P^{\bfn_{1}};k[\bft]) \otimes_{\Z[\eta]} A^{*}(\P^{\bfn_{2}};k[\bft])$.
  \smallskip
\item \label{item:22} Let $X_{i}\in
  Z_{r_{i}}^+\Big(\P^{\bfn_{i}}_{_{k(\bft)}}\Big) $ for
  $i=1,2$. {The above isomorphism identifies $[X_{1}\times
    X_{2}]_{_{k[\bft]}}$ with $[X_{1}]_{_{k[\bft]}}\otimes [X_{2}]_{_{k[\bft]}}.$}
In particular, for $\bfc_{i}\in \N^{m_{i}}$ such that
$|\bfc_{1}|+|\bfc_{2}|=r_{1}+r_{2}+1$,
\begin{displaymath}
   \h_{(\bfc_1,\bfc_2)}(X_{1}\times X_{2})=
\left\{\begin{array}{cl}
 \deg_{\bfc_2}(X_{2})\,\h_{\bfc_1}(X_{1})& \text{ if } |\bfc_{1}|=r_{1}+1,
|\bfc_{2}|=r_{2}, \\[1mm]
\deg_{\bfc_1}(X_{1})\,\h_{\bfc_2}(X_{2}) & \text{ if } |\bfc_{1}|=r_{1},
|\bfc_{2}|=r_{2}+1, \\[1mm]
0 & \text{ otherwise.}
\end{array}\right.
\end{displaymath}
\end{enumerate}
\end{proposition}

\begin{proof} \eqref{item:19} is immediate from the definition of the
  extended Chow ring while
  \eqref{item:22} follows directly
from Proposition~\ref{resultantesproductosciclos}.
\end{proof}

Finally, we compute the class in the extended Chow ring  of the ruled join of two projective varieties.
Let $n_{1},n_{2}\in \N$ and consider the $\Z$-linear map
$\jmath_{i}:A^{*}(\P^{n_{i}}; k[\bft])\hookrightarrow A^{*}(\P^{n_{1}+n_{2}+1};k[\bft])$ defined by
$\jmath_{i}(\theta^{l}\eta^{b})=\theta^{l}\eta^{b}$ for $0\le l\le
n_{i}$ and $b=0,1$.

 \begin{proposition} \label{prop:12}
Let $X_{i} \in Z_{r_{i}}\Big(\P^{n_{i}}_{_{k(\bft)}}\Big)$
for $i=1,2$. Then
\begin{displaymath}
  [X_{1}\#X_{2}]_{_{k[\bft]}} =   \jmath_{1}([X_{1}]_{_{k[\bft]}})  \cdot \jmath_{2}([X_{2}]_{_{k[\bft]}}).
\end{displaymath}
In particular,
$\h (X_{1}\#X_{2})= \deg(X_{1})\, \h (X_{2})+\deg(X_{2})\,\h (X_{1})$.
 \end{proposition}

 We need some lemmas for the proof of this result. The first of them
 deals with 0-dimensional cycles.  For $i=1,2$ and  $j=0,1$, let
 $\bfx_{i}=\{x_{i,0},\dots, x_{i,n_{i}}\}$ be the homogeneous
 coordinates of $\P^{n_{i}}_{_{k(\bft)}}$ and
 $\bfu_{j}^{(i)}=\{u_{j,l}^{(i)}\}_{0\le l\le n_{i}}$ a group of
 $n_{i}+1$ variables. Write~$L_{j}^{(i)}$ for the general linear
 form in the variables $\bfx_{i}$.

\begin{lemma}\label{lemm:13}
Let $X_{i}\in Z_{0}(\P^{n_{i}}_{_{k(\bft)}})$ for $i=1,2$. Then
$  \h(X_{1}\# X_{2})= \deg(X_{1})\, \h(X_{2})+ \deg(X_{2})\, \h(X_{1}).$
\end{lemma}

\begin{proof}
  For $i=1,2$, write
\begin{math}
  \Res_{1}(X_{i})= \lambda_{i}\prod_{\bfxi_{i}} L_{0}^{(i)}(\bfxi_{i})^{m_{\bfxi_{i}}}
\end{math}
with $\lambda_{i}\in \ov{k(\bft)}^{\times}$, $\bfxi_{i}\in
\P^{n_{i}}(\ov{k(\bft)})$ and $m_{\bfxi_{i}} \in \Z$.  We claim that
there exists $\nu\in k(\bft)^{\times}$ such that
\begin{displaymath}
  \Res_{1,1}(X_1\#X_2)=
  \nu \, \lambda_{1}^{\deg(X_{2})}\, \lambda_{2}^{\deg(X_{1})}\, \prod_{\bfxi_{1},\bfxi_{2}} \Big(L_{0}^{(1)}(\bfxi_{1})L_{1}^{(2)}(\bfxi_{2})
  - L_{1}^{(1)}(\bfxi_{1})L_{0}^{(2)}(\bfxi_{2})\Big)^{m_{\bfxi_{1}}m_{\bfxi_{2}}}.
\end{displaymath}
Indeed, for each $\bfxi_{1}\in \P^{n_{1}}(\ov{k(\bft)})$ and $\bfxi_{2}\in \P^{n_{2}}(\ov{k(\bft)})$, Proposition
\ref{prop:7}\eqref{item:20} and Corollary~\ref{cor:3} imply that
\begin{math}
  \Res_{1,1}(\bfxi_1\#\bfxi_2)
\end{math} is equal to $L_{0}^{(1)}(\bfxi_{1})L_{1}^{(2)}(\bfxi_{2})
- L_{1}^{(1)}(\bfxi_{1})L_{0}^{(2)}(\bfxi_{2})$ up to a
constant in~$\ov{k(\bft)}^{\times}$.
We deduce that there {exists} $\mu\in \ov{k(\bft)}^{\times}$ such that
\begin{equation}\label{eq:36}
  \Res_{1,1}(X_1\#X_2)=
\mu\, \prod_{\bfxi_{1},\bfxi_{2}} \Big(L_{0}^{(1)}(\bfxi_{1})L_{1}^{(2)}(\bfxi_{2})
- L_{1}^{(1)}(\bfxi_{1})L_{0}^{(2)}(\bfxi_{2})\Big)^{m_{\bfxi_{1}}m_{\bfxi_{2}}}.
\end{equation}
By setting $\bfu^{(1)}_{1}=0$ in this equality and comparing its
right-hand side with the explicit expression for $\Res_{1}(X_{i})$
plus the fact that $ \Res_{1,1}(X_1\# X_2)$ has coefficients in
$k(\bft)$, we get that $\mu= \nu \, \lambda_{1}^{\deg(X_{2})}\,
\lambda_{2}^{\deg(X_{1})}$ with $\nu\in k(\bft)^{\times}$.

Let $E$ be a sufficiently large extension of $k(\bft)$, $H\subset
\P^{p}_{k}$ a hypersurface and  $v$ a valuation of $E$ extending
$\ord_{H}$. For $i=1,2$ and $\bfxi_i=(\xi_{i,j})_j \in E^{n_i+1}$ set
$v(\bfxi_i)=\min\{v(\xi_{i,j})\}$. Observe that
$$v\Big(L_{0}^{(1)}(\bfxi_{1})L_{1}^{(2)}(\bfxi_{2})
-
L_{1}^{(1)}(\bfxi_{1})L_{0}^{(2)}(\bfxi_{2})\Big)=v(\bfxi_1)+v(\bfxi_2)=v(L_{0}^{(1)}(\bfxi_{1}))+
v(L_{0}^{(2)}(\bfxi_{2})).$$
 Therefore, applying
Gauss' lemma, we obtain that $\ord_{H}(\Res_{1,1}(X_1\#X_2))$ is
equal to
\begin{align*}
&\ord_{H}(\nu) +\deg(X_{2}) v(  \lambda_{1}) +\deg(X_{1}) v(\lambda_{2})+\sum_{\bfxi_{1},\bfxi_{2}}
  m_{\bfxi_{1}}m_{\bfxi_{2}}\Big(v(\bfxi_{1})+v(\bfxi_{2})\Big)\\
=&\ord_{H}(\nu) +\deg(X_{2}) \bigg(v(  \lambda_{1})
+\sum_{\bfxi_{1}}
  m_{\bfxi_{1}}v(\bfxi_{1})\bigg)
+\deg(X_{1}) \bigg(v(  \lambda_{2})+\sum_{\bfxi_{2}}
  m_{\bfxi_{2}}v(\bfxi_{2})\bigg)\\
=&\ \ord_{H}(\nu) +\deg(X_{1})\, \ord_{H}(\Res_{1}
(X_{2}))+\deg(X_{2})\, \ord_{H}(\Res_{1} (X_{1})).
\end{align*}
Let $H_{\infty}$ be the hyperplane at infinity. For  $H\ne
H_{\infty}$, $\ord_{H}(\Res_{1,1}(X_1\#X_2))=\ord_{H}
(\Res_{1}(X_i))=0$ and the identity above implies that
$\ord_{H}(\nu)=0$ in this case. Since this holds for all $H\ne
H_{\infty}$, it follows that $\nu \in k^{\times}$. Hence,
$\ord_{H_{\infty}}(\nu)=0$. For
$H=H_{\infty}$ in the same identity, we get
\begin{displaymath}
\ord_{H_{\infty}}(\Res_{1,1}(X_1\#X_2))= \deg(X_{1})\,
\ord_{H_{\infty}}(\Res_{1} (X_{2}))+\deg(X_{2})
\ord_{H_{\infty}}(\Res_{1} (X_{1})).
\end{displaymath}
This implies the statement.
\end{proof}

\begin{lemma}\label{lemm:14}
  Let $k$ be an algebraically closed field and $t$ a single
  variable. Let $\cV\subset \P^{1}_{k}\times \P^{n_{1}+n_{2}+1}_{k}$
  be the standard model of $V_{1}\# V_{2}$, where each $V_{i}\subset
  \P^{n_{i}}_{_{k(t)}}$ is an irreducible $k(t)$-variety and
  $\dim(V_{2})\ge 1$. Let $\ell\in k[\bfx_{2}]$ be a generic linear
  form. Then $\cV\cdot \div(\ell)$ is the standard model of $V_{1}\#
  (V_{2}\cdot \div(\ell))$.
\end{lemma}

\begin{proof} By Lemma
\ref{lemm:2}\eqref{item:17}, it suffices to prove that  $V_{1}\#
(V_{2}\cdot
  \div(\ell))$ is the generic fiber of $\cV\cdot \div(\ell)$ and that
the support of  the cycle $\cV\cdot \div(\ell)$  has no vertical
components.

 Let $\cI_{i,0}=
  I(V_{i})\cap k[t][\bfx_{i}]$ and $\cI_{0}= \Big( I(V_{1})
  + I(V_{2})\Big) \cap k[t][\bfx_{1},\bfx_{2}]$ be the prime ideals of $\cV_{i,{0}}$ and
  of $\cV_{{0}}$ respectively, following the  notation in Remark~\ref{rem:1}.  Then $\cI_{0}\supset
  \cI_{1,0}+\cI_{2,0}$. In addition, there is an isomorphism of
  $k[t]$-algebras
$$
k[t][\bfx_{1},\bfx_{2}]/ (\cI_{1,0}+\cI_{2,0})\simeq
k[t][\bfx_{1}]/\cI_{1,0}\otimes_{k[t]}k[t][\bfx_{2}]/\cI_{2,0}.
$$
Hence $ \cI_{1,0}+\cI_{2,0}$ is a prime ideal of Krull dimension
$r_{1}+r_{2}+2$. By Remark~\ref{rem:1} and Lemma \ref{lemm:2}\eqref{item:16}, $\cI_{0}$ is also a prime
ideal of Krull dimension $r_{1}+r_{2}+2$. Hence,
$\cI_{0}=\cI_{1,0}+\cI_{2,0}$.  This implies that
$\cI_{0}+(\ell)=\cI_{1,0}+\cI_{2,0}+(\ell)$ and so
$$
k[t][\bfx_{1},\bfx_{2}]/ (\cI_{0} +(\ell))\simeq
k[t][\bfx_{1}]/\cI_{1,0}\otimes_{k[t]}k[t][\bfx_{2}]/(\cI_{2,0}+(\ell)).
$$
Let $\bftau \in \P^{1}({k})\setminus \{(0:1)\}$ and write
$\bftau=(1:\tau)$ with $\tau\in k$.  The $k[t]$-algebra $
k[t][\bfx_{1},\bfx_{2}]/(\cI_{0} +(\ell))\otimes_{k[t]} k(\tau)$ is
isomorphic to
\begin{displaymath}
 \Big( k[t][\bfx_{1}]/\cI_{1,0} \otimes_{k[t]}
  k(\tau)\Big)\otimes_{k[t]}\Big(k[t][\bfx_{2}]/(\cI_{2,0} +(\ell)) \otimes_{k[t]} k(\tau)\Big)
\end{displaymath}
The cycle $X(\cI_{0} +(\ell)) $ coincides with $\cV_{0}\cdot
\div(\ell)$ because of the definition of the intersection product.
Since $\dim(V_{2})\ge 1$, Proposition \ref{lemm:15}\eqref{item:44}
and Lemma \ref{lemm:16} imply that the cycle $X(\cI_{2,0} +(\ell))$
coincides with the standard model of $V_{2}\cdot \div(\ell)$.

We also deduce that $\dim(\pi^{-1}(\bftau))= r_{1}+r_{2}= \dim(\cV_{0}
\cdot \div(\ell))-1$, and so there are no vertical components.  The
case when $\bftau=(0:1)$ can be treated in the same way by
considering~$\cI_\infty$ and $\cI_{i,\infty}$ instead of $\cI_{0}$ and
$\cI_{i,0}$, respectively.  \nred{Finally}, applying the isomorphism above to
the generic point of $\P^{1}_{k}$, we deduce that the generic fiber of
$\pi:\cV \cdot \div(\ell)\to \P^{1}_{k}$ coincides with
$V_{1}\#(V_{2}\cdot \div(\ell))$. This concludes the proof.
\end{proof}

\begin{proof}[Proof of Proposition \ref{prop:12}]
The statement is equivalent to the equalities
\begin{displaymath}
  \deg(X_{1}\#X_{2})=\deg(X_{1})\, \deg(X_{2}) \quad, \quad \h
(X_{1}\#X_{2})= \deg(X_{1})\, \h (X_{2})+\deg(X_{2})\, \h (X_{1}).
\end{displaymath}
The first one is \eqref{degjoin} and so we only need to prove the
second. It is enough to prove this equality for irreducible $k(\bft)$-varieties
$V_1$, $V_{2}$ over an algebraically closed field~$k$.

Suppose first that the group of parameters consists in a single
variable $t$. In this case, the proof will be done by induction on
the dimension of the $V_{i}$'s. Write $r_{i}=\dim(V_{i})$. The case
$r_{1}=r_{2}=0$ is covered by Lemma~\ref{lemm:13}, so we assume that
either $r_{1}$ or $r_{2}$ are not~0. By symmetry, we can suppose
that $r_{2}\ge 1$. Let $\cV_{i}$ and $\cV$ be the standard models of
$V_{i}$ and of $V_{1}\#V_{2}$, respectively, and $\ell\in
k[\bfx_{2}]$ a generic linear form. Then, by Proposition
\ref{lemm:12}, Theorem \ref{bezgeom},  Lemma~\ref{lemm:14} and again
Proposition \ref{lemm:12},
\begin{displaymath}
    \h(V_{1}\#V_{2})= \deg_{0,r_{1}+r_{2}+2}(\cV) =
    \deg_{0,r_{1}+r_{2}+1}(\cV\cdot \div(\ell))=\h(V_{1}\#(V_{2}\cdot
    \div(\ell))).
\end{displaymath}
Using the inductive hypothesis and Proposition \ref{lemm:15}\eqref{item:52},
we get
\begin{multline*}
  \h(V_{1}\#(V_{2}\cdot \div(\ell)))= \deg(V_{1})\, \h(V_{2}\cdot
  \div(\ell) ) +\deg(V_{2}\cdot \div(\ell) )\,
  \h(V_{1}) \\
  =\deg(V_{1})\, \h(V_{2}) +\deg(V_{2})\, \h(V_{1}),
\end{multline*}
which proves the statement for this case.

The case of an arbitrary number $p$ of parameters reduces
to the previous one as follows. Given an irreducible subvariety $W
\subset \P_{_{k(\bft)}}^{n}$, consider the field $\wt
k=k\Big(\frac{t_{1}-\gamma_{1}}{w}, \dots,
\frac{t_{p}-\gamma_{p}}{w}\Big)$ where $w$ is an additional variable
and $\gamma_{i}\in k$ is a generic element. Observe that~$\wt k (w)$
coincides with $k(\bft)(w)$ and so it is a transcendental extension
of $k(\bft)$. The scalar extension $W_{_{\wt k(w)}}\subset
\P^{n}_{_{\wt
    k(w)}}$ is an irreducible subvariety of the same dimension and
degree as $W$.  Let $ \Res_{\bfone}\Big(W_{_{\wt k(w)}}\Big)$ be the
Chow form of $W_{_{\wt k(w)}}$, primitive with respect to the base
ring~$\wt k[w]$.  By Proposition~\ref{prop:6-1}, there {exists}
$\lambda \in \wt k(w)^{\times}$ such that
\begin{displaymath}
\lambda\, \Res_{\bfone}\Big(W_{_{\wt k(w)}}\Big)(w)=\Res_{\bfone}(W)(\bft)=
\Res_{\bfone}(W)\bigg(\frac{t_{1}-\gamma_{1}}{w}w+\gamma_{1},
\dots, \frac{t_{p}-\gamma_{p}}{w}w+\gamma_{p}\bigg).
\end{displaymath}
Indeed, $\Res_{\bfone}(W)$ is a polynomial over $\wt k[w]$ and $
\Res_{\bfone}\Big(W_{_{\wt k(w)}}\Big)$ is primitive with respect to this
ring. Hence, $\lambda \in \wt k[w]\setminus\{0\}$.  Furthermore, let $\zeta$ be a
root of $\lambda$ in the algebraic closure of $\wt k$. If $\zeta\ne
0$, setting $w=\zeta$ in the equality above we get $
\Res_{\bfone}(W)=0$, which is impossible. If $\zeta=0$,
we get $0= \Res_{\bfone}(W)(\bfgamma)$, which is also impossible since
$\bfgamma$ is generic. We conclude that $\lambda\in \wt
k^\times$. In particular,
\begin{equation}\label{eq:70}
  \h(W)=\deg_{\bft}( \Res_{\bfone}(W))=\deg_{w}\Big(
  \Res_{\bfone}(W_{_{\wt
    k(w)}})\Big)=\h_{w}\Big(W_{_{\wt k(w)}}\Big),
\end{equation}
where $\h_{w}$ denotes the height with respect to the base ring $\wt
k[w]$.  Coming back to our problem, we observe that
$(V_{1}\#V_{2})_{\wt k(w)}= V_{1,\wt k(w)}\#V_{2,\wt k(w)}$, which
follows readily from the definition of the ruled join. Applying
\eqref{eq:70} and the previously considered case when $p=1$, we
conclude that
\begin{math}\h(V_1\#V_2)= \h_{w}\Big(V_{1,\wt k(w)}\#V_{2,\wt
  k(w)}\Big)= \deg\Big(V_{1,\wt k(w)} \Big)\h_{w}\Big(V_{2,\wt
  k(w)}\Big)+\deg\Big(V_{2,\wt k(w)} \Big)\h_{w}\Big(V_{1,\wt
  k(w)}\Big)=\deg(V_{1} )\h(V_{2})+\deg(V_{2} )\h(V_{1}) \end{math}.
\end{proof}

\subsection{Measures of complex polynomials}
\label{Heights of polynomials}

To study cycles defined {over $\Q$,} we will use different measures for
the size of a complex multivariate polynomial.  We introduce them in
this section and collect some of their properties.

\begin{definition}
  \label{def:1}
Let $f = \sum_{\bfa} \alpha_\bfa\, \bfx^\bfa \in
\C[x_1,\dots,x_n]$.
The {\em $\ell^\infty$-norm}, the
{\em $\ell^1$-norm} and the {\em sup-norm of $f$} are respectively
defined as
\begin{displaymath}
  ||f||_{\infty}=\max_{\bfa}|\alpha_a| \quad, \quad
||f||_1= \sum_\bfa |\alpha_\bfa| \quad ,\quad
||f||_\sup= \sup_{\bfx\in (S^{1})^{n}} |f(\bfx)|,
\end{displaymath}
where $S^{1}=\{x\in \C: |x|=1\}$ denotes the unit circle of $\C$.
The {\it Mahler measure of $f$} is defined as
$$
 \m(f)=\int_{0}^1\cdots \int_0^1  \log |f(\e^{2\pi i
 u_1},\dots,\e^{2\pi i u_n})| \, du_1\dots du_n
=\int_{(S^{1})^{n}} \log |f| d\mu^{n}
$$
where $\mu$ denotes the Haar
measure on $S^{1}$ of total mass 1.
\end{definition}

We list some inequalities comparing these measures.

\begin{lemma}
  \label{lemm:1}
Let $f\in
\C[x_{1},\dots,x_{n}]$. Then
\begin{enumerate}
\item\label{item:1}
$\log||f||_{\infty} \le \log|| f||_{\sup}\le \log||f||_1 \le
\log||f||_{\infty} + \log(n+1)\deg(f)$;
\smallskip
\item \label{item:3} $ \m(f)\le \log|| f||_{\sup}$;
\smallskip
\item \label{item:4} $ \log||f||_1 \le \m(f) + \log(n+1)\deg(f)$;
\smallskip
\item \label{item:2} $|\,\m(f)-\log||f||_{\infty}\,| \le
  \log(n+1)\deg(f)$.
\end{enumerate}
\end{lemma}

\begin{proof}\

\eqref{item:1} Let $f = \sum_\bfa \alpha_\bfa\, \bfx^\bfa$. By
Cauchy's formula, for $\bfa\in \Supp(f)$,
\begin{displaymath}
  \alpha_{\bfa} = \int_{(S^{1})^{n}}\frac{f(\bfx)}{\bfx^{\bfa}} d\mu^{n}.
\end{displaymath}
Hence $||f||_{\infty} =\max_{\bfa}|\alpha_{\bfa}| \le \sup_{\bfx\in
  (S^{1})^{n}} |f(\bfx)|=|| f||_{\sup}$, which gives the first
inequality. The second follows easily from the definitions
while the third one follows from
\begin{displaymath}
||f||_1= \sum_\bfa |\alpha_\bfa| \le \# \Supp(f)
\,    ||f||_{\infty} \le {n+\deg(f)\choose n} \,    ||f||_{\infty}
\le (n+1)^{\deg(f)} \,   ||f||_{\infty}.
\end{displaymath}
\eqref{item:3} follows easily from the definitions,
\eqref{item:4} follows from~\cite[Lem.~1.13]{Philippon86} while~\eqref{item:2} follows from \eqref{item:1}, \eqref{item:3} and
\eqref{item:4}.
\end{proof}

We also list some well-known properties of the Mahler measure.  Recall
that a \emph{weight monomial order on $\C[\bfx]$} is a partial order
on the monomials of $\C[\bfx]$ defined by a vector $\bfv \in \R^{n}$
as $\bfx^{\bfa} \prec \bfx^{\bfb}$ if and only if $\langle \bfv,
\bfa\rangle < \langle \bfv, \bfb\rangle$.

\begin{lemma}
  \label{lemm:9} Let $f, g\in \C[\bfx]$. Then
\begin{enumerate}
\item \label{item:5}  $\m(fg)=\m(f)+\m(g)$.
\smallskip
\item \label{item:6}
$\m(f(\omega_{1}x_{1},\dots, \omega_{n}x_{n}))=\m(f)$ for any
$(\omega_{1},\dots,\omega_{n}) \in (S^{1})^{n}$.
\smallskip
\item \label{item:7}
$\m(f(x_1^{\ell},\dots,x_{n}^{\ell}))=\m(f)$ for any
$\ell\ge 1$.
\smallskip
\item \label{item:46} Let $\prec$ be a weight monomial  order on
  $\C[\bfx]$. Then  $\m(\init_{\prec}(f))\le \m(f)$.
\end{enumerate}
\end{lemma}

\begin{proof}\

  \eqref{item:5}, \eqref{item:6} These follow easily from the definition of
  the Mahler measure.

\smallskip
   \eqref{item:7} For $\bfu\in \R^{n}$ set $\bfw=(\ell u_1,\dots,\ell
   u_{n})$ and observe that $dw_{1}\dots dw_{n}=\ell^{n}du_{1}\dots du_{n}$. Hence,
  \begin{multline*}
\m(f(\bfx^\ell))=
    \int_{[0,1]^{n}}  \log |f(\e^{2\pi i
\ell \bfu})| \, du_1\dots du_n =
\frac{1}{\ell^{n}}\int_{[0,\ell]^{n}} \log |f(\e^{2\pi i
 \bfw})| \, dw_1\dots dw_n \\= \int_{[0,1]^{n}}
\log |f(\e^{2\pi i
 \bfw})| \, dw_1\dots dw_n =\m(f(\bfx)).
  \end{multline*}
\eqref{item:46}
Let $\bfv\in \R^{n}$ be a vector defining  $\prec$. The
exponents of the monomials of $f$ which are minimal with respect to $\prec$
are the vectors in $\Supp(f)$ lying in
the maximal face $F$ of the Newton polytope of $f$ which has $\bfv$ as
an inner exterior normal.
Hence, $\init_{\prec}(f)$ is the face polynomial corresponding to $F$, that is, the
sum of the terms in $f$ whose exponent lies in that face.
The result then follows from the fact that the Mahler measure of a
polynomial is bounded below by the Mahler measure of any of its face
polynomials, see for instance~\cite{Smyth81}.
\end{proof}

We collect in the following lemma some further inequalities comparing the measures
of polynomials depending on groups of variables.

\begin{lemma} \label{lemm:11}
Let $f\in \C[\bfy_{1},\dots, \bfy_{m}]$ be a polynomial in $m$
groups of $n_{i}$ variables each. Then
\begin{enumerate}
\item \label{item:47}
$ \log||f||_{\sup}\le \log||f||_1 \le \log||f||_{\infty} +
\sum_{i=1}^{m} \log(n_i+1)\deg_{\bfy_{i}}(f) ; $
\smallskip
\item  \label{krso01}
\begin{math} |\,\m(f)-\log||f||_{\infty}\,| \le
  \sum_{i=1}^m\log(n_i+1)\deg_{\bfy_i}(f).
\end{math}
\end{enumerate}
\end{lemma}

\begin{proof} {\eqref{item:47} can be proved in the same way as
    the third inequality in Lemma~\ref{lemm:1}\eqref{item:1}, while
    \eqref{krso01} follows from~\cite[Lem.~1.1]{KrPaSo01}}.
\end{proof}

For multihomogeneous polynomials, we will also need the following
variant of
the Mahler measure,
introduced and studied by
Philippon in \cite{Philippon91}.

\begin{definition}\label{mahlerPh}
  \label{def:3}
Let $f \in \C[\bfx_1,\dots,\bfx_m]$ be a multihomogeneous polynomial
in $m$ groups of $n_i+1$ variables each.
The {\it Philippon measure of $f$} is defined as
\begin{equation*}
\ph( f) = \int_{S^{2n_1+1}
\times \cdots \times S^{2n_m+1}} \log |f| \, d\mu_{1}   \dots
d\mu_{m}
+\sum_{i=1}^m \bigg(\sum_{j=1}^{n_i} {1\over 2\,j}\bigg) \deg_{\bfx_i}(f) ,
\end{equation*}
where
$S^{2n_i+1}=\{\bfx\in \C^{n_i+1}:|x_0|^2+\cdots +|x_{n_i}|^2=1\}$ denotes the unit sphere
of $\C^{n_i+1}$ and $\mu_{i}$ the Borel measure on $S^{2n_{i}+1}$  of
total mass $1$, invariant under the action of the unitary group.
\end{definition}

The Philippon measure is related to the Mahler measure by {the inequalities}
\begin{equation}   \label{item:54}
0 \le  \ph(f) - \m(f)  \le  \sum_{i=1}^m\log(n_i+1)\deg_{\bfx_i}(f),
 \end{equation}
which  follow from \cite[Thm. 4]{Lelong94}.
In particular,
\begin{equation} \label{eq:4}
\ph(f)
\le \log ||f||_\sup
+  \sum_{i=1}^m\log(n_i+1)\deg_{\bfx_i}(f).
\end{equation}

\begin{definition} \label{def:6}
Let $f\in \Z[x_{1},\dots, x_{n}]$. The
{\em height of $f$} is defined as
\begin{math}
\h(f) =    \log\|f\|_{\infty}
\end{math} if $f\ne0$ and as 0 otherwise.
\end{definition}

The following  lemma estimates the
behavior of the height of polynomials with respect to the arithmetic operations and
 composition. Its proof follows directly from the definitions.

\begin{lemma}\label{lemma:2-17} Let $f_1,\dots,f_s\in \Z[x_{1},\dots, x_{n}]$. Then
\begin{enumerate}
\item \label{item:45} $\h(\sum_i f_i)\le \max_i\,\h(f_i)+ \log(s)$;
\smallskip
\item \label{item:9} $\h(\prod_i f_i)\le \h(f_{1})+  \sum_{i=2}^{s}
  \log||f_i||_{1}\le\sum_i \h(f_i)+ \log(n+1)\sum_{i=2}^{s}\deg(f_i).$
\smallskip
\item \label{item:55} Let $g\in \C[y_1,\dots,y_s]$ and write
  $d=\max_{i}\deg(f_i)$ and $h=\max_{i}\h(f_{i})$. Then
  \begin{math}
  \h(g(f_1,\dots,f_s))\le
\h(g)+\deg(g)(h+ \log(s+1)+{d\log(n+1)}).
  \end{math}
\end{enumerate}
\end{lemma}

\subsection{Canonical mixed heights of cycles over $\Q$}\label{The height in the number field case}
The projective space $\P^n=\P^{n}(\Qbarra)$
has a standard structure of toric variety with torus
\begin{displaymath}
(\P^n)^\circ:=\P^{n}\setminus V(x_{0}\cdots x_{n}) \simeq
(\Qbarra^{\times})^{n}.
\end{displaymath}
The action of this torus on  $\P^{n}$ writes down as
$\bfu \cdot \bfx =(u_{0}x_{0}:\dots:
u_{n}x_{n})$ for $\bfu=(u_{0}:\dots:u_{n})\in   (\P^n)^\circ$ and
$\bfx=(x_{0}:\dots:x_{n})\in \P^{n}$.
This toric structure on $\P^{n}$ allows to define a notion of
canonical height for its
subvarieties~\cite{BPS11}.
Following David and Philippon,  this height can be defined
by a limit process {\it \`a la} Tate~\cite{DavPhi99}.
In precise terms, for $\ell\ge1$,
consider the \emph{$\ell$-power map}
\begin{displaymath}
[\ell]:\P^{n} \longrightarrow \P^{n} \quad, \quad (x_{0}:\cdots:x_{n})\longmapsto
(x_{0}^{\ell}:\cdots:x_{n}^{\ell}).
\end{displaymath}
Let $V\subset \P^{n}$ be an irreducible subvariety and
let $\h$ denote the Fubini-Study height of projective
varieties~\cite{Philippon95, KrPaSo01}.
The \emph{canonical} {(or \emph{normalized}) \emph{height of $V$} can
  be  defined as
\begin{equation}\label{defhnormalisee}
\wh(V) = \deg(V) \, \lim_{\ell\to \infty} \,\frac{\h({[\ell] \,
V})}{\ell\, \deg([\ell]\, V)}.
\end{equation}
Both $\h$ and $\wh$  extend to cycles by linearity. Alternatively,
the canonical height can be defined using Arakelov geometry, as the
height of $V$ with respect to canonical metric on the universal line bundle
$\mathcal{O}(1)$, see for
instance~\cite{Maillot00, BPS11}.

\medskip
We collect in the
proposition below some of the basic properties of this notion.
For a cycle $X$ of $\P^{n}$, we denote by $[\ell]_*X$ its direct
image under the map $[\ell]$. We denote by $\mu_{\infty}$ the subgroup
of $\Qbarra^{\times}$ of roots of unity and by   $\mu_{\ell}$  the
subgroup of $\ell$-roots.
A variety is
called a \emph{torsion subvariety of $\P^{n}$} if it is the closure of
the orbit of the action of a subtorus of $(\P^{n})^{\circ}$ over a
point with coordinates in~$\{0\}\cup\mu_\infty$.

\begin{proposition}\label{defn_hnorm} \
\begin{enumerate}
\item \label{item:26}
Let $X\in Z_{r}(\P^{n})$. Then $\wh(X) = \lim_{\ell\to \infty} \ell^{-r-1}\,
\h([\ell]_* X)$.
\item \label{item:26b}
Let $X\in Z_{r}(\P^{n})$ and $\ell\in \Z_{>0}$. Then
$\deg([\ell]_{*}X)=\ell^r \deg(X)$ and $\wh([\ell]_*X)=\ell^{r+1}\,
\wh(X)$.
\smallskip
\item \label{item:28} If $X$ is an effective cycle, then  $\wh(X)\ge 0$.
\smallskip
\item \label{item:32} If  $X$  is a sum
of torsion subvarieties, then  $\wh(X)=0$. In particular, $\wh(\P^n)=0$.
\item \label{item:29} Let  $X\in Z_{r}^{+}(\P^{n})$. Then $ \Big|\wh(X) - \h(X) \Big| \le \frac{7}{2} \, (r+1) \,\log(n+1)
\deg (X).$
\smallskip
\item \label{item:30} Let $\bfxi\in \P^{n}$ be a point with rational
  coordinates and write $\bfxi=(\xi_{0}:\cdots : \xi_{n})$ for some coprime
   $\xi_{i}\in \Z$. Then $\wh(\bfxi)=\log(\max_{i} |\xi_{i}|)$.
\smallskip
\item \label{item:31} Let $D\in \Div^{+}(\P^{n}_\Q)$ and $f_{D}\in \Z[\bfx]$
  its primitive defining polynomial. Then $\wh(D)=\m(f_{D})$.
\end{enumerate}
\end{proposition}

\begin{proof}\

(\ref{item:26}-\ref{item:26b}) By linearity, we reduce to the case of an
irreducible variety $V$. We can assume that
$V\cap (\P^n)^\circ \ne \emptyset$, otherwise we restrict to a
sufficiently small standard subspace of $\P^n$, that is, a linear
subspace defined by a subset of the coordinates $x_{i}$.

For shorthand, let $[\ell]$ denote also the
restriction of the $\ell$-power map to the torus $(\P^n)^\circ$.
This is a group homomorphism with kernel
$\ker[\ell]\simeq \mu_{\ell}^{n}$.
Let $\Stab(V)=\{
\bfu\in (\P^n)^\circ : \, \bfu\cdot V =V\} $ be the stabilizer of
$V$.
On the one hand, by~\cite[Prop.~2.1(i)]{DavPhi99},
\begin{equation*}
\deg([\ell] \, V) = \frac{\ell^r}{\# (\Stab(V) \cap \ker[\ell])} \,
\deg(V).
\end{equation*}
On the other hand, for a generic point $\bfx\in  V$,
\begin{displaymath}
  \deg([\ell]\Big|_V)= \# \{ \bfy\in V : [\ell]\bfy = [\ell]\bfx \}
= \# \{ \bfomega\in \ker[\ell] :
\bfomega \,\bfx \in V\} =\# (\Stab(V) \cap \ker[\ell]).
\end{displaymath}
Therefore,  $\deg([\ell]_*V)=
\deg([\ell]\Big|_V)\, \deg([\ell]V)= \ell^r
\deg(V)$.
Furthermore,
$$\wh(V)=\lim_{\ell\to\infty} \frac{\deg(V)
\h([\ell]V)}{\ell\deg([\ell]V)}=\lim_{\ell\to\infty}\frac{\deg([\ell]\Big|_V)\h([\ell]V)}{\ell^{r+1}}=
\lim_{\ell\to\infty}\frac{\h([\ell]_*V)}{\ell^{r+1}}.$$

Finally
$$\wh([\ell]_*V)=\lim_{j\to\infty}\frac{\h([j]_*[\ell]_*V)}{j^{r+1}}=\ell^{r+1}
\lim_{j\to\infty}\frac{\h([j\,\ell]_*V)}{(\ell\,j)^{r+1}}=\ell^{r+1}\wh(V).$$

\eqref{item:28} This is  a direct consequence of the definition of
$\wh$ and the analogous property for
the Fubini-Study height.

\smallskip (\ref{item:32}--\ref{item:31}) These follow from
\cite[Prop.~2.1 and display~(2)]{DavPhi99}.
\end{proof}

In the sequel, we extend the notion of canonical height  to
the multiprojective setting and  study its
behavior under  geometric constructions. Our approach
relies on the analogous theory for the Fubini-Study mixed height
developed by R\'emond in~\cite{Remond01a,Remond01b}.
For simplicity, we will restrict to subvarieties of $\P^{\bfn}$ defined over $\Q$ or
equivalently, to $\Q$-varieties in
$\P^\bfn_{\Q}$, see Remark \ref{rem:4}.
We will apply the resultant theory
in~\S\ref{Some resultant theory} for the case when $A=\Z$.
In particular, the resultant of an irreducible $\Q$-variety is a primitive polynomial in $\Z[\bfU]$.

\begin{definition}[{\cite[\S 2.3]{Remond01b}}]
\label{def:4} Let  $V\subset \P^\bfn_\Q$ be an irreducible
$\Q$-variety of  dimension $r$, $\bfc\in \N_{r+1}^{m}$ and
$\bfe(\bfc)$ as defined in~\eqref{d(a)}. The \emph{Fubini-Study
(mixed) height of $V$ of index $\bfc$} is defined as
$$
\h_\bfc(V)  = \ph(\Res_{\bfe(\bfc)}(V)).
$$ This definition extends by
linearity to cycles in $Z_r(\P^\bfn_\Q)$.
\end{definition}

Next, we introduce morphisms relating different multiprojective spaces.

\begin{definition}
  \label{def:7}
Let  $\delta_{i}\in \Z_{>0}$ and set $N_i= {\delta_{i}+n_{i}\choose n_{i}}$.
The \emph{Veronese embedding} and the  \emph{modified
  Veronese embedding of index $\delta_{i}$} are the embeddings
$v_{\delta_i},\wt v_{\delta_i}:\P^{n_{i}} \hookrightarrow \P^{
  N_i-1}$ respectively defined for $ \bfx_i \in \P^{n_{i}}$ as
\begin{displaymath}
  v_{\bfdelta_i}(\bfx_i)=  (\bfx_i^{\bfa_i})_{\bfa_i\in \N^{n_{i}
+1}_{\delta_i}} \quad, \quad   \wt v_{\bfdelta_i}(\bfx_i)= \Bigg({\delta_i\choose
\bfa_i}^{1/2}\bfx_i^{\bfa_i}\Bigg)_{\bfa_i\in \N^{n_{i}
+1}_{\delta_i}}
\end{displaymath}
For $\bfdelta=(\delta_1,\dots,\delta_m)\in (\Z_{>0})^m$, we set
$v_{\bfdelta}=v_{\delta_{1}}\times \cdots\times v_{\delta_{m}}$ and $\wt v_{\bfdelta}=\wt v_{\delta_{1}}\times \cdots\times \wt
v_{\delta_{m}}$ for the \emph{Veronese embedding} and the  \emph{modified
  Veronese embedding of index $\bfdelta$}, respectively.
These are embeddings of $\P^{\bfn}$ into $\P^{\bfN-\bfone}$, where
$\bfN=(N_1,\dots,N_m)$  and $\bfone=(1,\dots, 1)$.

Consider the Segre embedding
\begin{math}
s:\P^{\bfN-\bfone}\hookrightarrow \P^{ N_1\cdots N_m -1}
\end{math}
defined as $s(\bfy_1,\dots,\bfy_m) = \Big(y_{1,j_1}\cdots
y_{m,j_m}\Big)_{1\le j_{i}\le N_i}$ for $\bfy_{i}\in \P^{N_{i}-1}$.
The composed maps
\begin{math}
 s\circ v_{\bfdelta}, s\circ \wt
v_{\bfdelta} : \P^{\bfn} \hooklongrightarrow \P^{ N_1\cdots N_m -1}
\end{math}
 are called  the \emph{Segre-Veronese embedding} and the \emph{modified  Segre-Veronese embedding of
   index $\bfdelta$},  respectively.
\end{definition}

Let $\bfdelta=(\delta_1,\dots,\delta_m)\in (\Z_{>0})^m$. For each $i$, consider the
diagonal endomorphism $\Delta_i:\P^{N_i-1} \stackrel{\sim}{\to} \P^{ N_i-1}$
defined as
$$\Delta_i((y_{i,\bfa_{i}})_{\bfa_{i}})=
\Bigg({\delta_i\choose \bfa_i}^{1/2} y_{i,\bfa_i}\Bigg)_{\bfa_i\in \N^{n_{i} +1}_{\delta_i}}.
$$
Consider also the  diagonal endomorphism $\Delta: \P^{N_1\cdots N_m-1}
\stackrel{\sim}{\to}\P^{N_1\cdots N_m-1} $ defined as
\begin{equation}
  \label{eq:29}
\Delta ((y_\bfa)_\bfa ) =  \Bigg({\delta_1 \choose
\bfa_1}^{1/2} \cdots {\delta_m \choose \bfa_m}^{1/2} \, y_\bfa
\Bigg)_{\bfa\in \N^{\bfn+\bfone}_{\bfdelta}}.
\end{equation}
These linear maps allow to write the modified Veronese and the modified Segre-Veronese embeddings
in terms of the  Veronese and Segre-Veronese embeddings
as
\begin{equation}\label{eq:30}
  \wt
v_{\delta_i}=\Delta_i\circ  v_{\delta_i}
\quad , \quad s\circ \wt
v_{\bfdelta}= \Delta \circ s\circ v_{\bfdelta}.
\end{equation}

The  degree and the Fubini-Study height of the  direct image  of a
cycle $X \in Z_{r}(\P^{\bfn}_\Q)$ under the Segre-Veronese and the
modified Segre-Veronese embeddings decompose in terms of mixed degrees and  Fubini-Study mixed heights
of $X$.

\begin{proposition} \label{prop:13}
Let  $X \in Z_{r}(\P^{\bfn}_\Q)$ and  $\bfdelta\in
(\Z_{>0})^m$. Then
\begin{displaymath}
  \deg \Big( (s\circ v_{\bfdelta})_{*}
  X\Big)= \sum_{\bfb \in \N_{r}^{m}} {r \choose \bfb} \,
  \deg_\bfb(X) \, \bfdelta^{\bfb} \ , \
 \h\Big( (s\circ \wt v_{\bfdelta})_{*}
  X\Big) =  \sum_{\bfc \in \N_{r+1}^{m}} {r+1
    \choose \bfc} \, \h_\bfc(X) \, \bfdelta^{\bfc}.
\end{displaymath}
\end{proposition}

\begin{proof}
This follows from \cite[p.~103]{Remond01b}.
\end{proof}

Let $\ell\in\Z_{>0}$. The \emph{$\ell$-power map of $\P^{\bfn}$} is
defined as
\begin{displaymath}
  [\ell] : \P^{\bfn} \longrightarrow \P^{\bfn}
\quad , \quad \bfx=(x_{i,j})_{i,j}\longmapsto \bfx^{\ell}=(x_{i,j}^{\ell})_{i,j}.
\end{displaymath}

\begin{definitionproposition}\label{mixed height and sum}
Let  $X\in Z_r(\P^\bfn_\Q)$ and  $\bfc\in \N_{r+1}^{m}$. Then the
sequence $(\ell^{-r-1} \h_\bfc([\ell]_*X))_{\ell\ge 1}$ converges
for ${\ell\to \infty}$. The limit
\begin{displaymath}
  \wh_{\bfc}(X):= \lim_{\ell\to \infty} \ell^{-r-1}
\h_\bfc([\ell]_*X)
\end{displaymath}
is called the  \emph{canonical (mixed) height of $X$ of
index $\bfc$}.
For any  $\bfdelta\in (\Z_{>0})^{m}$ it holds
\begin{equation} \label{eq:31}
\wh\Big(  (s\circ v_{\bfdelta})_{*}X\Big) = \sum_{\bfc \in \N_{r+1}^{m}} {r+1
\choose \bfc} \, \wh_\bfc(X) \, \bfdelta^{\bfc}.
\end{equation}
In particular, for a projective cycle $X\in Z_{r}(\P^n_\Q)$ we have
that $\wh_{r+1}(X)=\wh(X)$.
\end{definitionproposition}

\begin{proof}
Proposition \ref{prop:13} applied to the cycle $[\ell]_{*}X$ implies
that
\begin{equation}\label{eq:40}
 \h\Big( (s\circ \wt v_{\bfdelta})_{*}
  [\ell]_{*}X\Big) =  \sum_{\bfc \in \N_{r+1}^{m}} {r+1
    \choose \bfc} \, \h_\bfc([\ell]_{*}X) \,  \bfdelta^{\bfc}.
\end{equation}
Let $I\subset(\Z_{>0})^{m}$ be a subset of cardinality  $\#(\N_{r+1}^m)$ such that  the
square matrix $\Big({r+1
\choose \bfc} \, \bfdelta^{\bfc}\Big)_{\bfc \in \N_{r+1}^{m}, \,\bfdelta\in
I}$ is of maximal rank. Inverting this matrix, we can write the mixed
heights in the formula above in terms of heights of projective cycles as
\begin{equation}\label{eq:41}
  \h_\bfc([\ell]_{*}X) = \sum_{\bfdelta\in I } \nu_{\bfc,\bfdelta}\,
\h\Big( (s\circ \wt v_{\bfdelta})_{*}  [\ell]_{*}X\Big).
\end{equation}
with $\nu_{\bfc,\bfdelta}  \in \Q$ not depending on $\ell$.
Observe that
\begin{displaymath}
 (s\circ \wt v_{\bfdelta}\circ [\ell])_{*}X=\Delta_{*}(s\circ
v_{\bfdelta }\circ [\ell])_{*}X,
\end{displaymath}
where  $\Delta$ denotes the linear map in \eqref{eq:29}.
By \cite[Lem. 2.7]{KrPaSo01}  applied to the projective cycle $
 (s\circ v_{\bfdelta }\circ [\ell])_{*}X$, the map
 $\Delta$ and its inverse,  there exists
$\kappa(m,r,\bfdelta)\ge0$ such that \nred{the quantity $|  \h\Big( (s\circ \wt v_{\bfdelta})_{*}
  [\ell]_{*}X\Big)-\h\Big( (s\circ v_{\bfdelta})_{*}
  [\ell]_{*}X\Big)|$} is bounded above by
\begin{equation*}
\kappa(m,r,\bfdelta)\,\log\Bigg(\prod_{i=1}^m{\delta_i+n_i\choose n_i}\Bigg)
   \deg((s\circ v_{\bfdelta})_{*}
  [\ell]_{*}X).
\end{equation*}
We have that  $[\ell]$ commutes with $s\circ v_{\bfdelta}$. By
Proposition~\ref{defn_hnorm}\eqref{item:26b},
\begin{displaymath}
 \deg((s\circ v_{\bfdelta})_{*}
  [\ell]_{*}X)=\deg([\ell]_{*}(s\circ v_{\bfdelta})_{*} X)=
\ell^{r}\, \deg((s\circ v_{\bfdelta})_{*} X).
\end{displaymath}
We deduce
\begin{math}
\h\Big( (s\circ \wt v_{\bfdelta})_{*}
  [\ell]_{*}X\Big)=\h\Big( [\ell]_{*} (s\circ v_{\bfdelta})_{*}
  X\Big) +O(\ell^{r}).
\end{math}
Therefore, for each $\bfdelta\in I$,
 \begin{displaymath}
\lim_{\ell\to \infty}\ell^{-r-1}\, \h\Big( (s\circ \wt v_{\bfdelta})_{*}
  [\ell]_{*}X\Big) =
\lim_{\ell\to \infty}\ell^{-r-1}\, \h\Big([\ell]_{*} (s\circ
v_{\bfdelta})_{*}X\Big) = \wh\Big( (s\circ v_{\bfdelta})_{*}
  X\Big).
\end{displaymath}
This proves that the sequence $(\ell^{-r-1}
\h_\bfc([\ell]_*X))_{\ell\ge 1}$ converges for ${\ell\to \infty}$,
since it is a linear combination of convergent
sequences as shown in \eqref{eq:41}. The formula \eqref{eq:31} follows from \eqref{eq:40}
by passing to the limit for ${\ell\to \infty}$. The last statement is
this formula applied to  $X\in
Z_{r}(\P^n_\Q)$ and $\delta=1$.
\end{proof}

\begin{remark}
  \label{rem:8}
  The definition of canonical mixed heights in
  \cite[Formula~(1.3)]{PhiSom08} is different {from the one presented
  here.  Nevertheless, both notions coincide as they both
  satisfy~\eqref{eq:31}.} These mixed heights can be alternatively
  defined using Arakelov geometry, as explained at the end of
  \cite[\S{}I]{PhiSom08}, and they correspond to the canonical mixed
  heights induced by the toric structure of $\P^{\bfn}_{\Q}$, see \cite{BPS11}.
\end{remark}

\begin{definition}\label{def:8}
Let $\eta$ be an indeterminate. The \emph{extended Chow ring of $\P^{\bfn}_\Q$}
is the graded ring
\begin{displaymath}
  A^{*}(\P^{\bfn}_{\Q};{\Z})=A^{*}(\P^{\bfn}_\Q)\otimes _{\Z}\R[\eta]/(\eta^{2})\simeq
\R[\eta,\theta_1,\dots,\theta_m]/(\eta^2,\theta_1^{n_1+1},\dots,\theta_m^{n_m+1}),
\end{displaymath}
where $\theta_{i}$ denotes the class in
$A^{*}(\P^{\bfn}_{\Q})$ of the inverse image of a
hyperplane of $\P^{n_{i}}_{\Q}$ under the projection $\P^{\bfn}_{\Q}\to
\P^{n_{i}}_{\Q}$.
For short, we alternatively denote this ring as
$A^{*}(\P^{\bfn};\Z)$.
To a cycle   $X\in Z_r(\P^\bfn_\Q)$ we associate an element of this
ring, defined as
\begin{displaymath}
[X]_{_{\Z}}=\sum_{\bfc\in\N_{r+1}^m,\, \bfc\le \bfn}\wh_{\bfc}(X)\, \eta\,
\theta_{1}^{n_{1}-c_{1}} \cdots \theta_{m}^{n_{m}-c_{m}} +
\sum_{\bfb\in\N_{r}^{m}, \bfb\le \bfn}\deg_{\bfb}(X) \, \theta_{1}^{n_{1}-b_{1}}
\cdots \theta_{m}^{n_{m}-b_{m}}.
\end{displaymath}
It is a homogeneous element of degree $|\bfn|-r$.
\end{definition}
{There is an inclusion}
$\imath:A^{*}(\P^{\bfn}_{\Q})\hookrightarrow
A^{*}(\P^{\bfn}_{\Q};\Z)$ {which satisfies}
\begin{math} [X]_{_{\Z}} \equiv \imath([X]) \pmod{\eta}.
\end{math}
In particular,
the class of a cycle in the Chow ring
is determined by its class in the extended Chow ring.

For $X\in Z_{r}(\P^{\bfn}_{\Q})$, Theorem
\ref{multihom}\eqref{item:25} shows that $\wh_{\bfc}(X)=0$ for
every $\bfc$ such that $c_{i}>n_{i}$ for some $i$.  Hence
$[X]_{_{\Z}}$ contains the information of all mixed degrees and
heights, since $\{\bftheta^\bfa, \eta\, \bftheta^\bfa\}_{\bfa\le
  \bfn}$ is a basis of $A^{*}(\P^{\bfn};\Z)$.

\smallskip
Next proposition extends the first properties of  the canonical
height in Proposition~\ref{defn_hnorm} to the multiprojective setting.
The space $\P^{\bfn}$ is a toric variety with torus
$$
(\P^{\bfn})^{\circ}= \prod_{i=1}^{m}(\P^{n_{i}})^{\circ}
\simeq (\Qbarra^{\times})^{|\bfn|}.
$$
A variety $V$ is
called a \emph{torsion subvariety of $\P^{\bfn}$} if it is the closure of
the orbit of the action of a subtorus of $(\P^{\bfn})^{\circ}$ over a
point with coordinates in $\{0\}\cup\mu_\infty$.

\begin{proposition} \label{prop:10}
\
\begin{enumerate}
\item \label{item:33} Let $X\in Z_{r}(\P^{\bfn}_\Q)$ and  $\ell\ge 1$. Then
$\deg_{\bfb}([\ell]_*X)=\ell^{r}\deg_{\bfb}(X)$ and
 $
\wh_{\bfc}([\ell]_*X)=\ell^{r+1}\wh_{\bfc}(X)$ for all $\bfb\in \N_r^m$ and  $\bfc\in \N^{m}_{r+1}$.
\smallskip
\item \label{item:34}  Let $X\in Z_{r}^{+}(\P^{\bfn}_\Q)$. Then
  $[X]_{_{\Z}}\ge0$. In particular, $\wh_{\bfc}(X)\ge 0$ for all $\bfc\in \N^{m}_{r+1}$.
\smallskip
\item \label{item:35} If $X$ is a linear combination of torsion
  subvarieties, then  $ [ X]_{_{\Z}}
= \imath([X])$ or equivalently,
  $\wh_{\bfc}(X)=0$ for all $\bfc\in \N^{m}_{r+1}$.  In particular,
  $[\P^\bfn_\Q]_{_{\Z}}=1$.
  \smallskip
\item \label{item:36} Let $X\in Z_r^+(\P^\bfn_\Q)$. Then there exists $\kappa(r,m) \ge 0$
such that for all $\bfc\in \N^{m}_{r+1}$,
$$\Big|\wh_\bfc(X)-\h_\bfc(X)\Big|\le
\kappa(r,m)\log(|\bfn|+1)\sum_{\bfb\in \N^{m}_r} \deg_\bfb(X).
$$
\end{enumerate}
\end{proposition}

\begin{proof}\

\eqref{item:33} Recall that $[\ell]$ commutes with $s\circ
v_{\bfdelta}$. Hence, propositions~\ref{prop:13}
and~\ref{defn_hnorm}\eqref{item:26b} imply that
\begin{multline*}\ell^r\sum_{\bfb\in \N_r^m}{r \choose \bfb} \,\deg_\bfb(X)\bfdelta^\bfb  =
\ell^r \deg\Big( (s\circ v_{\bfdelta})_{*}X\Big)\\[-5mm] = \deg \Big( (s\circ v_{\bfdelta})_{*}[\ell]_*
  X\Big) = \sum_{\bfb \in \N_{r}^{m}} {r \choose \bfb} \,
  \deg_\bfb([\ell]_* X) \, \bfdelta^{\bfb}. \end{multline*}
  Since this holds for all $\bfdelta\in (\Z_{>0})^m$, we deduce that $\deg_\bfb([\ell]_*X)=\ell^r\deg_\bfb(X)$.
The statement for the height follows analogously by using
\eqref{eq:31} and Proposition~\ref{defn_hnorm}\eqref{item:26b}.

\smallskip

\eqref{item:34} The non-negativity of the canonical mixed heights is a
consequence of the non-negativity of the Fubini-Study mixed heights. The
latter follows from the estimates in \eqref{item:54} and the
non-negativity of the Mahler measure of a polynomial with integer
coefficients. Hence, $\wh_{\bfc}(X)\ge 0$ for all $\bfc\in
\N^{m}_{r+1}$. The rest of the statement follows from this together
with Proposition \ref{prop:5}\eqref{item:48}.

\smallskip
\eqref{item:35} It suffices to prove this statement for a torsion
subvariety $V$ of $\P^{\bfn}$.
Given $\bfdelta\in
(\Z_{>0})^{m}$, Proposition~\ref{defn_hnorm}\eqref{item:32} implies
that $\wh((s\circ v_\bfdelta)_*X)=0$ since the image of $V$
under a Segre-Veronese embedding is
a torsion subvariety of a projective space.
Therefore, the right-hand side of~\eqref{eq:31} is equal to 0 for all   $\bfdelta$,
which implies that all  its coefficients are 0.
Hence all canonical mixed heights of $V$ are 0, as stated.

\smallskip
\eqref{item:36}
By \eqref{eq:31} and Proposition \ref{prop:13},
for each $\bfdelta\in  (\Z_{>0})^{m}$,
\begin{equation*}
\wh\Big(  (s\circ v_{\bfdelta})_{*}X\Big) - \h\Big(  (s\circ \wt v_{\bfdelta})_{*}X\Big)= \sum_{\bfc \in \N_{r+1}^{m}} {r+1
\choose \bfc} \, \Big(\wh_\bfc(X)-\h_\bfc(X)\Big) \, \bfdelta^{\bfc}.
\end{equation*}
As in the proof of Proposition-Definition~\ref{mixed height and sum},
we pick a subset   $I\subset(\Z_{>0})^{m}$ of cardinality
$\#(\N_{r+1}^m)$ such that the square matrix  $M_{I}:=\Big({r+1
\choose \bfc} \, \bfdelta^{\bfc}\Big)_{\bfc  \in \N_{r+1}^{m}, \bfdelta\in
I}$ is of maximal rank. Inverting this matrix, we obtain for each $\bfc$
\begin{equation}\label{eq:41-}
  \wh_\bfc(X)-\h_\bfc(X) = \sum_{\bfdelta\in I } \nu_{\bfc,\bfdelta}\,
\Big(\wh\Big(  (s\circ v_{\bfdelta})_{*}X\Big) - \h\Big(  (s\circ \wt v_{\bfdelta})_{*}X\Big)\Big)
\end{equation}
with  $\nu_{\bfc,\bfdelta}  \in \Q$.
Using \eqref{eq:30}, \cite[Lem.~2.7]{KrPaSo01},
Proposition~\ref{prop:13} and the inequality
$\log\Big(\prod_{i=1}^m{\delta_i+n_i\choose n_i}\Big)\le
|\bfdelta|\,\log(|\bfn|+1)$,
we get  that there exists  $ \kappa_1(r,m,\bfdelta)\ge 0$ such that
 \begin{equation*}
  \Big|    \h\Big(  (s\circ \wt v_{\bfdelta})_{*}X\Big) -\h\Big(
(s\circ v_{\bfdelta})_{*}X\Big)|\le \kappa_1(r,m,\bfdelta)\,
\log(|\bfn|+1) \, \sum_{\bfb\in \N_r^m}\deg_\bfb(X).
 \end{equation*}
By propositions~\ref{defn_hnorm}\eqref{item:29} and~\ref{prop:13},
there exists  $\kappa_2(r,m,\bfdelta)\ge 0$ such that
\begin{equation*}
 \Big|\wh\Big((s\circ v_{\bfdelta})_{*}X\Big)-
\h\Big((s\circ v_{\bfdelta})_{*}X\Big)\Big| \le
\kappa_2(r,m,\bfdelta)\,
\log(|\bfn|+1) \, \sum_{\bfb\in \N_r^m}\deg_\bfb(X).
\end{equation*}
Therefore, setting
$\kappa{_{3}}(r,m,\bfdelta)=\kappa_{1}(r,m,\bfdelta)+\kappa_{2}(r,m,\bfdelta)$,
we get
\begin{equation*}
   \Big|\wh\Big((s\circ v_{\bfdelta})_{*}X\Big)-
  \h\Big((s\circ\wt  v_{\bfdelta})_{*}X\Big)\Big| \le
  \kappa{_{3}}(r,m,\bfdelta)\,
  \log(|\bfn|+1) \,  \sum_{\bfb\in \N_r^m}\deg_\bfb(X).
\end{equation*}
Observe that the matrix $M_{I}$ does not depend on  $\bfn$.
Using  \eqref{eq:41-}, we deduce that there exists $\kappa(r,m)$
such that
$$ |\wh_\bfc(X)-\h_\bfc(X)|  \le
\kappa(r,m) \, \log(|\bfn|+1)  \, \sum_{\bfb\in \N^{m}_r} \deg_\bfb(X)
$$
for all $\bfc\in \N^{m}_{r+1}$, as stated.
\end{proof}

The following proposition  describes the mixed heights and
classes  of points and
divisors, extending
Proposition \ref{defn_hnorm}(\ref{item:30}-\ref{item:31}) to the multiprojective setting.

\begin{proposition}\label{mahler} \
  \begin{enumerate}
\item \label{item:51} Let $\bfxi=(\bfxi_{1},\dots, \bfxi_{m})\in \P^{\bfn}$
be a point with coordinates in
  $\Q$ and  for each  $1\le i\le m$  write $\bfxi_{i}=(\xi_{i,j})_j$ for
  coprime  $\xi_{i,j}\in \Z$. Then
  \begin{displaymath}
    [\bfxi]_{_{\Z}}=   \sum_{i=1}^m \wh(\bfxi_i)\, \eta\,\bftheta^{\bfn-\bfe_i} +  \bftheta^\bfn.
  \end{displaymath}
In particular,
$\wh_{\bfe_{i}}(\bfxi)= \wh(\bfxi_i)$.
\smallskip
\item \label{item:37} Let $D\in \Div^{+}(\P^{\bfn}_\Q)$ and    $f_{D}\in
  \Z[\bfx]$    its primitive defining polynomial. Then
  \begin{displaymath}
 [D]_{_{\Z}} =   \m(f_D)\, \eta  + \sum_{i=1}^m\deg_{\bfx_i}(f_D)\, \theta_i .
  \end{displaymath}
In particular,  $\wh_{\bfn}(D)=\m(f_{D})$.
  \end{enumerate}
\end{proposition}

We need the following lemma.

\begin{lemma}
  \label{lemm:8}
Let  $H$ be an irreducible $\Q$-hypersurface of $\P^{\bfn}_\Q$ which
is not a standard hyperplane and $\ell\ge 1$. Then
  \begin{displaymath}
f_{[\ell]^{*}[\ell]_{*}H} = \prod_{\bfomega\in \ker[\ell]}
f_{H}({\bfomega}\cdot \bfx).
  \end{displaymath}
\end{lemma}

\begin{proof}
Let $
\Stab(H)=\{ \bfu\in (\P^{\bfn})^{\circ} : \, \bfu \cdot H=H\}
$
be the stabilizer of $V$. Then
$$
[\ell]^{-1}[\ell] H=\bigcup_{\bfomega}\bfomega \cdot H,
$$
{the union being over a set of representatives} of
${\ker[\ell]}/{(\Stab(H)\cap\ker[\ell])}$.
As in the proof of Proposition \ref{defn_hnorm}(\ref{item:26}-\ref{item:26b}), we can show
that $\deg([\ell]\Big|_{H})= \# (\Stab(H)\cap\ker[\ell])$.
Hence,
\begin{displaymath}
[\ell]^{*}[\ell]_{*} H= \deg([\ell]\Big|_{H})\, [\ell]^{-1}[\ell]\,H
= \sum_{\bfomega \in \ker[\ell]} \bfomega\cdot H.
\end{displaymath}
This implies that the primitive polynomial defining this divisor satisfies $$f_{[\ell]^{*}[\ell]_{*} H}= \lambda \,
\prod_{\bfomega\in \ker[\ell]} f_{H}({\bfomega}\cdot \bfx)$$ with
$\lambda\in \Q^{\times}$.
It only remains to prove that $\lambda=\pm1$.
Let $K$ be the $\ell$-th cyclotomic field, so
that $f_{H}({\bfomega}\cdot \bfx)\in K[\bfx]$ for all $\bfomega\in \ker[\ell]$.
Let $p\in \Z$ be a prime number and $v$ a valuation of $K$ extending
$\ord_{p}$.
By Gauss Lemma,
\begin{displaymath}
\ord_{p}\bigg(  \prod_{\bfomega} f_{H}({\bfomega}\cdot \bfx)\bigg)
= \sum_{\bfomega} v\Big(f_{H}({\bfomega}\cdot \bfx)\Big) = 0
\end{displaymath}
since $f_{H}$ is primitive and $v(\bfomega)= 0$
for all $\bfomega\in \ker[\ell]$. This implies that
$\ord_{p}(\lambda)=0$ for all $p$ and $\lambda=\pm1$ as desired.
\end{proof}

\begin{proof}[Proof of Proposition~\ref{mahler}.] {\ }\\
\eqref{item:51} It is enough to prove that $\wh_{\bfe_{i}}(\bfxi) = \wh(\bfxi_i)$.
We have that $[\ell]_{*}\bfxi= \bfxi^{\ell}$, hence
$\Res_{\bfe_{i}}([\ell]_{*}\bfxi)=  \lambda_i\,L_{i}(\bfxi_{i}^{\ell})$, where
 $L_{i}$ is the general linear form of multidegree~$\bfe_{i}$ and
 $\lambda_i\in \Q^\times$, thanks to Corollary~\ref{cor:2}. Since
 $\Res_{\bfe_i}(\bfxi)$ and $L_i(\bfxi_i)$ are primitive polynomials,
 we deduce that
  $\lambda_i=\pm 1$.
Applying Definition~\ref{def:4}, the estimates in \eqref{item:54} and
Lemma \ref{lemm:1}\eqref{item:2}, we obtain
   \begin{displaymath}
   \ell^{-1}\,  \h_{\bfe_{i}}([\ell]_{*}\bfxi) = \ell^{-1}\,  \ph\Big(
    L_{i}(\bfxi_{i}^{\ell})\Big)
= \log(\max_{j} |\xi_{i,j}|) + O(\ell^{-1}).
  \end{displaymath}
The statement follows by letting   $\ell\to \infty$.

\smallskip
\eqref{item:37} By Proposition~\ref{prop:5}\eqref{gradosparciales},
$$
[ D]_{_{\Z}} = \wh_{\bfn}(D) \, \eta +  [D]=
 \wh_{\bfn}(D) \, \eta + \sum_{i=1}^m\deg_{\bfx_i}(f_{D})\, \theta_i.
$$
Thus, we only have to prove that $\wh_{\bfn}(D)=\m(f_{D})$.
We reduce  without loss of generality to the case of an irreducible
hypersuface $H$.
If $H$ is a standard hyperplane of $\P^{\bfn}_\Q$, then
$f_{H}=x_{i,j}$ for some $i,j$. Hence
$\wh_{\bfn} (H)=0=\m(x_{i,j})$, and  the statement is true in this case.

We can then suppose that $H$ is an irreducible $\Q$-hypersurface which
is not a standard hyperplane.
Let $\ell \ge 1$. By definition,
$f_{[\ell]^{*}[\ell]_{*}H} =f_{[\ell]_{*}H}(\bfx^{\bfell})$.
Applying successively  Lemma \ref{lemm:9}\eqref{item:7},
Lemma~\ref{lemm:8} and Lemma \ref{lemm:9}\eqref{item:6}, we get
\begin{equation*}
\m(f_{[\ell]_{*}H} )=
  \m(f_{[\ell]_{*}H}(\bfx^{\bfell}) )= \sum_{\omega\in \ker[\ell]}
\m(f_{H}({\omega}\cdot  \bfx))=  \ell^{|\bfn|}
\m(f_H).
\end{equation*}
By \cite[Thm.~3.4]{Remond01b}, the estimates in \eqref{item:54} and Proposition~\ref{prop:10}\eqref{item:33},
\begin{equation*}
 \h_\bfn([\ell]_{*}H)=\ph(f_{[\ell]_{*}H})+
\sum_{i=1}^m  \bigg( \sum_{j=1}^{n_{i}-1} \sum_{l=1}^{j}
\frac{1}{2l}\bigg) \deg_{\bfn -\bfe_{i}}([\ell]_{*}H)
= \m(f_{[\ell]_{*}H})+ O(\ell^{|\bfn|-1}).
\end{equation*}
Therefore,
$
{\ell^{-|\bfn|}} \, \h_{\bfn}([\ell]_{*} H) =
\m(f_{H}) + O(\ell^{-1}) .
$
We conclude that
$\wh_{\bfn}(H)=\m(f_{H})$ by letting  $\ell\to \infty$.
\end{proof}

The following result gives the behavior of mixed degrees and heights with
respect to Veronese embeddings.

\begin{proposition}\label{veronese}
For  $\bfdelta\in (\Z_{>0})^{m}$ let
 $ v_\bfdelta$ be the Veronese embedding of index $\bfdelta$, $X\in Z_{r}(\P^\bfn_\Q)$, $\bfb\in \N^{m}_{r}$ and $\bfc\in
\N^{m}_{r+1}$. Then
\begin{displaymath}
\deg_\bfb({(v_{\bfdelta})_{*}}X) = \bfdelta^\bfb\,\deg_\bfb (X) \quad,
\quad  \wh_\bfc({(v_{\bfdelta})_{*}}X) = \bfdelta^\bfc\,\wh_\bfc (X).
\end{displaymath}
\end{proposition}

\begin{proof}
  {It is enough to prove the statement for an irreducible
    $\Q$-variety~$V$.}  Set $N_{i}={\delta_i+n_i\choose n_i}$ for
  $1\le i\le m$, $\bfN=(N_1,\dots,N_m)$ and $\bfone=(1,\dots, 1)\in
  \N^{m}$.  The embedding $v_\bfdelta$ induces an isomorphism of
  multigraded algebras
\begin{equation*}
\Q[\P^{\bfN-\bfone}]/I(v_\bfdelta(V))=\bigoplus_{\bfd\in\N^m}\Big(\Q[\P^{\bfN-\bfone}]/I(v_\bfdelta(V))\Big)_\bfd\simeq \bigoplus_{\bfd\in
  \N^{m}}\Big(\Q[\P^{\bfn}]/I(V)\Big)_{(d_{1}\delta_{1} ,\dots, d_{m}\delta_{m})}.
\end{equation*}
Hence, the Hilbert-Samuel functions of~$V$ and of $v_\bfdelta(V)$
satisfy  $H_{v_{\bfdelta}(V)}
(d_{1},\dots, d_{m})=  H_{V}(\delta_{1} d_{1},\dots, \delta_{m} d_{m})$ for all $\bfd$.
Comparing the coefficient of the monomial $\bfd^{\bfb}$ in the
corresponding Hilbert polynomials and using that
$v_{\bfdelta}$ is a map of degree 1, it follows that
 $\deg_\bfb({(v_{\bfdelta})_{*}}V)=\deg_\bfb(v_{\bfdelta
 }(V))=\bfdelta^\bfb\deg_\bfb(V)$.

\medskip
Concerning the height, consider the embeddings
\begin{displaymath}
  s\circ v_{\bfd \bfdelta}: \P^{\bfn}\hooklongrightarrow
  \P^{N'_1\cdots N'_m -1}\quad, \quad s\circ v_\bfd\circ v_{\bfdelta}: \P^{\bfn}\hooklongrightarrow
\P^{M_1\cdots M_m-1},
\end{displaymath}
where $N'_i={d_i\delta_i+n_i\choose n_i}$ and  $M_i={d_i+N_i-1\choose
  N_i-1}$.
For $\bfx_{i}\in \P^{n_{i}}$,
$$
v_{d_i\delta_i}(\bfx_i)=(\bfx_i^{\bfa_i})_{|\bfa_i|=d_i\delta_i}
\quad, \quad
v_{d_i}\circ v_{\delta_i}(\bfx_i)=
\Big(\big((\bfx_i^{\bfb_i})_{|\bfb_i|=\delta_i}\big)^{\bfc_i}\Big)_{|\bfc_i|=d_i}
$$
{Observe that the monomials appearing in the image of both maps
  are the same.}  This implies that there are linear maps
$A_{i}:\P^{M_i-1}\to \P^{N'_i-1} $ and $B_{i}: \P^{N'_i-1} \to
\P^{M_i-1}$ such that $ A_{i}\circ v_{d_i}\circ v_{\delta_i} =
v_{d_i\delta_i}$ and $ v_{d_i}\circ v_{\delta_i} = B_{i}\circ
v_{d_i\delta_i}$.  In turn, this implies that there exist linear maps
$A:\P^{M_1\cdots M_m-1}\to \P^{N'_1\cdots N'_m -1}$ and
$B:\P^{N'_1\cdots N'_m -1}\to \P^{M_1\cdots M_m-1}$ such that
\begin{displaymath}
  A\circ s\circ v_{\bfd}\circ
  v_{\bfdelta} =  s\circ v_{\bfd\bfdelta} \quad ,\quad
s\circ v_{\bfd}\circ
  v_{\bfdelta} =  B\circ s\circ v_{\bfd\bfdelta}.
\end{displaymath}
Let $\ell\ge 1$.
We apply  \cite[Lem.~2.7]{KrPaSo01} to
compare the Fubini-Study height of the image of
the cycle $[\ell]_{*}V$ under the maps
$s\circ v_{\bfd\bfdelta}$ and $s\circ
v_{\bfd}\circ v_{\bfdelta}$.
Using also propositions \ref{prop:13} and
\ref{prop:10}\eqref{item:33}, we obtain
that there exists $\kappa(\bfn,\bfd,\bfdelta)\ge0$ such that
\begin{equation}\label{eq:66}
\Big|\h(s\circ v_{\bfd\bfdelta}  ([\ell]_{*}V))- \h(s\circ
v_{\bfd}\circ v_{\bfdelta} ([\ell]_{*}V))\Big| \le
\kappa(\bfn,\bfd,\bfdelta)\sum_\bfb \deg_\bfb([\ell]_{*}V)= O(\ell^{r}).
 \end{equation}
The map $[\ell]$ commutes with $v_{\bfdelta}$,
$v_{\bfd}$, $v_{\bfd\bfdelta}$ and $s$. Hence,
$s\circ v_{\bfd\bfdelta}  ( [\ell]_*V)=[\ell]_*s\circ v_{\bfd\bfdelta}
( V)$ and $s\circ
v_{\bfd}\circ v_{\bfdelta} ( [\ell]_*V)=[\ell]_*s\circ
v_{\bfd}\circ v_{\bfdelta} ( V)$.
From \eqref{eq:66}, we deduce that
 $\h([\ell]_*s\circ
v_{\bfd}\circ v_{\bfdelta} ( V))=\h([\ell]_*s\circ v_{\bfd\bfdelta}  ( V))+ O(\ell^r)$.
Therefore,
\begin{multline*}\wh( s\circ
v_{\bfd}\circ v_{\bfdelta} ( V))=\lim_{\ell \to \infty} \ell^{-r-1} h([\ell]_*s\circ
v_{\bfd}\circ v_{\bfdelta} ( V))\\= \lim_{\ell \to \infty} \ell^{-r-1}\h([\ell]_*s\circ v_{\bfd\bfdelta}  ( V))=
\wh(s\circ v_{\bfd\bfdelta}  ( V) ).
\end{multline*}
Hence, Proposition \ref{prop:13} implies that
\begin{multline*}
\sum_{\bfc} {r+1 \choose \bfc} \, \wh_\bfc(v_\bfdelta (V)) \,
\bfd^{\bfc}=  \wh( s\circ
v_{\bfd}\circ v_{\bfdelta} ( V))\\[-2mm]=
 \wh(s\circ v_{\bfd\bfdelta}  ( V) )= \sum_{\bfc} {r+1 \choose \bfc} \, \wh_\bfc(V) \,
(\bfd\,\bfdelta)^{\bfc}.
\end{multline*}
As this hold for all $\bfd\in \N^{m}$, it follows that $
\wh_\bfc(v_\bfdelta (V))=\bfdelta^{\bfc} \,\wh_\bfc(V) $, as stated.
\end{proof}

The following result is the arithmetic analogue of B\'ezout's theorem
for multiprojective $\Q$-cycles. It contains Theorem \ref{thm:1} in
the introduction.  Given a multihomogeneous polynomial $f\in
\Z[\bfx_1,\dots,\bfx_m]$, we consider the element in the extended Chow
ring
\begin{equation} \label{clasedef} [f]_{\sup}=\log||f||_\sup\, \eta +
  \sum_{i=1}^m \deg_{\bfx_i}(f)\theta_i \in A^{*}(\P^{\bfn};\Z).
\end{equation}
Observe that $ [\div(f)]_{_{\Z}} \le [f]_{\sup}$.

\begin{theorem}\label{multihom}
  Let $X\in Z_{r}(\P^\bfn_{\Q})$ and $f\in
  \Z[\bfx_{1},\dots,\bfx_{m}]$ a multihomogeneous polynomial such that
  $X$ and $\div(f)$ intersect properly.
\begin{enumerate}
 \item \label{item:41}
If $X$ is effective, then for any $\bfb\in \N_{r}^{m}$,
\begin{displaymath}
\wh_\bfb(X\cdot \div(f)) \le
\sum_{i=1}^m\deg_{\bfx_i}(f)  \wh_{\bfb+\bfe_i}(X)
 +\log||f||_{\sup}\deg_\bfb(X).
 \end{displaymath}
\item \label{item:43}   $\wh_\bfc(X)=0$ for any $\bfc\in \N_{r+1}^m$ such that
 $c_i>n_i$ for some $i$.
\smallskip
\item \label{item:42} If $X$ is effective, then $[ X\cdot\div(f)]_{_{\Z}} \le
[ X]_{_{\Z}}\cdot [ f]_{\sup}$.
\end{enumerate}
\end{theorem}

\begin{proof}\

  \eqref{item:41} We reduce without loss of generality to the case of
  an irreducible $\Q$-variety $V$. Let $\ell\ge1$ and set
  $\bfell=(\ell,\dots,\ell)\in (\Z_{>0})^{m}$. Consider the Veronese
  embedding $v_\bfell:\P^\bfn \rightarrow \P^{\bfN-\bfone}$ where
  $\bfN=\Big({\ell+n_i\choose n_i}, \dots, {\ell+n_i\choose n_i}\Big),
  \bfone=(1,\dots, 1)\in \N^{m}$.  Set
  $\bfd=(d_{1},\dots,d_{m})=\bfdeg(f)$. We have that
  $f^{\ell}\in\Q[\P^{\bfn}]_{\ell\, \bfd}$ and so there is a unique
  polynomial $F_{\ell}\in \Q[\P^{\bfN-\bfone}]_{\bfd}$ such that
\begin{math}
  v_{\bfell}^{*} \, \div(F_{\ell})= \div(f^{\ell})= \ell \, \div(f).
\end{math}
Then,  the projection formula~\eqref{eq:38}  implies that
\begin{equation}\label{eq:44}
  {(v_{\bfell})_{*}}V \cdot \div(F_{\ell})= \ell \, {(v_{\bfell})_{*}}(V\cdot
  \div(f)).
\end{equation}
Set
\begin{math}
||F||_{_{v_{\ell}(V)}}  =\sup_{\bfxi\in
    v_{\ell}(V)}\frac{|F_{\ell}(\bfxi)|}{||\bfxi_{1}||^{d_{1}}\cdots ||\bfxi_{m}||^{d_{m}}}
\end{math}, where $||\cdot||$ is the Euclidean norm.  We apply the
arithmetic B\'ezout's theorem for the Fubini-Study mixed height
(\cite[Thm.~3.4 and Cor.~3.6]{Remond01b}) to the variety
$Z_{\ell}:=v_{\bfell} (V)= {(v_{\bfell})_{*}}V $ and the divisor
$\div(F_{\ell})$ and we obtain
\begin{displaymath}
\h_\bfb(Z_{\ell}\cdot \div(F_\ell))\le  \sum_{i=1}^m d_{i}\,
\h_{\bfb+\bfe_i}(Z_{\ell}) +\deg_\bfb(Z_{\ell}) \,
\log\Big(||F_{\ell} ||_{_{v_{\ell}(V)}}\Big) .
\end{displaymath}
Applying Proposition~\ref{prop:10}\eqref{item:36} to $Z_{\ell}$ and to
$Z_{\ell}\cdot \div(F_\ell)$ together with the multiprojective
B\'ezout's theorem~\ref{bezgeom}, it follows that there exist
$\kappa(r,m,\bfd)\ge0$ such that
\begin{multline}\label{eq:67}
\wh_\bfb(Z_{\ell}\cdot \div(F_\ell))\le  \sum_{i=1}^m d_{i}\,
\wh_{\bfb+\bfe_i}(Z_{\ell}) +\deg_\bfb(Z_{\ell}) \,
\log(||F_{\ell} ||_{_{v_{\ell}(V)}}\Big) \\  + \kappa(r,m,\bfd)
\log(|\bfN-\bfone|+1) \, \sum_{\bfb} \deg_\bfb(Z_{\ell}).
\end{multline}
By   Proposition~\ref{veronese}, $ \deg_{\bfb}(Z_{\ell})=  \ell^{r}\,
\deg_{\bfb}(V)$ and $ \wh_{\bfb+\bfe_i}(Z_{\ell})=  \ell^{r+1}\,
\wh_{\bfb+\bfe_i}(V)$. The same result  together with
\eqref{eq:44} also implies  that
\begin{displaymath}
\wh_\bfb(Z_{\ell}\cdot \div(F_\ell))
= \wh_\bfb(\ell \, v_{\bfell,*}(V\cdot \div(f))) = \ell^{r+1}\,
\wh_\bfb(V\cdot \div(f)).
\end{displaymath}
We have that
\begin{multline*}
  ||F_\ell ||_{_{v_{\ell}(V)}} = \sup_{\bfx\in
    V}\frac{|f^{\ell}(\bfx)|}{||v_{\ell}(\bfx_{1})||^{d_{1}}\cdots
    ||v_{\ell}(\bfx_{m})||^{d_{m}}}\\
\le \sup_{\bfx\in
    V}\frac{|f^{\ell}(\bfx)|}{\prod_{i=1}^{m}\max_{j}
    |x_{i,j}|^{\ell d_{i}}}
\le \sup_{\bfx\in \D^{\bfn+\bfone}}{|f^{\ell}(\bfx)|},
\end{multline*}
where $\D=\{z\in \C: |z|\le 1\}$ is the unit disk.
Hence,   $\log(||F_{\ell} ||_{_{v_{\ell}(V)}}\Big)\le \ell \, \log||f
 ||_{\sup} $ because of the maximum modulus principle applied to $f$.
Besides, $\log(|\bfN-\bfone|+1)= O(\log(\ell))$ and we deduce from \eqref{eq:67} that
\begin{displaymath}
\wh_\bfb\Big(V\cdot \div(f)\Big) \le   \sum_{i=1}^m d_{i}\,
\wh_{\bfb+\bfe_i}(V)+\deg_\bfb(V) \log||f ||_{\sup} + O\bigg(\frac{
\log(\ell)}{\ell} \bigg).
\end{displaymath}
The result follows by letting $\ell\to \infty$.

\smallskip \eqref{item:43} This follows by adapting the proof of
Theorem~\ref{bezaritff}\eqref{item:25} to this setting without major
changes.

\smallskip \eqref{item:42} This is an immediate consequence of
\eqref{item:41} and \eqref{item:43}, together with
Theorem~\ref{bezgeom} and
Proposition~\ref{prop:5}\eqref{gradosparciales}.
\end{proof}

{The following results} can be proved by  adapting the arguments in
the  proofs of corollaries
\ref{cor:1} and \ref{desffcase} without major changes.

\begin{corollary} \label{cor:213} Let $V\subset \P^{\bfn}_{\Q}$ be a
  $\Q$-variety of pure dimension $r$ and $f\in
  \Z[\bfx_{1},\dots,\bfx_{m}]$ a multihomogeneous polynomial. Let $W$
  denote the union of the components of dimension $r-1$ of the
  intersection $V\cap V(f)$. Then
\begin{displaymath} [ W]_{_{\Z}} \le [ V]_{_{\Z}}\cdot [ V(f)]_{\sup}.
\end{displaymath}
In particular,
\begin{math}\wh_\bfb(W) \le
 \sum_{i=1}^m\deg_{\bfx_i}(f) \, \wh_{\bfb+\bfe_i}(V)+\log||f||_{\sup}\deg_\bfb(V)
 \end{math} for all $\bfb\in \N_{r}^{m}$.
\end{corollary}

\begin{corollary}\label{desratcase}
  Let $n\in \Z_{>0}$, $X\in Z_r^+(\P^n_\Q)$ and $f_{j}\in
  \Z[x_{0},\dots, x_{n}]$ be $s\le r$ homogeneous polynomials such
  that $ X\cdot \prod_{j=1}^{i-1}\div(f_{j})$ and $\div(f_{i})$
  intersect properly for all $1\le i\le s$. Then
 \begin{displaymath}
 \wh\bigg(X\cdot \prod_{j=1}^{s}\div(f_{j})\bigg)\leq \bigg(\prod_{i=1}^s\deg(f_{i})\bigg)\bigg(\wh(X)+\deg(X)
\bigg(\sum_{\ell=1}^s\frac{\log||f_{\ell}||_{\sup}}{\deg(f_{\ell})}\bigg)\bigg).
\end{displaymath}
\end{corollary}

\begin{remark} \label{rem:10} Theorem \ref{multihom} is one of the
  main reasons for considering canonical mixed heights instead of
  others. For instance, the fact that the canonical mixed heights of
  index $\bfc$ such that $c_{i}>n_{i}$ for some $i$ are zero is quite
  convenient in the applications.  Observe that the analogue of this
  statement for the Fubini-Study mixed heights does not hold: for
  instance, $\h_{n+1}(\P^{n})$ equals the $n$-th Stoll number, which
  is not zero.  This fact implies that the class in
  $A^{*}(\P^{\bfn}_{\Q};\Z)$ of a cycle $X$ contains the information
  about all canonical mixed heights of $X$. This information is
  necessary if one wants to express the arithmetic B\'ezout's
  inequality in Theorem~\ref{multihom}\eqref{item:41} in terms of
  elements in the extended Chow ring.
\end{remark}

We next show that canonical mixed heights are monotonic with respect
to linear projections. We keep the notation from Proposition~\ref{113}
and consider also the injective $\R$-linear map $\jmath:
A^{*}(\P^\bfl;\Z)\hookrightarrow A^{*}(\P^\bfn;\Z)$ defined by
$\jmath(P)= \bftheta^{\bfn-\bfl}P$.

\begin{proposition}\label{113alturasrat}
  Let $\pi:\P^{\bfn}_\Q\to \P^{\bfl}_\Q$ be the standard projection
  defined in \eqref{eq:58} and $X\in Z^+_{r}\Big(
  \P^\bfn_\Q\Big)$. Then
$$
\jmath \Big([\pi_{*}X]_{_{\Z}}\Big)\le [X]_{_{\Z}}.
$$
In particular, $\wh_{\bfc}(\pi_{*}X) \le \wh_{\bfc}(X)$ for all
$\bfc\in \N^{m}_{r+1}$.
\end{proposition}

\begin{proof}
  The statement is equivalent to $\deg_{\bfb}(\pi_{*}X) \le
  \deg_{\bfb}(X)$ and $\wh_{\bfc}(\pi_{*}X) \le \wh_{\bfc}(X)$ for all
  $\bfb,\bfc$.  Because of Proposition~\ref{113}, we only need to
  prove the latter inequality.

  Let $\ell\ge1$. By Proposition~\ref{formainicialproy},
  $\Res_{\bfe(\bfc)}(\pi_*[\ell]_*X)$ divides
  $\init_\prec(\Res_{\bfe(\bfc)}([\ell]_*X))$ in $\Z[\bfU]$.  Using
  the fact that the Mahler measure of a polynomial with integers
  coefficients is $\ge0$ and Lemma~\ref{lemm:9}\eqref{item:46}, we
  deduce that
\begin{displaymath}
\m(\Res_{\bfe(\bfc)}(\pi_*[\ell]_*X))
\le \m\Big(\init_\prec(\Res_{\bfe(\bfc)}([\ell]_*X)) \Big) \le
\m\Big(\Res_{\bfe(\bfc)}([\ell]_*X) \Big).
\end{displaymath}
From the estimates in \eqref{item:54} and
Proposition~\ref{prop:10}\eqref{item:33}, it follows that
\begin{align*}
&\h_\bfc(\pi_*[\ell]_{*}X)) = \ph (\Res_{\bfe(\bfc)}(\pi_*[\ell]_{*}X))
=  \m(\Res_{\bfe(\bfc)}(\pi_*[\ell]_{*}X)) + O(\ell^{r}), \\[1mm]
&\h_\bfc([\ell]_{*}X)) = \ph (\Res_{\bfe(\bfc)}([\ell]_{*}X))
=  \m(\Res_{\bfe(\bfc)}([\ell]_{*}X)) + O(\ell^{r}).
\end{align*}
Since $\pi$ commutes with $[\ell]$,
\begin{multline*}
\wh_\bfc(\pi_*(X)) = \lim_{\ell\to\infty} \frac{ \h_\bfc([\ell]_* \pi_*(X))}{\ell^{r+1}} =
\lim_{\ell\to\infty}  \frac{ \h_{\bfc}(\pi_*[\ell]_*X) }{\ell^{r+1}}
\le  \lim_{\ell\to\infty}\frac{\h_{\bfc}([\ell]_*X)}{\ell^{r+1}} = \wh_{\bfc}(X),
\end{multline*}
which completes the proof.
\end{proof}

\begin{example} \label{exm:3} We revisit the problem in
  Example~\ref{exm:5} of computing the pairs (eigenvalue, eigenvector)
  for the case of a matrix $M$ with entries in $\Z$.  We use the
  notations therein and furthermore we set
  $Z=\prod_{i=1}^{n}\div(f_{i})\in
  Z_{0}^{+}(\P^{1}\times\P^{n-1};\Q)$.  By Corollary \ref{desratcase},
$$[Z]_{_\Z} \le  \prod_{i=1}^n[ \div(f_i)]_{\sup}.$$
We have that $[Z]_{_\Z}=\wh_{(1,0)}(Z)\, \eta\, \theta_2^{n-1} +
\wh_{(0,1)}(Z)\, \eta\, \theta_1\, \theta_2^{n-2} + \deg(Z)\,
\theta_1\, \theta_2^{n-1}$ and that $ [ \div(f_i)]_{\sup} =
\log||f_i||_\sup\, \eta + \theta_1+\theta_2$.  By comparing
coefficients, it follows that
$$
\wh_{(1,0)}(Z)\le \sum_i\log||f_i||_\sup \quad, \quad
\wh_{(0,1)}(Z)\le (n-1)\sum_i\log||f_i||_\sup \quad, \quad \deg(Z) \le
n.
$$
Let $\pi_{1}$ and $\pi_{2}$ denote the projections from
$\P^{1}\times\P^{n-1}$ to the first and the second factor,
respectively.  We have that
\begin{displaymath}
Z_{1}:={(\pi_{1})_{*}}Z \in Z_{0}^{+}(\P^{1}_\Q) \quad, \quad
Z_{2}:={(\pi_{2})_{*}}Z \in Z_{0}^{+}(\P^{n-1}_\Q)
\end{displaymath}
are the cycles of eigenvalues and of eigenvectors of $M$,
respectively.  Applying Proposition~\ref{113alturasrat}, we deduce
that the heights of these cycles satisfy
 $$\wh(Z_{1})\le  \sum_i\log||f_i||_\sup\quad, \quad
\wh(Z_{2})\le
(n-1)\, \sum_i\log||f_i||_\sup.$$
{A straightforward application of the arithmetic B\'ezout inequality
for projective cycles (Corollary \ref{desratcase}) would have given
the much worse bound $2^{n-1}\, \sum_i\log||f_i||_\sup.$}
\end{example}

Next result gives the behavior of extended Chow rings and
classes with respect to products.

\begin{proposition}\label{altrodgenrat}
Let $m_i\in \Z_{>0}$ and $\bfn_{i}\in \N^{m_{i}}$ for $i=1,2$. Then
\begin{enumerate}
\item \label{item:39} $A^{*}(\P^{\bfn_{1}}\times \P^{\bfn_{2}};\Z) \simeq
  A^{*}(\P^{\bfn_{1}};\Z) \otimes_{\R[\eta]} A^{*}(\P^{\bfn_{2}};\Z)$.
  \smallskip
\item \label{item:40}  Let $X_{i}\in Z_{r_{i}}(\P^{\bfn_{i}}_\Q) $ for
  $i=1,2$. {The above isomorphism identifies
$[X_{1}\times X_{2}]_{_{\Z}}$ with $[X_{1}]_{_{\Z}}\otimes [X_{2}]_{_{\Z}}.$}
In particular, for $\bfc_{i}\in \N^{m_{i}}$ such that
$|\bfc_{1}|+|\bfc_{2}|=r_{1}+r_{2}+1$,
\begin{displaymath}
   \wh_{(\bfc_1,\bfc_2)}(X_{1}\times X_{2})=
\left\{\begin{array}{cl}
\deg_{\bfc_1}(X_{1})\,\wh_{\bfc_2}(X_{2}) & \text{ if } |\bfc_{1}|=r_{1},
|\bfc_{2}|=r_{2}+1, \\[1mm]
 \deg_{\bfc_2}(X_{2})\,\wh_{\bfc_1}(X_{1}) & \text{ if } |\bfc_{1}|=r_{1}+1,
|\bfc_{2}|=r_{2}, \\[1mm]
0 & \text{ otherwise.}
\end{array}\right.
\end{displaymath}
\end{enumerate}
\end{proposition}

\begin{proof}\

 \eqref{item:39} This is immediate from the definition of the
  extended Chow ring.

\smallskip
  \eqref{item:40} Consider the case
when  $|\bfc_1|=r_{1}+1,\, |\bfc_2|=r_{2}$.
By propositions~\ref{resultantesproductosciclos} and~\ref{prop:10}\eqref{item:33},
\begin{align*}
\h_\bfc([\ell]_*(X_{1}\times X_{2}))& = \ph(\Res_{\bfe(\bfc)}([\ell]_*X_{1}\times
 [\ell]_*X_{2}))\\ & =\ph(\Res_{\bfe(\bfc_1)}([\ell]_*X_{1})^{\deg_{\bfc_2}([\ell]_*X_{2})}) =
\ell^{r_{2}}\, \deg_{\bfc_2}(X_{2})\, \h_{\bfc_1}([\ell]_{*}X_{1}).
\end{align*}
Hence,
\begin{multline*}
\wh_\bfc(X_{1}\times X_{2}) = \lim_{\ell\to
  \infty}\ell^{-r_{1}-r_{2}-1}{\h_\bfc([\ell]_*(X_{1}\times
  X_{2}))}{}\\  = \deg_{\bfc_2}(X_{2}) \,
\lim_{\ell\to \infty}{\ell^{-r_{1}-1}}{\h_{\bfc_1}([\ell]_{*}X_{1})}=
\wh_{\bfc_1}(X_{1})\, \deg_{\bfc_2}(X_{2}).
\end{multline*}
The other cases can be proved  similarly. The equality of classes in
the extended Chow ring follows
from this  together with Proposition \ref{prop:3}.
\end{proof}

Finally, we compute the class in the extended Chow ring  of the ruled join of two projective varieties.
For $i=1,2$ consider the injective  $\R$-linear map
$\jmath_{i}:A^{*}(\P^{n_{i}}; \Z)\hookrightarrow A^{*}(\P^{n_{1}+n_{2}+1};\Z)$ defined by
$\theta^{l}\eta^{b}\mapsto \theta^{l}\eta^{b}$ for $0\le l\le  n_{i}$ and $b=0,1$.

 \begin{proposition} \label{prop:11}
Let $n_{i}\in\N$ and $X_{i} \in Z_{r_{i}}(\P^{n_{i}}_{\Q})$ for $i=1,2$. Then
\begin{displaymath}
  [X_{1}\#X_{2}]_{_{\Z}} =   \jmath_{1}([X_{1}]_{_{\Z}}) \cdot \jmath_{2}([X_{2}]_{_{\Z}}).
\end{displaymath}
In particular,
$\wh (X_{1}\#X_{2})= \deg(X_{1})\, \wh (X_{2})+\deg(X_{2})\,\wh (X_{1})$.
 \end{proposition}

 \begin{proof}
The equality of classes is equivalent to
$\deg(X_{1}\#X_{2})=\deg(X_{1})\, \deg(X_{2})$ and $\wh
(X_{1}\#X_{2})= \deg(X_{1})\, \wh (X_{2})+\deg(X_{2})\,\wh (X_{1})$.
The first one is \eqref{degjoin} while the second one  is
 \cite[Prop.~4.9(a)]{PhiSom08}.
\end{proof}

\section{The height of the implicit equation}\label{sec:height-impl-equat}

As an application of the results in the previous sections, we obtain
sharp bounds for the exponents and the coefficients of an equation
defining the closure of the image of an affine variety under a
rational map.  We consider separately {the cases $K$ a field of
functions and $K=\Q$.} Besides of their intrinsic interest, these
results play a central role in our treatment of the parametric and
arithmetic Nullstellens\"atze.

\medskip In the sequel, we will be {mostly} concerned with
affine varieties instead of multiprojective cycles.  Given $\bfn\in
\N^{m}$, we identify the affine space
\begin{displaymath}
\A^{\bfn}_{K}:=\A^{n_{1}}_{K}\times\cdots\times \A^{n_{m}}_{K}
\end{displaymath}
with the dense open subset $\P^{\bfn}_{K} \setminus
V(x_{1,0},\dots,x_{m,0})$.
This identification allows to transfer
notions and results from  $\P^{\bfn}_{K}$ to $\A^\bfn_{K}$. For instance,
subvarieties, cycles and divisors of~$\A^{\bfn}_{K}$ correspond to the
restriction of subvarieties, cycles and divisors of
$\P^{\bfn}_{K}$ to  this open subset. Thus, we can consider their mixed degrees
and heights, which will correspond to the analogous notions for their
closure  in the corresponding multiprojective space.
In particular, the degree and the
height of a subvariety of $\A^{n}_{K}$ are respectively defined as the
degree and the height of its closure in $\P^{n}_{K}$.

For $n\in \N$, the set of $K$-varieties of $\A^{n}_{K}$ is in
one-to-one correspondence with the set of radical ideals of
$K[x_{1},\dots, x_{n}]$.  For an affine $K$-variety $V \subset
\A^{n}_{K}$, we denote by $V(\ov K)$ the set of closed points of
$V_{\ov K}$.  It can be identified with the subset of $\ov K^{n}$
defined by $I(V)$.  A polynomial relation is said to \emph{hold on
  $V$} if it holds for every point of $V(\ov K)$.

Since $\Q$ is a perfect field, a  $\Q$-variety of $\A^{n}_{\Q}$ can be identified with
a subvariety of $\A^{n}(\Qbarra)$  defined over $\Q$, see Remark \ref{rem:4}.

\subsection{The function field case}\label{implicit}

Let $k$ be an arbitrary field and consider the groups of variables
$\bft=\{t_{1},\dots, t_{p}\}$ and $\bfx=\{x_{1},\dots, x_{n}\}$.
The following result is a parametric analogue of Perron's
theorem on the size of an equation of algebraic dependence for $r+1$
polynomials over a variety  of dimension $r$.

\begin{theorem}\label{perronpar}
  Let $V \subset \A^n_{_{k(\bft)}}$ be a $k(\bft)$-variety of pure
  dimension $r$ and $q_1,\dots,q_{r+1} $ $ \in k[\bft,\bfx]\setminus
  k[\bft]$.  Set $d_j=\deg_\bfx (q_j)$, $ h_j=\deg_\bft (q_j)$ for
  $1\le j\le r+1$ and write $\bfd=(d_1,\dots,d_{r+1})$, $\bfh=
  (h_1,\dots,h_{r+1})$. Then there exists
$$E=\sum_{\bfa\in \N^{r+1}}\,
\alpha_{\bfa }\bfy^\bfa \ \in k[\bft][y_1,\dots,y_{r+1}]\setminus
 \{0\}$$ satisfying $ E(q_1,\dots,q_{r+1})=0$ on $V$ and such
that,  for all $\bfa \in \Supp(E)$,
\begin{align*}
  \bullet & \ \langle \bfd,\bfa\rangle \le
  \bigg(\prod_{j=1}^{r+1}d_j\bigg)\deg(V),\\
  \bullet & \ \deg(\alpha_{\bfa}) + \langle \bfh,\bfa \rangle \le
  \bigg(\prod_{j=1}^{r+1}d_j\bigg) \bigg(\h(V) + \deg(V)
  \sum_{\ell=1}^{r+1}\frac{h_\ell}{d_\ell}\bigg).
\end{align*}
\end{theorem}

{Our proof below follows Jelonek's approach
in~\cite[Thm.~3.3]{Jelonek05} and heavily relies on the arithmetic
intersection theory developed in the previous sections.}



\begin{lemma} \label{relalturas} Let $V\subset \A_{_{k(\bft)}}^{n}$
   be a $k(\bft)$-variety of pure  dimension $r$ and $l\in \N$. Then
   \begin{displaymath}
\deg\Big(V\times  \A_{_{k(\bft)}}^{l}\Big)=\deg(V) \quad, \quad \h\Big(V\times  \A_{_{k(\bft)}}^{l}\Big)=      \h(V).
   \end{displaymath}
\end{lemma}

\begin{proof}
Consider the standard inclusions $\iota_{1}$, $\iota_{2}$ and
$\iota_{3}$ of $ \A^{n}$, $\A^{l-1}$ and $\A^{n+l}$ into $ \P^{n}$,
$\P^{l-1}$ and $\P^{n+l}$, respectively.
We can easily verify that $\iota_{3}(V\times \A^{l})=\iota_{1}(V)\# \iota_{2}(\A^{l-1})$.
Applying propositions \ref{prop:12} and \ref{prop:9}\eqref{item:18},
\begin{displaymath}
  [V\times  \A^{l}]_{_{k(\bft)}}=[V]_{_{k(\bft)}}\cdot
  [\A^{l-1}]_{_{k(\bft)}}= [V]_{_{k(\bft)}},
\end{displaymath}
which implies the  statement.
\end{proof}

\begin{proof}[Proof of Theorem~\ref{perronpar}.]
  Consider first the case when the map $$\psi: V \longrightarrow
  \A_{_{k(\bft)}}^{r+1}\quad, \quad \bfx\longmapsto (q_1(\bfx), \dots,
  q_{r+1}(\bfx))$$ is generically finite onto its image, that is, when
  the fiber of a generic point in $\Im(\psi)$ is finite.  In our
  setting, this implies {that $\ov{\Im(\psi)}$ is} a hypersurface.  Let
  $E\in k[\bft][\bfy]\setminus \{0\}$ be a primitive and squarefree
  polynomial defining this hypersurface.  Let $v, w_1,\dots, w_{r+1}$
  be a set of auxiliary variables and write
$$
P=E(w_{1}+v^{h_1}y_1^{d_1}, \dots,
w_{r+1}+v^{h_{r+1}}y_{r+1}^{d_{r+1}}) \in k[\bft, v][\bfw, \bfy].
$$
Hence,
\begin{equation}\label{eq10}
  \max_{\bfa} \{ \langle \bfd,\bfa\rangle\} =\deg_{\bfw,\bfy}(P)
  \quad, \quad
  \max_{\bfa} \{
  \deg(\alpha_\bfa ) + \langle \bfh,\bfa\rangle\} = \deg_{\bft,v}(P).
\end{equation}
We have that $E$ is primitive and squarefree as a polynomial in
$k[\bft,v][\bfw,\bfy]$. {Hence the same holds for $P$, since the map
defined by $y_{j}\mapsto w_{j}+v^{h_j}y_j^{d_j}$, $w_{j}\mapsto
y_{j}$ is the identity on $k[\bft,v]$ and an automorphism of
$k[\bft,v][\bfw,\bfy]$ which sends $P$ to $E$.}

\medskip
Write  $V(P)$ for the ${k(\bft,v)}$-hypersurface of
$\A^{2r+2}_{_{k(\bft,v)}}$ defined by $P$. By Proposition
\ref{prop:9}\eqref{item:15},
 \begin{equation}
   \label{eq:11}
   \deg_{\bfw,\bfy}(P)=\deg(V(P)) \quad, \quad
   \deg_{\bft,v}(P)= \h(V(P)).
 \end{equation}
We will bound both the degree and the height
 of this hypersurface.
Consider the following subvarieties of  $\A^{n+2r+2}_{_{k(\bft,v)}}$:
$$
G = (V\times \A^{2r+2}) \cap \bigcap_{j=1}^{r+1} V(y_j-q_j) \quad,
\quad G(\bfd, \bfh) = (V\times \A^{2r+2}) \cap \bigcap_{j=1}^{r+1}
V(w_j+v^{h_j}y_j^{d_j}-q_j).
$$
{Both varieties are of pure dimension~$2r+1$: indeed, each of them
  is defined by a set of $r+1$ polynomials which form a complete
  intersection over $V\times \A^{2r+2}$ as they depend on different
  variables $y_j$.}  Let $\rho:\A^{r+1}\to \A^{r+1}$ be the map
defined by $y_j\mapsto w_{j}+v^{h_j}y_j^{d_j}$ and $\pi$ the
projection $\A^{n+2r+2}\simeq \A^{n}\times \A^{2r+2}\to \A^{2r+2}$. We
have a commutative diagram
\begin{displaymath}
  \xymatrix{
    G(\bfd, \bfh)  \ar[rr]^{\Id_{\A^n\times \A^{r+1}}\times \rho} \ar[d]_\pi
    & \hspace{10mm}&G \ar[d]^\pi \\
    V(P)\ar[rr]_{\Id_{\A^{r+1}}\times \rho} &&
    V(E)}
\end{displaymath}
Both $G(\bfd,\bfh)\to G$ and $ G\to V(E)$ are generically finite since $\psi$ is generically finite, {and
  so this is also the case for} $\pi:G(\bfd,\bfh)\to V(P)$.
Proposition~\ref{113alturasff} then implies that
\begin{equation}
  \label{eq:12}
  \deg (V(P))\le  \deg(G(\bfd,\bfh))  \quad , \quad
  \h(V(P)) \le  \h(G(\bfd,\bfh)).
\end{equation}

Observe that $\deg_{\bfx,\bfw,\bfy}(v^{h_j}y_j^{d_j}+w_j-q_j)=d_j$
and $\deg_{\bft,v}(v^{h_j}y_j^{d_j}+w_j-q_j)=h_j$. Hence, since
$G(\bfd,\bfh)$ is an open set of a component  defined by the
homogenization of  these equations in $ \P^{n+2r+2}$,
\begin{align*}
  \deg(G(\bfd,\bfh)) &\le  \Bigg(\prod_{j=1}^{r+1} d_j\Bigg)  \deg(V\times\A^{2r+2})= \Bigg(\prod_{j=1}^{r+1} d_j\Bigg)  \deg(V),
  \\
  \h(G(\bfd,\bfh))
  &\le \Bigg(\prod_{j=1}^{r+1}d_j\Bigg) \Bigg( \h(V\times \A^{2r+2})+\deg(V\times \A^{2r+2})
  \sum_{\ell=1}^{r+1} \frac{h_\ell }{d_\ell }\Bigg) \\[0mm]
  &\le \Bigg(\prod_{j=1}^{r+1}d_j\Bigg) \Bigg( \h(V )+\deg(V)
  \sum_{\ell=1}^{r+1} \frac{h_\ell }{d_\ell }\Bigg),
\end{align*}
thanks to Corollary \ref{cor:1} and Lemma \ref{relalturas}. The
statements follows from these bounds, together
with~(\ref{eq10}),~(\ref{eq:11}) and~(\ref{eq:12}).

\medskip The general case reduces to the generically finite one by
a deformation argument. Choose  variables $x_{i_1}, \dots,
x_{i_{r+1}}$ among those in the group $\bfx$ in such a way that the
projection $V\to \A_{_{k(\bft)}}^{r+1}, \bfx\mapsto (x_{i_1}, \dots,
x_{i_{r+1}})$ is generically finite onto its image.  Adding a
further variable $z$, we consider the map
$$
V_{{k(z,\bft)}}\longrightarrow \A^{r+1}_{_{k(z,\bft)}} \quad,
\quad \bfx\longmapsto\left(q_1(\bfx)+z \, x_{i_1}
  ,\dots,q_{r+1}(\bfx)+z \, x_{i_{r+1}}\right).
$$
It is also generically finite onto its image.
Thus, we are in the hypothesis of the previous case with respect to
the base ring $k(z)[\bft]$. We deduce that there is a polynomial $\wt E=\sum_{\bfa}\, \wt \alpha_{\bfa
}\bfy^\bfa \in k(z)[\bft][\bfy]\setminus \{0\}$ defining the closure
of the image of this map and satisfying, for all $\bfa\in \Supp( \wt E)$,
$$\langle \bfd,\bfa\rangle \le
\Bigg(\prod_{j=1}^{r+1}d_j\Bigg)\deg(V) \quad, \quad \deg_\bft(\wt
\alpha_{\bfa})+ \langle \bfh,\bfa \rangle \le
\Bigg(\prod_{j=1}^{r+1}d_j\Bigg) \Bigg(\h(V) + \deg(V)
\sum_{\ell=1}^{r+1}\frac{h_\ell}{d_\ell}\Bigg).
$$
After multiplying by a suitable non-zero polynomial in $z$, we can
assume without loss of generality that $\wt E$ lies in $k[z][\bft,
\bfy]$ and that it is primitive as a polynomial in the variables
$\bft,\bfy$. Set $E=\wt E(0, \bft,\bfy)\in k[\bft, \bfy]
\setminus\{0\}$.  We have that $\wt E(q_1+z \, x_{i_1}
,\dots,q_{r+1}+z \, x_{i_{r+1}}) \in I(V)\otimes k(\bft)[z]$ and so
$E(\bfq)\in I(V)$ or, equivalently, $E(\bfq)=0$ on $V$. We deduce
that $E=0$ is an equation of algebraic dependence for $q_1,\dots,
q_{r+1}$ which satisfies the same bounds as $\wt E$.
\end{proof}

For $V=\A_{_{k(\bft)}}^n$ we have $r=n$, $\deg(V)=1$ and $\h(V)=0$.
The above result gives the bounds
\begin{equation} \label{eqn:1} \langle \bfd,\bfa\rangle  \le
  \prod_{j=1}^{n+1}d_j \quad, \quad \deg(\alpha_{\bfa})+ \langle
  \bfh,\bfa \rangle \le  \sum_{\ell=1}^{n+1} \bigg( \prod_{j\ne
    \ell} d_j\bigg) h_\ell
\end{equation}
for the $\bfy$-degree and the $\bft$-degree of an equation of
algebraic dependence for $n+1$ polynomials
in $k[\bft,\bfx]$ of $\bfx$-degree $d_j\ge 1$ and $\bft$-degree~$h_j$.

\begin{example}\label{ejemplopar}
  Let $d_1,d_2\in \N$ be coprime integers and set $q_j=g_j x^{d_j}-1
  \in k[t,x]$ for a generic univariate polynomial $g_j\in k[t]$ of
  degree $h_j$, $j=1,2$.  The implicit equation of the closure of the image of
$$\A^1_{_{k(t)}}\longrightarrow \A^2_{_{k(t)}} \quad,\quad
  x\longmapsto(q_1(x), q_2(x)) $$ is  $ g_2^{d_1}
  (y_1+1)^{d_2} - g_1^{d_2}(y_2+1)^{d_1} =0$.
  The bounds~(\ref{eqn:1}) are optimal in this case.
\end{example}

Theorem~\ref{perronpar} can be regarded as an estimate for the Newton
polytope of the equation~$E$: if we write $E=\sum_{\bfa, \bfc}
\gamma_{\bfa, \bfc} \bft^\bfc\bfy^\bfa\in k[\bft,\bfy]$ with
$\gamma_{\bfa,\bfc} \in k$, the corresponding Newton polytope is the
convex hull
$$
\newton(E) =\Conv\Big\{ (\bfa,\bfc): \gamma_{\bfa,\bfc} \ne 0 \Big\}
\subset\R^{r+1}\times\R^{p}.
$$
Theorem~\ref{perronpar} is equivalent to the statement that
$\newton(E)$ is contained in the intersection of the non-negative
orthant with the two
half-spaces defined by the inequalities
$$
\langle \bfd,\bfa\rangle \le \Bigg(\prod_{j=1}^{r+1}d_j\Bigg)\deg(V)
\quad, \quad \langle \bfone, \bfc\rangle + \langle \bfh,\bfa \rangle
\le \Bigg(\prod_{j=1}^{r+1}d_j\Bigg) \Bigg(\h(V) + \deg(V)
\sum_{\ell=1}^{r+1}\frac{h_\ell}{d_\ell}\Bigg),
$$
where $\bfone=(1,\dots, 1)\in \N^{p}$.
Indeed, it can be shown that, for $V=\A^n_{_{k(\bft)}}$ and generic polynomials $q_j$ of
$\bfx$-degree $d_j\ge 1$ and $\bft$-degree $h_j$, the Newton polytope
of $E$ coincides with the set cut out by these inequalities.

\medskip
For instance, consider in Example~\ref{ejemplopar} the case when
$d_2h_1 \ge d_1h_2$.
Then, the Newton polytope of $E$ is the convex hull of the
points $(0,0,0), (d_2,0,0), (d_2, 0, d_1h_2)$, $(0,0,d_2h_1), (0,d_1,
d_2h_1), (0,d_1,0)$, as shown in the  figure:
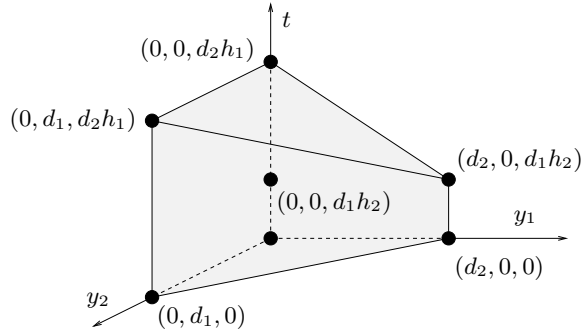
\begin{figure}[htbp]
 \input politopo.pstex_t \vspace{-11mm}
\end{figure}

It coincides with the subset of $\R^3$ cut out by the
inequalities~(\ref{eqn:1}), namely
$$
\Big\{(a_1,a_2,c): \quad a_1,a_2,c\ge0, \quad d_1a_1+d_2a_2 \le
d_1d_2, \quad c+h_1a_1+h_2a_2 \le d_1h_2+d_2h_1\Big\}.
$$

\begin{example}\label{ejemploeliptico}
  Consider the elliptic curve $C=V\Big((t+1)x_1^3+x_1^2-x_2^2\Big)
  \subset \A^2_{_{k(t)}}$ and the polynomials $q_1=x_1+(t+1)x_2-1,
  q_2=x_1x_2+(t-1)x_2^2+t\in k[t,x_1,x_2]$.  The implicit equation of
  {the closure of the image of}
$$
C\longrightarrow\A^2_{_{k(\bft)}}\quad, \quad (x_1,x_2)\longmapsto
\Big(q_1(x_1,x_2),q_2(x_1,x_2)\Big)
$$
is defined by a polynomial $E\in k[t,y_1,y_2]$ with 138 terms. Its
Newton polytope is the polytope $ \Conv\Big( (0,0,0), (6,0,0),
(6,0,5), (0,0,11), (0,3,8), (0,3,0) \Big) \subset \R^3: $
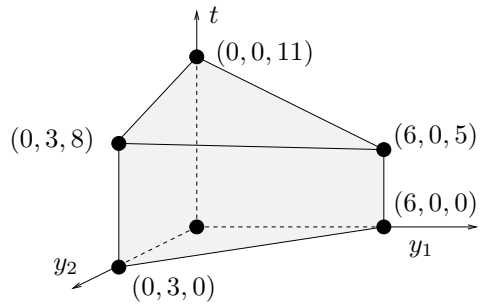
\vspace{-1mm}
\begin{figure}[htbp]
 \input eliptica.pstex_t
\vspace{-12mm}
\end{figure}

{It also} coincides with the polytope cut out by the inequalities in
the parametric Perron's theorem: we have that $\deg(C)=3$ and $\h(C)=1$
while $\deg_{\bfx}(q_1)=1$, $\deg_{\bfx}(q_2)=2$ and $\deg_t(q_i)=1$, $i=1,2$. Theorem~\ref{perronpar} implies the
inclusion of polytopes
$$
\newton(E) \subset \Big\{(a_1,a_2,c): \quad a_1,a_2,c\ge0, \quad
a_1+2a_2 \le 6, \quad c+a_1+a_2 \le 11\Big\},
$$
which turns out to be an equality.
\end{example}

The following result is an upper bound for the degree and
the height of the implicit equation of a hypersurface defined as the
closure of the image of a general {rational} map.  These estimates
are not used in the proof of the arithmetic Nullstellens\"atze.
Nevertheless, we include them because they may have some
independent interest.

Recall that a rational map $\rho: V\dashrightarrow
\A_{_{k(\bft)}}^{r+1}$ is defined by quotients of polynomials whose denominators do
not vanish identically on any of the components of $V$. It is
{\it generically finite onto its image} if the fiber of a generic
point in $\Im(\rho)$ is finite.  In our setting, this implies
that the closure of its image is a hypersurface.

\begin{theorem}\label{racionalpar}
  Let $V \subset \A^n_{_{k(\bft)}}$ be a $k(\bft)$-variety of
  pure dimension $r$ and
$$
\psi: V \dashrightarrow \A^{r+1}_{_{k(\bft)}}\quad, \quad
\bfx\longmapsto \Bigg( \frac{q_1}{p_1}(\bfx)
,\dots,\frac{q_{r+1}}{p_{r+1}}(\bfx) \Bigg)
$$
a rational map, generically finite onto its image, defined by
polynomials $q_j, p_j\in k[\bft,\bfx]$ such that $q_j/p_j \notin
k(\bft)$.
Let $E \in k[\bft,\bfy]$ be a primitive and squarefree polynomial
defining ${\ov{\Im(\psi)}}$.
Set $d_j=\max\{\deg_\bfx (q_j),
\deg_\bfx(p_j)\}$ and $h_j=\max\{\deg_\bft (q_j),
\deg_\bft(p_j)\}$ for  $1\le j\le r+1$. Then
\begin{align*}
  \bullet & \ \deg_{y_i} (E) \le
  \bigg(\prod_{j\ne i}d_j\bigg)\deg(V) \quad \mbox{for } 1\le i\le r+1,\\[-1mm]
  \bullet & \ \deg_\bft (E) \le  \bigg(\prod_{j=1}^{r+1}d_j\bigg)
  \bigg(\h(V) + \deg(V)
  \sum_{\ell=1}^{r+1}\frac{h_\ell}{d_\ell}\bigg).
\end{align*}
\end{theorem}

\begin{proof} Let $U=V\setminus \bigcup_{j=1}^{r+1}V(p_j)$ be the
  dense open subset of $V$ where $\psi$ is defined. The graph of
  $\psi$ is
$$
\{(\bfx, \bfy): \ \bfx \in U, \ y_jp_j(\bfx)=q_j(\bfx) \mbox{ for }
1\le j\le r+1\} \subset U\times \A^{r+1}(\ov{k(\bft)}).
$$
Let $G$ be the closure of this set in $V\times \A^{r+1}_{_{k(\bft)}}$.
The equations $ y_jp_j=q_j$ intersect properly on $ U\times
\A^{r+1}_{_{k(\bft)}}$ because they depend on different variables
$y_j$.  Hence, $G$ is an equidimensional variety of dimension $r$, and
the projection $V\times \A^{r+1} \to \A^{r+1}$ induces a generically
finite map between $G$ and $V(E)$.

We consider mixed degrees and heights of $G$ and $V(E)$ with respect
to the inclusions
$$V\times \A^{r+1}\hookrightarrow \P^n\times (\P^1)^{r+1} \quad ,\quad
\A^{r+1}\hookrightarrow (\P^1)^{r+1}.
$$
Set $\bfone=(1,\dots, 1)\in \N^{r+1}$ and let $\bfe_i$ denote the
$i$-th vector of the standard basis of~$\R^{r+1}$.  By Proposition
\ref{prop:9}\eqref{item:15},
\begin{equation}
  \label{eq:42}
  \deg_{y_i} (E) = \deg_{\bfone-\bfe_i}(V(E)) \quad, \quad
\deg_\bft(E)= \h_\bfone(V(E)).
\end{equation}
Proposition~\ref{113alturasff} applied to the projection $
\P^n\times (\P^1)^{r+1} \to (\P^1)^{r+1}$ implies that
\begin{equation}
  \label{eq:43}
\deg_{\bfone-\bfe_i}(V(E))
\le \deg_{0,\bfone-\bfe_i}(G) \quad, \quad \h_\bfone(V(E))
\le \h_{0,\bfone}(G).
\end{equation}

Let $[ G\ck\in
\Z[\eta,\theta_0,\bftheta]/(\eta^2,\theta_0^{n+1},\theta_1^2,\dots,\theta_{r+1}^2)$
be the class of  $G$ in the extended Chow ring of $\P^n\times
(\P^1)^{r+1}$.
Observe that $G$ is contained in the part of dimension $r$ of the
intersection
$$(V \times \A^{r+1})\cap \bigcap_{j=1}^{r+1} V(y_j p_j-q_j).
$$
We will give an upper bound for the class of $G$ by applying
Corollary~\ref{cor:1} recursively to the variety $V\times \A^{r+1}$
and the polynomials $g_j=y_jp_j-q_j$ for $j=1,\ldots, r+1$. By
propositions~\ref{altrodgenff} and \ref{prop:9}\eqref{item:18},
\begin{displaymath}
  [V\times \A^{r+1}\ck =[V\ck \otimes [\A^{1}]_{_{k[\bft]}}^{\otimes (r+1)}=
  [V]_{_{k[\bft]}}= \h(V)\eta\theta_0^{n-r-1} +\deg(V)\theta_0^{n-r}.
\end{displaymath}
Hence,
\begin{align*}
\nonumber  [ G\ck &\le  [V\times \A^{r+1}\ck
  \prod_{j=1}^{r+1}\bigg(\deg_\bft(g_j )\eta + \deg_\bfx(g_j)\theta_0+
\sum_{\ell=1}^{r+1}\deg_{y_\ell}(g_j )\theta_\ell\bigg) \\
  &=  \Big(\h(V)\eta\theta_0^{n-r-1} +\deg(V)\theta_0^{n-r}\Big)
  \,\prod_{j=1}^{r+1}(h_j\eta + d_j\theta_0+\theta_j).
\end{align*}
Observe that $\h_{0,\bfone}(G)$ and $\deg_{0,\bfone-\bfe_i}(G)$ are the
coefficients in  $[ G\ck$ of the monomials~$\eta\,\theta_0^n$ and $\theta_0^n\,\theta_i$, respectively.
Therefore
$$
\deg_{0,\bfone-\bfe_i}(G) \le \bigg(\prod_{j\ne i}d_j\bigg)\deg(V)\quad
,\quad  \h_{0,\bfone}(G) \le\bigg(\prod_{j=1}^{r+1}d_j\bigg) \bigg( \h(V) +
\deg(V)\sum_{\ell=1}^{r+1}\frac{h_\ell}{d_{\ell}}\bigg).
$$
The statement follows from these estimates, together with~\eqref{eq:42}
and~\eqref{eq:43}.
\end{proof}

For $V=\A^n_{_{k(\bft)}}$, the above result gives the bounds
\begin{equation} \label{eqn:2} \deg_{y_i} (E) \le \prod_{j\ne i}d_j
  \quad , \quad \deg_\bft (E) \le \sum_{\ell=1}^{n+1} \bigg(
  \prod_{j\ne \ell} d_j\bigg) h_\ell.
\end{equation}
It can be shown that, for generic polynomials
of $\bfx$-degree $d_j$ and $\bft$-degree $h_j$, the Newton
polytope of $E$ coincides with the subset of
$\R^{r+1}\times \R^{p}$ cut out by these inequalities.

\begin{example}\label{ejemploracionalpar}
  Let $d_1,d_2\in \N$ be coprime integers and $g_i\in k[t]$
  univariate polynomials of degree $h_i$, $i=1,2$.  Set
$$
\frac{q_1}{p_1}=\frac{x^{d_1}}{(x+1)^{d_1}g_1}-1 \quad, \quad
\frac{q_2}{p_2}= \frac{(x+1)^{d_2}}{x^{d_2}g_2}-1.
$$
The implicit equation of the {closure of the image of the map}
$$\A^1_{_{k(t)}}\dashrightarrow \A^2_{_{k(t)}} \quad , \quad
x\longmapsto\bigg(\frac{q_1}{p_1}(x), \frac{q_2}{p_2}(x)\bigg) $$ is
given by the polynomial $ E= {g_1}^{d_2}{g_2}^{d_1}(y_1+1)^{d_2}(y_2+1)^{d_1}-1\in k[t,
y_1,y_2] $. Theorem~\ref{racionalpar} is optimal in this case, since
$ \deg_{y_1}(E)=d_2$, $\deg_{y_2}(E)=d_1$ and $
\deg_t(E)=d_1h_2+d_2h_1$. Moreover, we can check that the Newton
polytope of $E$ coincides with the set cut out by the inequalities
in (\ref{eqn:2}).
\end{example}

The previous results can be extended to polynomials depending on
\emph{groups} of parameters.
For the sequel, we will need the  multiparametric version of
Theorem~\ref{perronpar} that we state below.
Let $\bft_l=\{t_{l,1},\dots,t_{l,p_l}\}$ be a group of {variables} for
$1\le l\le m$ and set $\bft=\{\bft_{1},\dots, \bft_{m}\}$.
For each $1\le l\le m$, write
  $k_l=k(\bft_1,\dots,\bft_{l-1},\bft_{l+1},\dots,\bft_m)$ and observe
  that $ k_{l}(\bft_{l})=k(\bft)$.
Hence, for a given projective
$k(\bft)$-variety $V$, we can consider its height with respect to the
base ring  $k_{l}[\bft_{l}]$. We denote this height $\h_{\bft_{l}}(V)$.

\begin{corollary}\label{implicitgroups}
  Let $V \subset \A^n_{k(\bft)}$ be a $k(\bft)$-variety  of pure dimension
  $r$ and  $q_1,\dots,q_{r+1}\in
  k[\bft][\bfx]\setminus k[\bft]$ such that the map
$$
\psi: V\longrightarrow \A^{r+1} \quad, \quad \bfx\longmapsto
(q_1(\bfx), \dots, q_{r+1}(\bfx))
$$
is generically finite onto its image. Let $E=\sum_{\bfa\in
\N^{r+1}}\, \alpha_{\bfa }\bfy^\bfa \ \in
k[\bft][y_1,\dots,y_{r+1}]$ be a primitive and squarefree polynomial
defining  ${\ov{\Im(\psi)}}$.
Set $d_j=\deg_\bfx (q_j)$, $h_{l,j}=\deg_{\bft_l} (q_j)$ for
$1\le j\le r+1$, $1\le l\le m$ and write
$\bfd=(d_1,\dots,d_{r+1})$, $ \bfh_l=
(h_{l,1},\dots,h_{l,r+1})$. Then, for all $\bfa\in \Supp(E)$,
\begin{align*}
  \bullet & \ \langle \bfd,\bfa\rangle  \le
  \Bigg(\prod_{j=1}^{r+1}d_j\Bigg)\deg(V),\\[-3mm]
  \bullet & \ \deg_{\bft_l}(\alpha_{\bfa}) + \langle \bfh_l,\bfa
  \rangle \le  \Bigg(\prod_{j=1}^{r+1}d_j\Bigg)\Bigg( \h_{\bft_l}(V)+ \deg(V)
  \sum_{\ell=1}^{r+1}\frac{h_{l,\ell}}{d_\ell}\Bigg) \quad \mbox{ for }  1\le l\le m.
\end{align*}
\end{corollary}

\begin{proof} Observe that $E$ is primitive and squarefree as an element of
  $k_{l}[\bft_{l}][\bfy]$.
The result then follows from Theorem~\ref{perronpar}
applied to the field $k_l$ and the group of variables  $\bft_l$.
\end{proof}

\subsection{The rational case}\label{sec:rational-case}

We now turn our attention to the problem of estimating the size of the implicit
equation {for a rational map} defined over~$\Q$.
It will be convenient to consider a more general situation where the
input polynomials depend on groups of parameters.
As in the end of the previous section, we set  $\bft_l=\{t_{l,1},\dots,t_{l,p_l}\}$  for
$1\le l\le m$ and  $\bft=\{\bft_{1},\dots, \bft_{m}\}$.

\begin{theorem}\label{implicitmultiparZ}
  Let $V \subset \A^n_{\Q}$ be a $\Q$-variety  of pure dimension
  $r$ and $\ q_1,\dots,q_{r+1}\in
  \Z[\bft][\bfx]\setminus \Z[\bft]$ such that the map
$$
\psi: V_{_{\Q(\bft)}}\longrightarrow \A^{r+1}_{_{\Q(\bft)}}\quad, \quad \bfx\longmapsto
(q_1(\bfx), \dots, q_{r+1}(\bfx))
$$
is generically finite onto its image. Let   $E=\sum_{\bfa\in
\N^{r+1}}\, \alpha_{\bfa }\bfy^\bfa \ \in
\Z[\bft][y_1,\dots,y_{r+1}]$ be a primitive and squarefree polynomial
defining {$\ov{\Im(\psi)}$}
Set $d_j=\deg_\bfx (q_j)$, $\delta_{l,j}=\deg_{\bft_l} (q_j),\, h_{j}=\h (q_j)$ for $1\le j\le r+1$, $1\le l\le m$ and write
$\bfd=(d_1,\dots,d_{r+1})$,
$\bfdelta_l=(\delta_{l,1},\dots,\delta_{l,r+1})$, $\bfh=
(h_{1},\dots,h_{r+1})$. Then, for all $\bfa\in \Supp(E)$,
\begin{align*}
  \bullet & \ \langle \bfd,\bfa\rangle  \le
  \bigg(\prod_{j=1}^{r+1}d_j\bigg)\deg(V),\\[-3mm]
  \bullet & \ \deg_{\bft_l}(\alpha_{\bfa}) + \langle \bfdelta_l,\bfa
  \rangle \le  \bigg(\prod_{j=1}^{r+1}d_j\bigg) \deg(V)
  \sum_{\ell=1}^{r+1}\frac{\delta_{l,\ell}}{d_\ell} \quad \mbox{ for }  1\le l\le m, \\[-1mm]
  \bullet & \ \h(\alpha_{\bfa}) + \langle \bfh,\bfa \rangle \le
  \bigg(\prod_{j=1}^{r+1}d_j\bigg) \Bigg(\wh(V) + \deg(V) \Bigg( \log(r+2)\\[-2mm]
&  \hspace{37mm}+\sum_{\ell=1}^{r+1}\frac{1}{d_\ell}\bigg( h_\ell
  +\log\Big(\#\Supp(q_\ell)+2\Big)+\sum_{l=1}^m
  \delta_{l,\ell}\log(p_l+1) \bigg) \Bigg)\Bigg).
\end{align*}
\end{theorem}

We need the following lemma for the proof of this result. It is the analogue for $\Q$ of Lemma
\ref{relalturas} and can be proved using  the same arguments and Proposition \ref{prop:11} instead of
 \ref{prop:12}.

\begin{lemma} \label{lemm:7}
Let $V\subset \A^{n}_{\Q}$
   be a $\Q$-variety of pure  dimension $r$ and $l\in \N$. Then
   \begin{displaymath}
\deg(V\times  \A^{l}_{\Q})=\deg(V) \quad, \quad \wh(V\times  \A^{l}_{\Q})=
\wh(V).
   \end{displaymath}
\end{lemma}

\begin{proof}[Proof of Theorem \ref{implicitmultiparZ}]
By
 Proposition \ref{prop:9}\eqref{item:18}, the height $h_{\bft_{l}}$ of $V_{_{\Q(\bft)}}$ with respect to the base ring
$\Q(\bft_1,\dots,\bft_{l-1},\bft_{l+1},\dots,\bft_m)[\bft_{l}]$ is
zero.
Then, the bounds for the degrees are a direct
  consequence of Corollary~\ref{implicitgroups}.
  For the height, the proof follows closely the lines of that  of
  Theorem~\ref{perronpar} with~$\Z$ instead of $k[\bft]$. We will
  avoid repeating the same verifications when they follow {\it mutatis
    mutandis} the parametric case.

\medskip
Let $w_1,\dots,w_{r+1}$ be a group of variables and consider the polynomial
$$
P=E\big(w_{1}+{H_{1}}y_1,\ldots,w_{r+1}+{H_{r+1}}y_{r+1}\big)\in
\Z[\bft,\bfw,\bfy],
$$
{with  $H_{j}=\e^{h_{j}}\in \N$}. It verifies
\begin{equation}
  \label{eq:21}
  \max_{\bfa} \{ \h(\alpha_\bfa) + \langle \bfh,\bfa\rangle\} = \h(P(\bfzero,\bfy)).
\end{equation}
Applying successively lemmas \ref{lemm:11}(\ref{krso01}) and
\ref{lemm:9}\eqref{item:46},
\begin{align}  \label{eq:3}
\nonumber   \h(P(\bfzero,\bfy)) &\le \m(P(\bfzero,\bfy)) +
  \deg_{\bfy}(P(\bfzero,\bfy))\log(r+2)+
  \sum_{l=1}^m\deg_{\bft_l}(P(\bfzero,\bfy))\log(p_l+1)\\[-2mm]
&\le \m(P) +
  \deg_{\bfy}(E)\log(r+2)+
  \sum_{l=1}^m\deg_{\bft_l}(E)\log(p_l+1).
\end{align}
The polynomial $P$ is primitive and squarefree. {Hence, it} gives a
defining equation for the $\Q$-hypersurface $ V(P)\subset \A^\bfp_{\Q} \times
\A^{2r+2}_{\Q}$, where $\bfp:=(p_{1},\dots, p_{m})$. Considering
the standard inclusion of this affine space into $\P^\bfp_{\Q}\times
\P^{2r+2}_{\Q}$,  Proposition~\ref{mahler}\eqref{item:37}
implies
\begin{equation}\label{b2.9}
  \m(P)=\wh_{\bfp, 2r+2}(V(P)).
\end{equation}
Consider the variety $$ G(\bfd, \bfh) = (\A^{\bfp}\times V\times
\A^{2r+2}) \cap \bigcap_{j=1}^{r+1} V(w_{j}+{H_{j}}y_j-q_j)
\subset \A^\bfp\times \A^{n+2r+2}.$$ As in the proof of
Theorem~\ref{perronpar}, we can verify that {it is} of
pure dimension $|\bfp|+2r+1$ and that the projection
$\A^\bfp\times \A^{n+2r+2}\to \A^\bfp\times \A^{2r+2}$ induces a
generically finite map $\pi: G(\bfd, \bfh)\to V(P)$. Considering
mixed heights with respect to the standard inclusion $ \A^\bfp\times
\A^{n+2r+2}\hookrightarrow \P^\bfp\times \P^{n+2r+2}$,
Proposition~\ref{113alturasrat} implies
\begin{equation}\label{b2.5}
  \wh_{\bfp, 2r+2}(V(P))\leq \wh_{\bfp,2r+2}\big(G(\bfd,\bfh)\big).
\end{equation}

{The closure of $\A^{\bfp}\times V\times \A^{2r+2}$ is
$\P^{\bfp}\times \iota (V\times \A^{2r+2})$, where $\iota$ denotes
the standard inclusion $\A^{n+2r+2}\hookrightarrow \P^{n+2r+2}$.
We will consider classes in the extended Chow ring
\begin{displaymath}
{A^{*}}(\P^\bfp\times \P^{n+2r+2};\Z)=
\R[\eta,\bftheta,\zeta]/(\eta^2,\theta_1^{p_1+1},\dots,\theta_l^{p_l+1},\zeta^{n+2r+3}).
\end{displaymath}
With this notation, }propositions \ref{altrodgenrat}\eqref{item:40} and
\ref{prop:10}\eqref{item:35} together with Lemma~\ref{lemm:7} imply
\begin{displaymath}
  [\A^{\bfp}\times V\times
\A^{2r+2}\cz =   [\A^{\bfp}\cz \otimes [V\times \A^{2r+2}\cz = [V\cz =
\wh(V)\eta\zeta^{n-r-1}+\deg(V)\zeta^{n-r}.
\end{displaymath}
Write $g_{j}=w_{j}+{H_{j}}y_j-q_j\in\Z[\bft,\bfx,\bfw,\bfy]$ and
consider the class  associated to its sup-norm as defined in
\eqref{clasedef}:
$$
 [g_j]_\sup  =  \log||g_j||_{\sup} \eta + \sum_{l=1}^m\deg_{\bft_l}(g_j)\theta_l
  +\deg_{\bfx,\bfw,\bfy}(g_j)
  \zeta   = \log||g_j||_{\sup} \eta + \sum_{l=1}^m\delta_{l,j}\theta_l +d_{j}\zeta.
$$
The divisors defined by these polynomials intersect $\A^{\bfp}\times
V\times \A^{2r+2}$ properly. Applying recursively
Corollary~\ref{cor:213},
\begin{align*}
\nonumber  [G(\bfd,\bfh)]_{_{\Z}}&\le [\A^{\bfp}\times V\times
\A^{2r+2}\cz  \prod_{j=1}^{r+1}[g_j]_\sup \\
&=  \Big( \wh(V)\eta\zeta^{n-r-1}+\deg(V)\zeta^{n-r}\Big)
 \prod_{j=1}^{r+1}\Big(\log||g_j||_{\sup} \eta +
 \sum_{l=1}^m\delta_{l,j}\theta_l +d_j\zeta\Big).
\end{align*}
The mixed height $\wh_{\bfp,2r+2}\big(G(\bfd,\bfh)\big)$ is the
coefficient of  $\eta\,\zeta^n$ in $[ G(\bfd,\bfh)\cz$. Furthermore,
 $\log||g_{\ell}||_{\sup}\le
h_\ell +\log\Big(\#\Supp(q_\ell)+2\Big)$ by
Lemma~\ref{lemm:11}\eqref{item:47}. The above
{inequality} implies
 \begin{equation*}
   \wh_{\bfp,2r+2}\big(G(\bfd,\bfh)\big)
   \le
   \Bigg(\prod_{j=1}^{r+1}d_j\Bigg)\Bigg(\wh(V)+\deg(V)\sum_{\ell=1}^{r+1}\frac{h_\ell +\log\Big(\#\Supp(q_\ell)+2\Big)}{d_\ell}\Bigg).
 \end{equation*}
 The statement follows from this inequality together with
 (\ref{eq:21}), \eqref{eq:3}, \eqref{b2.9}, (\ref{b2.5}), and the
 already considered bounds for the partial degrees of $E$.
\end{proof}

Theorem~\ref{PerronIntro} in the introduction follows from the case
$m=0$ of this result and the inequality $\#\Supp(q_\ell)+2\le (n+3)^{d_\ell}$.

\begin{remark} \label{rem:2}
Theorem~\ref{PerronIntro} can be regarded as a first bound for the \emph{extended Newton
polytope} of the implicit equation, defined as the convex envelope of
the subset
$$\{(\bfa, \lambda) : \bfa\in \Supp(E),\,0\leq \lambda\leq
\h(\alpha_\bfa) \}\subset \R^{r+1}\times\R.$$

Indeed, it would be interesting to have a better understanding of this
``arithmetic'' polytope  in terms of finer invariants of the input
polynomials $q_j$ like for instance their extended Newton polytope,
instead of just their degree and height.
\end{remark}

\begin{example}\label{ejemploarith} Let $d_1,d_2\in \N$ and
  $H_1,H_2\in \N$ be two pairs of coprime integers and set $q_j=H_j
  x^{d_j} \in \Z[x]$ for $j=1,2$.  The implicit equation of the image
  of the map $$\A^1_{\Q} \longrightarrow \A^2_{\Q} \quad , \quad
  x\longmapsto (q_1(x), q_2(x)) $$ is  given by the polynomial $ E= H_2^{d_1}
  y_1^{d_2} - H_1^{d_2}y_2^{d_1} \in \Z[y_1,y_2].  $
Then, $\max_{(a_1,a_2)} \{d_1a_1+d_2a_2\} = d_1d_2$ and
$$
\max_{(a_1,a_2) }
\{\log|\coeff_{(a_1,a_2)}(E)| + h_1a_1+h_2a_2 \}= d_1h_2+d_2h_1,
$$
with $h_j=\log(H_j)$.
Hence, the bounds in Theorem~\ref{PerronIntro}  are optimal in this
example, up to a term of size~$O(d_1d_2)$.
\end{example}

Using a deformation argument,
we can extend Theorem~\ref{implicitmultiparZ} to the  case
when the map is not generically finite.
For simplicity, we will only state this result for polynomials not depending on
parameters.

\begin{corollary}\label{corperronmultiparZ}
  Let $V \subset \A^n_{\Q}$ be a $\Q$-variety of pure dimension
  $r$ and $q_1,\dots,q_{r+1}\in
  \Z[\bfx]\setminus \Z$.
Set $d_j=\deg (q_j)$ and $h_{j}=\h (q_j)$ for $1\le j\le r+1$ and
write $\bfd=(d_1,\dots,d_{r+1})$, $\bfh= (h_{1},\dots,h_{r+1})$.
Then there exists
$$E=\sum_{\bfa\in \N^{r+1}}\,
\alpha_{\bfa }\bfy^\bfa \ \in \Z[y_1,\dots,y_{r+1}]\setminus \{0\}$$
satisfying $ E(q_1,\dots,q_{r+1})=0$ on $V$ and such that,  for
all $\bfa \in \Supp(E)$,
\begin{align*}
  \bullet & \ \langle \bfd,\bfa\rangle \le
  \bigg(\prod_{j=1}^{r+1}d_j\bigg)\deg(V),\\[-1mm]
  \bullet & \ \h(\alpha_{\bfa}) + \langle \bfh,\bfa \rangle \le
  \bigg(\prod_{j=1}^{r+1}d_j\bigg) \bigg(\wh(V) + \deg(V) \bigg(
  \sum_{\ell=1}^{r+1}\frac{h_\ell}{d_\ell}+
  (r+2)\log(2n+8)\bigg)\bigg).
\end{align*}
\end{corollary}

\begin{proof}
  We will follow the arguments in the proof of Theorem~\ref{perronpar}
  adapted to {this} situation.  Choose $r+1$ variables $x_{i_1},
  \dots, x_{i_{r+1}}$ in the group $\bfx$ so that the linear
  projection $V\to \A^{r+1}_{\Q}, \bfx\mapsto (x_{i_1}, \dots,
  x_{i_{r+1}})$ is generically finite onto its image.  Adding a
  further variable $z$, consider the map
$$
V_{{\Q(z)}}\longrightarrow \A^{r+1}_{\Q(z)} \quad, \quad
\bfx\longmapsto\left(q_1(\bfx)+z \, x_{i_1} ,\dots,q_{r+1}(\bfx)+z \,
  x_{i_{r+1}}\right),
$$
which is generically finite onto its image. Let $\wt E= \sum_{\bfa}\alpha_{\bfa}\bfy^{\bfa}\in
\Z[z][\bfy]$ be a primitive and squarefree polynomial defining the
closure of the image of this map. The
polynomials $q_j(\bfx)+z \, x_{i_j} $
have $\bfx$-degree $d_{j}$, $z$-degree 1 and height $h_{j}$.
Theorem~\ref{implicitmultiparZ} applied to this case gives
$ \langle \bfd,\bfa\rangle \le   \Big(\prod_{j=1}^{r+1}d_j\Big)\deg(V)$ and
 \begin{multline*}
\h(\alpha_{\bfa}) + \langle \bfh,\bfa \rangle\le
   \bigg(\prod_{j=1}^{r+1}d_j\bigg)
   \Bigg(\wh(V) + \deg(V) \Bigg( \log(r+2)\\
+   \sum_{\ell=1}^{r+1}\frac{1}{d_\ell}\bigg( h_\ell
   +\log\Big(\#\Supp(q_\ell)+3\Big)+ \log(2) \bigg) \Bigg)\Bigg).
 \end{multline*}
The polynomial $E:=\wt E(0,\bfy)\in \Z[\bfy]$ gives a {non-trivial} relation of algebraic dependence for the
$q_{j}$'s and it satisfies the same  degree and height bounds as
$\wt E$. The
statement then follows from the inequality $\# \Supp(q_\ell)+3\le (n+1)^{d_\ell}+3\le (n+4)^{d_\ell}$.
\end{proof}

{Next} result gives an upper bound for the Mahler
measure and {\it a fortiori}, for  the height of the implicit equation
for a rational map.

\begin{theorem}\label{racionalZ}
  Let $V \subset \A^n_{\Q}$ be a $\Q$-variety of
 pure dimension $r$ and
$$
\psi: V \dashrightarrow \A^{r+1}_{\Q} \quad, \quad \bfx\mapsto
\Bigg( \frac{q_1}{p_1}(\bfx) ,\dots,\frac{q_{r+1}}{p_{r+1}}(\bfx)
\Bigg)
$$
a rational map, generically finite onto its image, defined by
polynomials $q_j, p_j\in \Z[\bfx]$ such that $q_j/p_j \notin \Q$.  Let
$E \in \Z[\bfy]$ be a primitive and squarefree polynomial
defining~${\ov{\Im(\psi)}}$.  Set $d_j=\max\{\deg_\bfx (q_j),
\deg_\bfx(p_j)\}$ and $h_j=\log(||q_j||_1+||p_j||_1)$ for $1\le j\le
r+1$. Then
\begin{align*}
  \bullet & \ \deg_{y_i} (E) \le
  \bigg(\prod_{j\ne i}d_j\bigg)\deg(V) \quad \mbox{for } 1\le i\le r+1,\\[-1mm]
  \bullet & \ \m(E) \le \bigg(\prod_{j=1}^{r+1}d_j\bigg)
  \bigg(\wh(V) + \deg(V)
  \sum_{\ell=1}^{r+1}\frac{h_\ell}{d_\ell}\bigg).
\end{align*}
\end{theorem}

\begin{proof} The bounds for the partial degrees of $E$ follow from
  Theorem~\ref{racionalpar}. Thus, we only have to prove the upper bound for
  the Mahler measure. Let $V(E)\subset \A^{r+1}_{\Q}$ and $G\subset
  V\times \A^{r+1}_{\Q}$ denote the closure of the image of $\psi$ and
  of the graph of $\psi$, respectively. We consider mixed heights of
  $G$ and $V(E)$ with respect to the inclusions
$$V\times \A^{r+1}\hookrightarrow \P^n\times (\P^1)^{r+1} \quad ,\quad
\A^{r+1}\hookrightarrow (\P^1)^{r+1}.
$$
Propositions~\ref{mahler} and \ref{113alturasrat} imply that
\begin{equation}\label{eq:18}
  \m(E)= \wh _\bfone(V(E))
\le \wh _{0,\bfone}(G).
\end{equation}
Let
\begin{math}
[ G\cz\in
\R[\eta,\theta_0,\bftheta]/(\eta^2,\theta_0^{n+1},\theta_1^2,\dots,\theta_{r+1}^2)
\end{math} be the class of the closure of $G$ in the extended Chow
ring of $\P^n\times (\P^1)^{r+1}$. We will  bound it by applying
Corollary \ref{cor:213}  recursively to $V\times \A^{r+1}$ and
$g_j:=y_jp_j-q_j$ for $j=1,\ldots, r+1$. Propositions
\ref{altrodgenrat}\eqref{item:40} and~\ref{prop:10}\eqref{item:35}
imply that
\begin{math}
  [ V\times \A^{r+1}\cz=\wh(V)\eta\theta_0^{n-r-1} +\deg(V)\theta_0^{n-r}.
\end{math}
Besides,  $||p_jy_j-q_j||_{\sup}\le ||p_{j}||_{1}+||q_j||_{1}$ and so  $[p_jy_j-q_j]_\sup\le h_j\eta + d_j\theta_0+\theta_j$. Hence,
\begin{displaymath}
  [ G\cz\le [ V\times \A^{r+1}\cz
  \prod_{j=1}^{r+1}[p_jy_j-q_j]_\sup\leq
 \Big(\wh(V)\eta\theta_0^{n-r-1} +\deg(V)\theta_0^{n-r}\Big)  \prod_{j=1}^{r+1}\Big(h_j\eta + d_j\theta_0+\theta_j).
\end{displaymath}
The mixed height $ \wh _{0,\bfone}(G)$
is the coefficient of  the monomial $\eta\,
\theta_0^n$ in $[ G\cz$.
The above inequality then implies
$$
\wh _{0,\bfone}(G) \le\bigg(\prod_{j=1}^{r+1}d_j\bigg) \wh(V) +
\deg(V)\sum_{\ell=1}^{r+1}\bigg(\prod_{j\ne \ell}d_j\bigg)h_\ell.
$$
The statement follows from this inequality and \eqref{eq:18}.
\end{proof}

For $V=\A^n_{\Q}$, the above result gives the bounds
\begin{equation*} \label{eqn:102} \deg_{y_i} (E) \le \prod_{j\ne i}d_j
  \quad , \quad \m(E) \le \sum_{\ell=1}^{n+1} \bigg( \prod_{j\ne \ell}
  d_j\bigg) h_\ell
\end{equation*}
for the degree and the Mahler measure of ${\ov{\Im(\psi)}}$.
Using Lemma \ref{lemm:11}\eqref{krso01}, we can bound the height of
this polynomial by
\begin{displaymath}
 \h(E) \le
\sum_{\ell=1}^{n+1} \bigg( \prod_{j\ne \ell} d_j\bigg)
(h_\ell+\log(2)) .
\end{displaymath}

\begin{example}\label{ejemploracionalZ}
  Let $d_1,d_2, H_1,H_2\ge 1$ such that $d_1,d_2$ are coprime and
  set
$$
\frac{q_1}{p_1}=\frac{x^{d_1}}{H_1(x+1)^{d_1}} \quad, \quad
\frac{q_2}{p_2}= \frac{(x+1)^{d_2}}{H_2\, x^{d_2}}.
$$
The implicit equation of the {closure of the image of}
$$\A^1_{\Q}\dashrightarrow \A^2_{\Q}, \quad
x\longmapsto\Bigg(\frac{q_1}{p_1}(x), \frac{q_2}{p_2}(x)\Bigg)$$ is
defined by $ E={H_1}^{d_2}{H_2}^{d_1}y_1^{d_2}y_2^{d_1}-1\in
\Z[y_1,y_2] $. We have  $\max\{\deg(p_j), \deg(q_j)\}=d_j$ while
$\log(||q_1||_1+||p_1||_1)= \log(1+H_12^{d_1})$ and $\log(||q_2||_1+||p_2||_1)=
\log(2^{d_2}+H_2)$.  Hence, Theorem~\ref{racionalZ} gives
{$$\m(E)\le d_2 \log(1+H_12^{d_1})+d_1 \log(2^{d_2}+H_2).$$} Indeed,
$\m(E)= d_1\log(H_2)+d_2\log(H_1)$ and so the obtained bound is optimal
up a term of size $O(d_1d_2)$.
\end{example}

\section{Arithmetic Nullstellens\"atze}

\subsection{An effective approach to the Nullstellensatz}\label{Jelonek}

Recall the statement of the weak Nullstellensatz over a variety:

\medskip
\begin{quote}{\it Let $V\subset \A^n_{K}$ be a $K$-variety and $f_1,\dots,f_s\in K[x_1,\dots,
    x_n]$ without common zeros in $V({\overline K})$.  Then there
    exist $g_1,\dots, g_s\in K[x_1,\dots, x_n]$ such that $
    1=g_1f_1+\cdots+g_sf_s$ on $V$.  }
\end{quote}

\medskip

We recall Jelonek's approach in \cite{Jelonek05} to
produce bounds for the degree of the $g_i$'s and explain {how it can be
adapted} to obtain bounds for their height. Set $r=\dim(V)$.
We assume without loss of
generality that $s\le r+1$, otherwise we reduce to this case by
taking generic linear combinations of the input polynomials.
Consider the regular map
\begin{equation}\label{phi}
  \varphi: V\times\A^1_K{\longrightarrow}V\times\A^s_K \quad, \quad
  (\bfx,z)\longmapsto\left(\bfx,zf_1(\bfx),\dots,z f_s(\bfx)\right).
\end{equation}
Since the $f_j$'s have no common zeros in $V({\ov{K}})$,  the
subsets $\{(\bfx,z_1,\dots,z_s): f_j(\bfx)\ne 0\}$ form an open
covering of $ V\times\A^s_K$. This implies that
$$
\Im(\varphi)=\Big\{(\bfx,z_1,\dots,z_s)\in
V\times\A^s_K:\,f_i(\bfx)z_j=f_j(\bfx)z_i \mbox{ for all }
i,j\Big\}.
$$
In particular, $\Im(\varphi)$ is a subvariety of $V\times\A^s_K$.
Furthermore, let $g_1,\dots,g_s$ be polynomials such that
$1=g_1f_1+\cdots + g_sf_s$. The map $\Im(\varphi)\to V\times \A^1_K$
defined as $(\bfx,z_1,\dots,z_s)\mapsto(\bfx,g_1(\bfx)z_1+\cdots
+g_s(\bfx)z_s)$ is a left inverse of  $\varphi$. Hence, $\varphi$
induces an isomorphism $ V\times \A^1_K\to \Im(\varphi)$.

Assume there exists a finite linear projection  $\pi:
\Im(\varphi)\to \A^{r+1}_K$. Such a map exists if the field $K$ is
sufficiently big (for instance, if it is infinite). In this case, we
can assume without loss of generality that $\pi$ is given by a
$(r+1)\times (n+s)$-matrix in reduced triangular form. Hence, there
are linear forms $\ell_i=\gamma_{i,1}x_1+\cdots+\gamma_{i,n}x_n
\in K[\bfx]$, $1\le i\le r+1$, such that this projection writes down
as
$$
\pi: \Im(\varphi){\to}\A^{r+1}_K \ , \ (\bfx,z_1,\dots,z_s)\mapsto
\Big(z_1+\ell_1(\bfx),\dots,
z_s+\ell_s(\bfx),\ell_{s+1}(\bfx),\dots,\ell_{r+1}(\bfx)\Big).
$$
Thus, the composition $\pi\circ \varphi:V\times
\A^1_K{\rightarrow}\A^{r+1}_K$ is a finite map or equivalently,  the
inclusion of algebras $(\pi\circ \varphi)^*: K[y_1,\dots, y_{r+1}]
\hookrightarrow
  K[\bfx,z]/I(V\times \A^1_K)$ is an integral extension. This
map writes down, for  $(\bfx,z) \in V\times \A^1_K$, as
$$
\pi\circ \varphi(\bfx,z)=
(zf_1(\bfx)+\ell_1(\bfx),\dots,zf_s(\bfx)+\ell_s(\bfx),\ell_{s+1}(\bfx),\dots,\ell_{r+1}(\bfx)).
$$
Up to a non-zero scalar in $K$, the minimal polynomial of $z$ over
$K[\bfy]$ is of the form
$$E=z^\delta +\sum_{j=1}^\delta \sum_{\bfa \in \N^{r+1}}
\alpha_{\bfa,j}\bfy^\bfa z^{\delta-j} \in K[\bfy,z].
$$
Therefore,
\begin{equation}\label{jelonek}
  z^\delta +\sum_{j=1}^\delta \sum_{\bfa}
  \alpha_{\bfa,j}(zf_1+
  \ell_1)^{a_1}\cdots
  (zf_{s}+\ell_s)^{a_s}\ell_{s+1}^{a_{s+1}}\cdots
  \ell_{r+1}^
  {\alpha_{r+1}} z^{\delta-j}=0 \quad \mbox{on } V\times \A^1
\end{equation}
and so all  the coefficients in the expansion of~(\ref{jelonek})
with respect to the variable $z$
 vanish identically on $V$.  We
derive a B\'ezout identity from the coefficient of $z^\delta$ as
follows: for each $1\le j \le \delta$ and $\bfa\in \N^{r+1}$, the
coefficient of $z^j$ in the expansion of $ (zf_1+
\ell_1)^{a_1}\cdots (zf_{s}+\ell_s)^{a_s}\ell_{s+1}^{a_{s+1}}\cdots
\ell_{r+1}^ {\alpha_{r+1}} $ writes down as
$$
\sum_{\bfb} g_\bfb f_1^{b_1}\cdots f_s^{b_s}
$$
with $g_\bfb\in K[\bfx]$ and $\bfb\in \N^s$ such that $|\bfb| = j\ge
1$.  Hence each term in this sum is a multiple of some $f_j$ and,
regrouping terms, we can write the coefficient of $z^{\delta}$ in
the expansion of~(\ref{jelonek}) as $1- {g_1}f_1-\cdots- {g_s}f_s$
with $ {g_i}\in K[\bfx]$.

\begin{remark}\label{rem:6}
  The  argument above is not a proof
of the Nullstellensatz since it relies on the {\it a priori}
existence of a B\'ezout identity: this is used to prove that the map
$\varphi$ in~(\ref{phi}) is an isomorphism onto its image.
\end{remark}

{In our treatment of the arithmetic Nullstellensatz, we will use the
previous construction for the general linear forms}
  \begin{equation}\label{linear}
    \ell_i:=u_{i,1}x_1+\cdots+u_{i,n}x_n \in K[\bfu_{i}][\bfx] \quad \mbox{ for } 1\le i\le r+1,
  \end{equation}
  where $\bfu_{i}=\{u_{i,j}\}_{1\le j\le n}$ is a group of
  auxiliary variables.  This is a valid choice, since   the associated
linear projection $\pi:
  \Im(\varphi)_{{K(\bfu)}}\to \A^{r+1}_{K(\bfu)}$ is a finite map,
  see for instance~\cite[Ch.~7, Thm.~2.2]{Lang93}.

Assume that $K$ is the field of fractions of a factorial ring  $A${.
Let} $E\in A[\bfu][\bfy,z]$ be the minimal polynomial of $z$ with
respect to the map $\pi \circ \varphi$, primitive with respect to
the ring $A[\bfu]$. We expand it as
\begin{equation}
  \label{eq:35}
E=\alpha_{{\bf0},0} z^\delta
  +\sum_{j=1}^{\delta}\sum_{\bfa\in \N^{r+1} }
  \alpha_{\bfa,j}\bfy^\bfa z^{\delta-j}
\end{equation}
with  $ \alpha_{\bfa,j}\in  A[\bfu]$ such that $\alpha_{{\bf0},0}
\ne 0$. For $1\le i \le s$, set
\begin{equation}
  \label{eq:19}
\wt g_i= -\sum_{j=1}^{\delta} \sum_{\bfa} \sum_{\bfb} \alpha_{\bfa,
  j} \Bigg( \Bigg(\prod_{k=1}^{r+1} {a_k\choose b_k}
\ell_k^{a_k-b_k}\Bigg)f_1^{b_1}\cdots f_{i-1}^{b_{i-1}}
f_{i}^{b_i-1}\Bigg) \in A[\bfu][\bfx],
\end{equation}
{the sums being indexed by all $\bfa, \bfb\in \N^{r+1}$  such that}
$(\bfa,j)\in \Supp(E)$ and $|\bfb|=j$, $b_k\le a_k$ for $1\le k\le
i$, $b_k=0$ for $i+1\le k\le r+1$ and $b_i\ge 1$.
Using~(\ref{jelonek}), we can verify that $  \alpha_{{\bf0},0}-\wt
g_1f_1-\cdots - \wt g_sf_s$ is the coefficient of $z^\delta$ in the
expansion of
$$E(zf_1+\ell_1,\dots,zf_s+\ell_s,\ell_{s+1},\dots,\ell_{r+1},z).$$ Hence,
\begin{equation}\label{eq:51}
  \alpha_{{\bf0},0}=\wt g_1f_1+\cdots + \wt g_sf_s \quad \mbox{ on }  V_{{K(\bfu)}}.
\end{equation}
We then extract a B\'ezout identity on $V$ {by considering the
coefficient of any monomial in $\bfu$ in \eqref{eq:51} appearing in the
monomial expansion of $\alpha_{{\bf0},0}$.}

\medskip We need the following lemma relating {minimal polynomials with
the implicitization problem.}

\begin{lemma} \label{lemm:10}
Let $V\subset \A_{K}^{n}$ be a $K$-variety of pure dimension $r$ and
$q_{1},\dots, q_{r+1}\in K[\bfx,z]$ such that the map
\begin{displaymath}
V\times \A^{1}_{K}\longrightarrow \A^{r+1}_{K} \quad ,\quad
(\bfx,z)\longmapsto (q_{1}(\bfx,z),\dots,
 q_{r+1}(\bfx,z))
\end{displaymath}
 is finite. Let $E\in  K[\bfy][z]\setminus \{0\}$ be the
  minimal polynomial of $z$ with respect to this map.
Then the map
\begin{displaymath}
\psi: V_{K(z)}\longrightarrow \A_{K(z)}^{r+1} \quad, \quad
\bfx\longmapsto (q_{1}(\bfx),\dots,
 q_{r+1}(\bfx))
\end{displaymath}
is generically
 finite onto its image. Furthermore, $E\in K[z][\bfy]$ is a squarefree
 polynomial, primitive with respect to the ring $K[z]$,
defining ${\ov{\Im(\Psi)}}$.
\end{lemma}

\begin{proof}
  We reduce without loss of generality to the case {of an
    irreducible $K$-variety $V$}. Consider the map
\begin{displaymath}
\rho:   V\times \A_{K}^{1}\longrightarrow  \A_{K}^{r+1}\times
\A_{K}^{1} \quad, \quad
 (\bfx,z)\longmapsto (q_{1}(\bfx,z),\dots,
 q_{r+1}(\bfx,z),z).
\end{displaymath}
Its image is the irreducible $K$-hypersurface defined by $E$.  In
algebraic terms, the kernel of $\rho^*:K[\bfy,z]\to
K[\bfx,z]/I(V\times \A^{1}_{K})$ is the principal ideal generated by
$E$. Furthermore, $E\notin K[z]$ as, otherwise, this would imply
that the $q_{j}$'s are constant on $V\times\A^{1}_{K}$, which is not
possible because of the finiteness assumption.

We have $K[\bfx,z]/I(V\times \A^{1}_{K})= K[z] \otimes_{K}
K[\bfx]/I(V) $. Hence, $E$ is also a generator of the kernel of
$\psi^{*}:K(z)[\bfy]\to K(z)[\bfx]/I(V_{K(z)})$.
 The image of $\psi$ is a hypersurface of~$\A^{r+1}_{K(z)}$, since $E$ does not
lie in $ K[z]$. Hence, this map  is generically finite onto its
image. Furthermore,  $E$ is an irreducible polynomial defining this
hypersurface, primitive with respect to the ring $K[z]$.
\end{proof}

Consider the polynomials $q_j \in  A[\bfu][\bfx,z]$ defined by
\begin{equation}
  \label{eq:20}
  q_j=
  \begin{cases}
    zf_j + \ell_j, & \mbox{for } 1\le j\le
    s,\\[1mm]
\ell_j, & \mbox{for } s+1\le j\le r+1.
  \end{cases}
\end{equation}
In our setting, $ \pi\circ \varphi:V_{K(\bfu)}\times
\A^1_{K(\bfu)}{\longrightarrow}\A^{r+1}_{K(\bfu)}$ is finite. Hence,
Lemma~\ref{lemm:10} implies that the polynomial $E$ in \eqref{eq:35}
is a squarefree polynomial, primitive with respect to the ring
$A[\bfu,z]$,
 defining the closure of the image of the generically finite map
\begin{equation}\label{eq:46}
\psi:  V_{K(\bfu,z)}\longrightarrow \A^{r+1}_{K(\bfu,z)}\quad, \quad
\bfx\longmapsto (q_{1}(\bfx),\dots, q_{r+1}(\bfx)).
\end{equation}
Thus, we can produce bounds for  its size by using a suitable
version of Perron's theorem. In turn, this will allow us to bound the
size of the polynomials in the associated B\'ezout identity. We will
see the details in the next sections.

\subsection{Parametric Nullstellens\"atze}

We now apply the {previous} construction together with the
parametric Perron's theorem to produce different Nullstellens\"atze
for polynomials with coefficients depending on groups of parameters.
Let $k$ be a field. {Consider groups of variables
  $\bfx=\{x_{1},\dots, x_{n}\}$ and
  $\bft_l=\{t_{l,1},\dots,t_{l,p_l}\}$, $1\le l\le m$.} Set
$\bft=\{\bft_{1},\dots, \bft_{m}\}$ and, for each $1\le l\le m$, write
$k_l=k(\bft_1,\dots,\bft_{l-1},\bft_{l+1},\dots,\bft_m)$.  Recall that
$ k_{l}(\bft_{l})=k(\bft)$ and that for a projective $k(\bft)$-variety
$V$, we denote by $\h_{\bft_{l}}(V)$ its height with respect to the
base ring~$k_{l}[\bft_{l}]$.

\begin{theorem} \label{multiparametric} Let $V\subset
  \A^n_{_{k(\bft)}}$ be a $k(\bft)$-variety of pure dimension~$r$ and
  $f_1,\dots, f_s\in k[\bft,\bfx] \setminus k[\bft]$ a family of $s\le
  r+1$ polynomials without common zeros in $V(\ov{k(\bft)})$. Set
  $d_j=\deg_\bfx (f_j)$ and $h_{l,j}=\deg_{\bft_l} (f_j)$ for $1\le
  j\le s$ and $1\le l\le m$.  Then there exist $\alpha\in
  k[\bft]\setminus \{ 0\} $ and $g_{1},\dots,g_s\in k[\bft,\bfx] $
  such that
  \begin{equation}
    \label{eq:50}
\alpha=g_1f_1+\cdots+g_sf_s\quad \mbox{on } V
  \end{equation}
with
\begin{align*}
  \bullet &\  \deg_\bfx(g_if_i) \leq \bigg(\prod_{j=1}^{s}d_j\bigg)\deg(V),\\
  \bullet & \ \deg_{\bft_l}(\alpha), \deg_{\bft_l}(g_if_i) \leq
  \bigg(\prod_{j=1}^{s}d_j \bigg)\bigg( \h_{\bft_{l}}(V)+\deg(V)
  \sum_{\ell=1}^{s} \frac{h_{l,\ell} }{d_\ell}\bigg) \quad \mbox{for }
  1\le l\le m.
\end{align*}
\end{theorem}

\begin{proof}
  We apply the construction explained in \S\ref{Jelonek} to the ring
  $A:=k[\bft]$, the variety~$V$ and the polynomials $f_{j}$. We will
  freely use the notations introduced in that section. As a result of
  that construction, we obtain $\alpha \in A\setminus \{0\}$ and
  $g_{i}\in A[\bfx]$ satisfying the B\'ezout identity \eqref{eq:50}.
  These scalar and polynomials are obtained as the coefficients of a
  monomial in the variables $\bfu$ in the B\'ezout identity on
  $V_{K(\bfu)}$ in~\eqref{eq:51}. Hence,
\begin{equation}
  \label{eq:52}
  \deg_\bfx\left(g_if_i\right) \leq  \deg_\bfx\left(\wt g_if_i\right)
  \ , \
  \deg_{\bft_l}(\alpha)\leq
  \deg_{\bft_l}(\alpha_{\bfzero,0}) \ , \ \deg_{\bft_l}(g_if_i) \leq\deg_{\bft_l}(\wt g_if_i).
\end{equation}
Let $E\in A[\bfu][\bfy,z]$ be the polynomial in \eqref{eq:35} and set
$\bfd=(d_1,\dots,d_s,1,\dots,1)\in \N^{r+1}$. From the definition of
$\wt g_{i}$ in \eqref{eq:19} and using that, by hypothesis, $d_i\ge1$
for all $i$,
 \begin{multline}\label{eq:53}
   \deg_\bfx(\wt g_i f_i) \le \max_{j,\bfa, \bfb} \Bigg\{
   \deg_\bfx\Bigg(\alpha_{\bfa, j} \Bigg( \prod_{e=1}^{r+1}
   {a_e\choose b_e}
   \ell_e^{a_e-b_e}\Bigg)f_1^{b_1}\cdots f_{i-1}^{b_{i-1}}f_{i}^{b_i}\Bigg)\Bigg\}\\[-1mm]
   \le \max_{\bfa, \bfb}\bigg\{\sum_{e=1}^{r+1} (a_e -b_e) +
   \sum_{\ell=1}^i d_\ell b_\ell \bigg\} \le \max_{\bfa} \langle
   \bfd,\bfa\rangle,
\end{multline}
with $1\le j\le \deg_{z}(E)$, $\bfa\in \N^{r+1}$ in the support of $E$
with respect to the variables~$\bfy$ and $\bfb\in \N^{r+1}$ satisfying
$|\bfb|=j$, $b_e\le a_e$ for $1\le e\le i$, $b_e=0$ for $i+1\le e\le
r+1$ and $b_i\ge 1$. For $1\le l\le m$ we set
$\bfh_{l}=(h_{l,1},\dots,h_{l,s},0,\dots,0)\in \N^{r+1}$. Using Lemma
\ref{lemma:2-1}, we obtain similarly
\begin{align}\label{eq:54}
  \deg_{\bft_{l}}( \wt g_i f_i) &\le \max_{j,\bfa, \bfb} \Bigg\{
  \deg_{\bft_{l}}\Bigg(\alpha_{\bfa, j} \Bigg(\prod_{e=1}^{r+1}
  {a_e\choose b_e}
  \ell_e^{a_e-b_e}\Bigg)f_1^{b_1}\cdots f_{i-1}^{b_{i-1}}f_{i}^{b_i}\Bigg)\Bigg\}\\[-1mm]
  & \nonumber \le \max_{j,\bfa,\bfb}\bigg\{
  \deg_{\bft_l}(\alpha_{\bfa,j
  })+ \sum_{\ell=1}^i h_{l,\ell} b_\ell \bigg\} \\
  &\nonumber \le \max_{\bfa} \Big\{ \deg_{\bft_{l}}(\alpha_{\bfa})+
  \langle \bfh_l,\bfa\rangle \Big\}.
\end{align}
By Lemma \ref{lemm:10}, $E$ is a primitive and squarefree polynomial
in $k[\bft,\bfu,z][\bfy]$ defining the closure of the image of the map
$\psi$ in \eqref{eq:46}. Hence, we can apply
Corollary~\ref{implicitgroups} to bound the partial degrees of this
polynomial.  We have $\deg_\bfx(q_j)= d_j, \, \deg_{\bft_l}(q_j)=
h_{l,j}$ for $1\le j\le s$ while $\deg_\bfx(q_j)= 1,\,
\deg_{\bft_l}(q_j)= 0$ for $s+1\le j\le r+1$. By
Proposition~\ref{prop:9}\eqref{item:9}, $\deg(V _{{k(\bft, \bfu,z)}})=
\deg(V)$ and $\h_{\bft_{l}}(V_{{k(\bft,\bfu,z)}} )= \h_{\bft_{l}}(V)$.
Therefore, for all $\bfa \in \Supp(E)$,
$$
\langle \bfd,\bfa \rangle\le \Bigg(\prod_{j=1}^{s}d_j\Bigg)\deg(V) \ ,
\ \deg_{\bft_{l}}(\alpha_{\bfa})+\langle \bfh_l,\bfa\rangle \le
\Bigg(\prod_{j=1}^{s}d_j\Bigg)\Bigg(\h_{\bft_{l}}(V)+ \deg(V)
\sum_{\ell=1}^{s}\frac{h_{l,\ell}}{d_\ell}\Bigg).
$$
The statement follows from this inequality together with
\eqref{eq:52}, \eqref{eq:53} and \eqref{eq:54}.
\end{proof}

Theorem \ref{mtp} in the introduction corresponds to the case of
polynomials depending on \emph{one} group of parameters. It follows
readily from the above result and the fact that a variety defined over
$k(\bft)$ can be identified with a $k(\bft)$-variety, see
Remark~\ref{rem:4}.

\begin{example}\label{masserfuncional}
Consider the following variant of a classical
  example due to Masser and Philippon: let $t$ be a variable,
$d_1,\dots, d_{n}\ge 1$, $h\ge 0$ and set
  \begin{multline*}
    f_1= x_1^{d_1}, f_2= x_1
    \, x_n^{d_2-1} - x_2^{d_2},
    \dots, f_{n-1}= x_{n-2} \, x_n^{d_{n-1}-1} - x_{n-1}^{d_{n-1}},\\[1mm]
    f_n= x_{n-1} \, x_n^{d_n-1} - t^{h} \in k[t,x_{1},\dots, x_{n}].
  \end{multline*}
  It is a system of $n$
  polynomials without common zeros in~$\ov{k(t)}^{n}$.  We have
  $\deg_{\bfx}(f_j)=d_j$ for all $j$ while $\deg_t(f_j)=0$ for $1\le j\le
  n-1$ and $\deg_t(f_n)=h$. Theorem \ref{multiparametric} implies that
  there exists a B\'ezout identity $ \alpha= g_1 \,f_1 + \cdots +
  g_n\,f_n$ with
  \begin{equation}
    \label{eq:7}
    \deg_\bfx\left(g_if_i\right) \leq d_1\cdots d_n \quad, \quad
    \deg(\alpha), \deg_t(g_if_i) \leq
    d_1\cdots d_{n-1}h .
  \end{equation}
To obtain a lower bound, consider a  further variable $u$  and
specialize any such B\'ezout identity {at $x_{i}=\gamma_{i}$ with
$$
\gamma_1= t^{d_2\cdots d_{n-1}h} \, u^{d_2\cdots d_n-1}, \dots,
\gamma_{n-2}=t^{d_{n-1}h}\, u^{d_{n-1}d_n-1}, \gamma_{n-1}= t^{h} \, u^{d_n-1},
\gamma_n = 1/u.
$$
We obtain that $\alpha = g_1(\gamma_1,\dots, \gamma_{n-1}, 1/u )
\, t^{d_1\cdots d_{n-1} h} u^{d_1\cdots d_n-d_1}$.} From this, we deduce
the lower bounds
$$  \deg_{\bfx}( g_1f_1) \ge d_1\cdots d_n
\quad, \quad \deg(\alpha) \ge d_1\cdots d_{n-1}h .
$$
Hence,  both bounds in~(\ref{eq:7}) are optimal in this case.
\end{example}

\begin{example} \label{ejemploresultante}  Let $V\subset   \A^n_{k}$
  be a $k$-variety of pure dimension~$r$.
  For $1\le j\le r+1$, let $d_{j}\ge 1$ and consider the general $n$-variate polynomial of degree $d_j$
$$
F_j=\sum_{|\bfa|\le d_j}u_{j,\bfa}\bfx^\bfa \in
k[\bfU_j][x_{1},\dots, x_{n}],
$$
where $\bfU_j=\{u_{j,\bfa}\}_{ |\bfa|\le d_j}$. Set
$\bfU=\{\bfU_1,\dots, \bfU_{r+1}\}$. It is not difficult to verify
that the $F_j$'s have no common zeros in
$V_{_{k(\bfU)}}(\ov{k(\bfU)})$. Hence, we can apply
Theorem~\ref{multiparametric} to the $F_{j}$'s as a system of
polynomials with coefficients depending on the groups of parameters
$\bfU_{1},\dots, \bfU_{r+1}$. We have that
$\deg(V_{_{k(\bfU)}})=\deg(V)$ and, by Proposition~\ref{prop:9}\eqref{item:18}, $\h_{\bfU_{l}}(V_{_{k(\bfU)}})=0$.
Besides, $\deg_\bfx(F_j)=d_j$ and $\deg_{\bfU_l}(F_j)$ is equal to  1 if
$l=j$ and to~$0$ otherwise. We deduce that there exist $\alpha\in
k[\bfU]\setminus \{0\}$ and $g_i\in k[\bfU][x_{1},\dots, x_{n}]$  such that $\alpha= g_1F_1+\cdots+g_{r+1}F_{r+1}$ on
$V_{_{k(\bfU)}}$ with
\begin{equation}\label{resultanteX}
  \deg_\bfx(g_iF_i) \le \bigg(\prod_{j=1}^{r+1} d_j \bigg) \deg(V)
\quad, \quad
  \deg_{\bfU_l}(\alpha), \deg_{\bfU_l} (g_iF_i) \le
  \bigg(\prod_{j\ne l} d_j\bigg) \deg(V).
\end{equation}

Using Lemma \ref{prop:8} and Corollary \ref{cor:3}, we can verify
that the elimination ideal
$$
\Big(I(V)\,k[\bfU][x_{1},\dots, x_{n}] +(F_1,\dots, F_{r+1}) \Big)
\cap k[\bfU] $$ is generated by the resultant $\Res_{d_1,\dots,
d_{r+1} }(\ov V)$ of the closure of $V$ in $\P^n_{k}$. This is a
multihomogeneous polynomial of partial degrees
 $$
 \deg_{\bfU_l}(\Res_{d_1,\dots, d_{r+1} }(\ov V))
 =\bigg(\prod_{j\ne l} d_j\bigg) \deg(V).
$$
Hence, $\alpha$ is  a multiple of this resultant and, comparing
degrees, we see that  $\alpha$ and $\Res_{d_1,\dots, d_{r+1} }(\ov
V)$ coincide up to a factor in $k^\times$. This implies that the
bound for the $\bfU_l$-degrees in~(\ref{resultanteX})  is optimal.
Observe that the bound for the $\bfx$-degree is not
 optimal, at least when $V=\A^n_{k}$. In this case, it can be shown
 that there exist
$g_i$'s satisfying the  same bound for the $\bfU_{l}$-degree and
such that $ \deg_\bfx(g_if_i) \le \Big(\sum_{j=1}^{n+1} d_j \Big)-
n$.
\end{example}

{The following result is a partial extension of Theorem
\ref{multiparametric} to an arbitrary number of polynomials. For
simplicity, we only state it for polynomials depending on one group of
parameters $\bft=\{t_{1},\dots,t_{p}\}$.  In this setting, we loose
track of the contribution of the $\bft$-degrees of most of the
individual input polynomials.  However, it is possible to
differentiate the contribution of {\it one} of them, which will be
important in the proof of the strong parametric Nullstellensatz.}

\begin{corollary} \label{parametricS}
Let $V\subset   \A^n_{_{k(\bft)}}$  be a $k(\bft)$-variety of pure
  dimension~$r$ and $f_1,\dots, f_s\in k[\bft,\bfx] \setminus
  k[\bft]$  without common
  zeros in $V(\ov{k(\bft)})$ such that $V(f_s)$ intersects~$V$
  properly.
Set $d_j=\deg_\bfx(f_j)$ for $1\le j\le s$.
  Assume that $d_1\ge \cdots \ge d_{s-1}$ and that $d_s$ is arbitrary. Set also $h=\max_{1\le
    j\le s-1}\deg_\bft (f_j)$ and $h_s=\deg_\bft(f_s)$.
  Then there exist $\alpha\in k[\bft]\setminus \{ 0\} $ and $g_1,\dots,g_s\in
  k[\bft,\bfx] $ such that $\alpha=g_1f_1+\cdots+g_sf_s$ on $V$ with
  \begin{align*}
    \bullet &\  \deg_\bfx\left(g_if_i\right) \leq \bigg(d_s\prod_{j=1}^{\min\{s-1,r\}}d_j\bigg)\deg(V),\\
    \bullet & \ \deg(\alpha), \deg_\bft(g_if_i) \leq
    \bigg(d_s\prod_{j=1}^{\min\{s-1,r\}}d_j \bigg)\bigg( \h(V
    )+\deg(V) \bigg( \frac{h_s}{d_s} + \sum_{\ell=1}^{\min\{s-1,r\}}
    \frac{h}{d_\ell}\bigg)\bigg).
  \end{align*}
\end{corollary}

\begin{proof}
  If $s\leq r+1$, the result follows from the case $m=1$ in
  Theorem~\ref{multiparametric}. In the case $s\ge r+2$, we can reduce
  to $r+1$ polynomials by taking linear combinations as follows: let
  $\bfv=\{v_{j,i}\}_{ 1\le j\le r, 1\le i\le s-r-1}$ be a group of
  variables and set $K=k(\bfv)$. Write
$$
\overline{f}_j= f_j+v_{j,1}f_{r+1}+\cdots + v_{j,s-r-1}f_{s-1} \quad
\mbox{for } 1\le j\le r,  \quad \overline{f}_{r+1}=f_s.
$$
The hypothesis that $V(f_s)$ and $V$ intersect properly implies that
the polynomials above do not have common zeros in
$V_{{K(\bft)}}(\overline{K(\bft)})$.  We have that
$\deg_\bfx(\overline{f}_j)=d_{j}$ and $\deg_\bft(\overline{f}_j)= h$
for $1\le j\le r$ while $\deg_\bfx(\overline{f}_{r+1})=d_s$ and
$\deg_\bft(\overline{f}_{r+1})=h_s$. By
Theorem~\ref{multiparametric}, there exist $\overline{\alpha}\in
k[\bfv,\bft]\setminus \{0\}$ and $\overline{g}_i\in
k[\bfv,\bft][\bfx]$ {such that  $
\overline{\alpha}=\overline{g}_1\overline{f}_1+\cdots
+\overline{g}_{r+1}\overline{f}_{r+1}$ on $V_{{K(\bft)}}$ with}
\begin{align*}& \deg_\bfx(\overline{g}_i\overline{f}_i) \le
  \bigg(d_s\prod_{j=1}^{r} d_j\bigg)\deg(V),\\
  & \deg_\bft(\overline{\alpha}),\deg_\bft(\overline{g}_i\overline{f}_i) \le \bigg(d_s\prod_{j=1}^{r}
  d_j\bigg) \bigg(\h(V)+\deg(V)
  \bigg(\frac{h_s}{d_s}+\sum_{\ell=1}^r\frac{h}{
    d_\ell}\bigg)\bigg).\end{align*} Unfolding the linear
combinations in the above identity and taking a non-zero coefficient
with respect to the variables $\bfv$, we extract a B\'ezout identity
$\alpha=g_1f_1+\dots+g_sf_s$ on $V$ satisfying the same degree
bounds.
\end{proof}

\begin{remark} \label{rem:5} The previous results by Smietanski are
  {for polynomials depending on at most two parameters without common
  zeros in the affine space. For instance, let $f_1,\dots, f_s$ be
  polynomials in $ k[t_1,t_2][x_1,\dots, x_n]\setminus k[t_{1},t_{2}]$
  without common zeros in~$\A^n(\ov{k(t_{1},t_{2})})$. Set
  $d_j=\deg(f_j)$, $h=\max_{j}\deg_{t_1,t_2} (f_j)$, and suppose that
  $d_2\ge \cdots \ge d_{s}\ge d_1$. Set also
  $\nu=\min\{n+1,s\}$. In~\cite{Smietanski93}, it is shown that} there
  exist $\alpha\in k[t_1,t_2]\setminus\{0\}$ and $g_i\in
  k[t_1,t_2][x_1,\dots, x_n]$ such that $\alpha=g_1f_1+\cdots+g_sf_s$
  with
\begin{align*}
\bullet \ &
\deg_\bfx\big(g_if_i\big) \le 3  \prod_{j=1}^{\nu}d_j,\\[-2mm]
\bullet \ & \deg_{t_1,t_2}(\alpha),\,\deg_{t_1,t_2}(g_if_i)\le \prod_{j=1}^\nu
(d_j+h)+3^{\nu-1}\bigg(\prod_{j=1}^{\nu}d_j\bigg)\bigg(\sum_{\ell=1}^{\nu}\frac{1}{d_\ell}\bigg)\,h.
 \end{align*}
\end{remark}

We deduce from Corollary \ref{parametricS} the following parametric version of the
strong effective Nullstellensatz.

\begin{theorem}\label{strongNSSpar} Let $V\subset
  \A^n_{_{k(\bft)}}$  be a $k(\bft)$-variety of pure
  dimension~$r$ and $g,f_1,\dots, f_s\in
  k[\bft,\bfx]$  such that $g$ vanishes on the set of
  common zeros of $f_{1},\dots, f_{s}$ in~$V(\ov{k(\bft)})$ and
  $\deg_\bfx(f_1)\ge \cdots \ge \deg_\bfx(f_s)\ge 1$.
  Set $d_j=\deg_\bfx(f_j)$ for $1\le
  j\le s$, $h=\max_{1\le j\le s}\deg_\bft (f_j)$,   $d_0=\max\{ 1,
  \deg_\bfx(g)\}$ and $h_0=\deg_\bft(g)$.  Then there exist $\mu\in
  \N$, $\alpha\in k[\bft]\setminus \{0\}$ and $g_1,\dots,g_s\in
  k[\bft,\bfx]$ such that
$$\alpha \,g^\mu=g_1f_1+\cdots+g_sf_s \quad \mbox{on } V$$  with
\vspace{-3mm}$$\aligned \bullet \ & \mu \le
2\bigg(\prod_{j=1}^{\min\{s,r+1\}}
d_j\bigg)\deg(V),\\[-2mm]
\bullet \ & \deg_\bfx(g_if_i)\le
4\bigg(\prod_{j=0}^{\min\{s,r+1\}}d_j\bigg)\deg(V),\\[-2mm]
\bullet \ & \deg(\alpha),\deg_\bft(g_if_i) \le
2\bigg(\prod_{j=0}^{\min\{s,r+1\}}d_j\bigg)\bigg(\h(V) +\deg(V)
\bigg(\frac{3h_0}{2d_0}+\sum_{\ell=1}^{\min\{s,r+1\}} \frac{h}{d_\ell}
\bigg)\bigg).
\endaligned$$
\end{theorem}

\begin{proof}
  Set $W=V\times \A^1_{_{k(\bft)}}$, let $y$ be an additional variable
  and consider
  \begin{equation}\label{eq:5}
    1-y^{d_0}g,f_1,\dots,f_s\ \in  k[\bft,\bfx,y].
  \end{equation}
  These polynomials have no common zeros on $W(\ov{k(\bft)})$ and
  $V(1-y^{d_0}g)$ intersects $W$ properly. We have that $\dim(W)=r+1$, $\deg(W)=\deg(V)$ and
  $\h(W)=\h(V)$,  thanks to  Lemma \ref{relalturas}. Besides, $\deg_{\bfx,y}(f_j)=d_j$ and
  $\deg_\bft(f_j)\le h$ for $1\le j\le s$, while
  $\deg_{\bfx,y}(1-y^{d_0}g)=2d_0$ and $\deg_\bft(1-y^{d_0}g)=h_0$.
  By Corollary~\ref{parametricS}, there exists $\alpha\in
  k[\bft]\setminus \{0\}$ and $\ov g_i\in k[\bft][\bfx,y]$ such that
  \begin{equation}
    \label{eq:45}
\alpha=\overline g_0 (1-y^{d_0}g)+ \overline g_1f_1+\cdots
+\overline g_sf_s \quad \mbox{on } W
  \end{equation}
with
\begin{align*} & \label{eq:62}  \deg_{\bfx,y}(\overline g_if_i)\le 2
  \bigg(\prod_{j=0}^{\min\{s,r+1\}} d_j\bigg)\deg(V),\\
&\deg(\alpha),
  \deg_\bft(\overline g_if_i)\le 2\bigg(\prod_{j=0}^{\min\{s,r+1\}}
  d_j\bigg)
  \bigg(\h(V)+\deg(V)\bigg(\frac{h_0}{2d_0}+\sum_{\ell=1}^{\min\{s,r+1\}}
  \frac{h}{d_\ell}\bigg)\bigg).
\end{align*}

The input system~(\ref{eq:5}) lies in the subring
$k[\bft,\bfx,y^{d_0}]$ of $k[\bft, \bfx,y]$ and we can suppose
without loss of generality that the B\'ezout identity \eqref{eq:45}
lies in  this subring. Let $\widehat g_i\in k[\bft,\bfx,y]$ such
that $\ov g_i(\bft,\bfx,y)=\widehat g_i(\bft,\bfx,y^{d_0})$. Then
\begin{equation}\label{gsombrero} \alpha=   \widehat
  g_0\Big(1-y\,g \Big)+\widehat
  g_1f_1+\cdots +\widehat
  g_sf_s.
\end{equation}

Specializing $y$ at $ 1/g(\bft,\bfx)$ in the above identity and
multiplying by a suitable denominator, we obtain an identity of the
form
 $\alpha\,g^\mu=  g_1 f_1+\cdots +   g_sf_s$
with $\mu = \max_{l}\deg_y(\widehat g_l)$ and $g_i= g^{\max_l\deg_y(\widehat
  g_l)}\, \widehat g_i(\bfx,1/g)$.
Therefore,
\begin{equation}\label{gradogsombrero}
\mu = \max_{l}\deg_y(\widehat g_l) \le
\max_{l}\bigg\{\frac{\deg_{\bfx,y}(\ov g_l)}{d_0} \bigg\}\le
2\bigg(\prod_{j=1}^{\min\{s,r+1\}}
  d_j\bigg)\deg(V).\end{equation}
Besides, $\deg_{\bfx,y}(\widehat g_if_i)\le \deg_{\bfx,y}(\overline
g_if_i)$ and $\deg_{\bft}(\widehat g_if_i)\le \deg_{\bft}(\overline
g_if_i)$. Therefore
\begin{align*}
&\deg_\bfx(g_if_i)\le \deg_\bfx(\widehat g_if_i) +
  \deg_\bfx(g)\max_l\deg_y(\widehat g_l) \le 4
  \bigg(\prod_{j=0}^{\min\{s,r+1\}}d_j\bigg)\deg(V),
  \\
  & \deg_\bft(g_if_i) \le \deg_\bft(\widehat g_if_i) +
  \deg_\bft(g)\max_l\deg_y(\widehat g_l) \\
  &\hspace{35mm}\le 2\bigg(\prod_{j=0}^{\min\{s,r+1\}} d_j\bigg)
  \bigg(\h(V)+\deg(V)\bigg(\frac{3h_0}{2d_0}+\sum_{\ell=1}^{\min\{s,r+1\}}
  \frac{h}{d_\ell}\bigg)\bigg),
\end{align*}
as stated.\end{proof}

\begin{remark} \label{rabinowicz} In the previous argument, the use of
  a differentiated version of the effective weak Nullstellensatz was
  crucial. Otherwise, the obtained bounds for the Noether exponent
  $\mu$ would have depended on the degree of $g$, as for instance in
  \cite{Brownawell87}, and the bounds for the height of the $g_{i}$'s
  would have been considerably worse.
\end{remark}

\subsection{Nullstellens\"atze over $\bf\Z$}\label{NssArithZ}
In this section we {present} different arithmetic Nullstellens\"atze
over $\Z$ for polynomials depending on groups of parameters.  As
before, let $\bft_l=\{t_{l,1},\dots,t_{l,p_l}\}$ be a group of
{variables} for $1\le l\le m$ and set $\bft=\{\bft_{1},\dots,
\bft_{m}\}$.  Write also $\bfx=\{x_{1},\dots, x_{n}\}$.

\begin{theorem}\label{arithparam}
   Let $V\subset \A^n_{\Q}$ be a $\Q$-variety
  of pure dimension $r$ and  $f_1,\dots, f_s\in
  \Z[\bft,\bfx]\setminus  \Z[\bft]$ a family of $s\le r+1$ polynomials without
  common zeros in~$V_{\Q(\bft)}(\ov{\Q(\bft)})$. Set
  $d_j=\deg_\bfx(f_j)$, $\delta_{l,j}=\deg_{\bft_l}(f_j)$ and $h_j=\h(f_j)$
  for $1\le j\le s$ and $1\le l\le m$.
  Then there exist $\alpha\in \Z[\bft]\setminus \{0\}$ and
  $g_1,\dots, g_s\in \Z[\bft,\bfx] $ such that
  \begin{equation}
    \label{eq:55}
  \alpha=g_1f_1+\cdots+g_sf_s \quad \mbox{on } V_{\Q(\bft)}
  \end{equation}
with
\vspace{-3mm} $$\aligned
\bullet& \ \deg_\bfx\left(g_if_i\right) \leq \bigg(\prod_{j=1}^{s}d_j\bigg)\deg(V) ,\\[-2mm]
\bullet & \ \deg_{\bft_l}(\alpha), \deg_{\bft_l}(g_if_i) \leq
\bigg(\prod_{j=1}^{s}d_j \bigg) \deg(V)
\sum_{\ell=1}^{s} \frac{\delta_{l,\ell} }{d_\ell} \quad \mbox{for } 1\le l\le m,\\[-1mm]
\bullet& \ \h(\alpha), \h(g_i)+\h(f_i) \leq
\bigg(\prod_{j=1}^{s}d_j\bigg) \Bigg(\wh(V)+ \deg(V) \Bigg((3r+7)\log(n+3) \\
& \hspace{3.4cm} +\sum_{\ell=1}^{s}\frac{1}{d_\ell}\bigg( h_\ell
+\log\big(\#\Supp(f_\ell)\big)+2\sum_{l=1}^m\delta_{l,\ell}\log(p_l+1)\bigg)
 \Bigg)\Bigg).
\endaligned
$$
\end{theorem}

\begin{proof}
We apply again the construction explained in \S\ref{Jelonek} to the
ring $A:=\Z[\bft_1,\ldots,\bft_m]$, the variety $V$ and the
polynomials $f_{j}$. As a result, we obtain $\alpha\in A^\times$ and
polynomials $g_{i}$ satisfying the B\'ezout identity \eqref{eq:55}.
The bounds for their $\bfx$-degree and $\bft_{l}$-degree follow from
Theorem~\ref{multiparametric}, since $\deg(V_{\Q(\bft)})=\deg(V)$
and $\h_{\bft_l}(V_{\Q(\bft)})=0$ by
Proposition~\ref{prop:9}\eqref{item:18}. So, we only have to bound
the height of these polynomials.

\medskip
We will  use the  notations introduced in  \S\ref{Jelonek}.
Let $\bfu_{i}$ denote the group of coefficients of
the general linear forms $
\ell_i=u_{i,1}x_1+\cdots+u_{i,n}x_n $ in \eqref{linear} and set
$\bfu=\{\bfu_{1},\dots, \bfu_{r+1}\}$.
Consider the minimal
polynomial $E\in \Z[\bft,\bfu][\bfy,z]$ of $z$ with respect to
the map $\pi\circ \varphi$, primitive with respect to $
\Z[\bft,\bfu]$. We expand it in two different ways as follows:
\begin{equation*}
  E=\alpha_{{\bf0},0} z^\delta
  +\sum_{j=1}^{\delta}\sum_{\bfa\in \N^{r+1} }
  \alpha_{\bfa,j}\bfy^\bfa z^{\delta-j}  =\sum_{\bfa\in \N^{r+1} } \alpha_{\bfa} \bfy^\bfa
\end{equation*}
with  $ \alpha_{\bfa,j}\in  \Z[\bft,\bfu]$ such that
$\alpha_{{\bf0},0} \ne 0$ and   $\alpha_\bfa=\sum_{k=0}^\delta
\alpha_{\bfa,k} z^{\delta-k}\in  \Z[\bft,\bfu,z]$. Set
$\Supp(E)=\{\bfa \,:\, \alpha_{\bfa}\ne 0\}$
and recall that, for $1\le i \le s$,
\begin{equation*}
\wt g_i= -\sum_{j=1}^{\delta} \sum_{\bfa\in \Supp(E)} \sum_{\bfb} \alpha_{\bfa,
  j} \Bigg( \Bigg(\prod_{k=1}^{r+1} {a_k\choose b_k}
\ell_k^{a_k-b_k}\Bigg)f_1^{b_1}\cdots f_{i-1}^{b_{i-1}}
f_{i}^{b_i-1}\Bigg) \in  \Z[\bft,\bfu][\bfx],
\end{equation*}
{the last sum being indexed by all $\bfb\in \N^{r+1}$  such that}
 $|\bfb|=j$, $b_k\le a_k$ for $1\le k\le i$, $b_k=0$ for $i+1\le
k\le r+1$ and $b_i\ge 1$. The constant $\alpha$ and polynomials
$g_i$ are obtained as the coefficient of some monomial in  $\bfu$ in
$\alpha_{\bfzero,0}$ and $\wt g_{i}$, respectively. Hence,
\begin{equation}\label{eq:55b}
 \h(\alpha)\leq
 \h(\alpha_{\bfzero,0}) \quad , \quad \h(g_i) +\h(f_{i}) \leq\h(\wt g_i) +\h(f_{i}).
\end{equation}
We have
$$
\sum_{j,\bfa,\bfb}\prod_{k=1}^{r+1}{a_k\choose
  b_k}\le \sum_{\bfa\in \Supp(E)} \sum_{\bfb\le \bfa }\prod_{k=1}^{r+1}{a_k\choose
  b_k} = \sum_{\bfa\in \Supp(E)} 2^{|\bfa|} \le
\#\Supp(E) \, 2^{\deg_{\bfy}(E)}.
$$
Hence, by Lemma~\ref{lemma:2-17}\eqref{item:45},
\begin{multline}
  \label{eq:56}
\h(\wt g_{i}) +\h(f_{i}) \le \max_{\bfa,j}
\Bigg\{\h\Bigg(\alpha_{\bfa,j}\Bigg( \prod_{k=1}^{r+1}
  \ell_k^{a_k-b_k}\Bigg)f_1^{b_1}\cdots
  f_{i-1}^{b_{i-1}}f_{i}^{b_i-1}\Bigg) \Bigg\} \\ +\h(f_{i}) + \log
  \Big(\#\Supp(E) \, 2^{\deg_{\bfy}(E)} \Big).
\end{multline}
Set $\bfd=(d_1,\dots,d_s,1,\dots,1)$, $\bfdelta_l=
(\delta_{l,1},\dots,\delta_{l,s},0,\dots,0)$ for $1\le l\le m$ and
$\bfh=(h_1,\dots,h_s,0,\dots,0)$.
Using Lemma \ref{lemma:2-17}\eqref{item:9},
\begin{multline}
  \label{eq:57}
  \h\Bigg(\alpha_{\bfa,j}\Bigg( \prod_{k=1}^{r+1}
  \ell_k^{a_k-b_k}\Bigg)f_1^{b_1}\cdots
  f_{i-1}^{b_{i-1}}f_{i}^{b_i-1}\Bigg)
\le \h(\alpha_{\bfa,j})+ |\bfa| \log (n) \\ + \langle \bfh-\bfe_{i},
\bfa\rangle + \langle \bfd,\bfa\rangle \log (n+1) +\sum_{l=1}^m
\langle \bfdelta_l,\bfa\rangle \log(p_l+1),
\end{multline}
since $\log||\ell_{k}||_{1} =\log(n)$, $\log||f_{j}||_{1}\le \h(f_{j})+
d_{j}\log (n+1)+ \sum_{l=1}^m \delta_{l,j}\log(p_l+1)$ and $\bfb\le
\bfa$.

\medskip
By Lemma \ref{lemm:10},   $E\in \Z[\bft,\bfu,z][\bfy]$ is  a
primitive and squarefree polynomial defining the closure of the image of the map
$\psi$ in \eqref{eq:46}. Hence, we can apply
Theorem~\ref{implicitmultiparZ} to bound its partial degrees and
height. We have $\deg_\bfx(q_j)= d_j$, $\deg_{\bft_{l}}(q_j)=
\delta_{l,j}$  and $\h(q_{j})= h_{j}$ for $1\le j\le s$ while
$\deg_\bfx(q_j)= 1$, $\deg_{\bft_{l}}(q_j)= \h(q_{j})= 0$ for
$s+1\le j\le r+1$. The partial degree of $q_{j}$ in the group of
$(r+1)n+1$ variables $\bfu\cup \{z\}$ is equal to 1. Set
$D=\prod_{j=1}^{s}d_j$. Then, a direct application of
Theorem~\ref{implicitmultiparZ} gives, for all $\bfa \in \Supp(E)$,
\begin{equation} \label{eq:60}
\langle \bfd,\bfa\rangle \le D\deg(V) \quad, \quad  \deg_{\bft_l}(\alpha_{\bfa}) + \langle \bfdelta_l,\bfa \rangle
  \le D \deg(V)  \sum_{\ell=1}^{s}\frac{\delta_{l,\ell}}{d_\ell} \quad
  \mbox{for }  1\le l\le m
\end{equation}
and
\begin{multline*}
  \h(\alpha_{\bfa}) + \langle \bfh,\bfa \rangle \le D \bigg(\wh(V) + \deg(V)
  \bigg(\log(r+2)
  + \sum_{\ell=1}^{s}\frac{1}{d_\ell}\bigg(
  h_\ell +\log\big(\#\Supp(f_\ell)+n+2\big) \\ +\sum_{l=1}^m \delta_{l,\ell}\log(p_l+1)\bigg)
  +(r+1-s)\log(n+2)+ (r+1)\log((r+1)n+2)\bigg)\bigg).
\end{multline*}
Using the inequalities $\log\big(\#\Supp(f_\ell)+n+2\big)\le \log\big(\#\Supp(f_\ell)\big) +
  \log(n+3)$ and
$  \log((r+1)n+2)\le 2\log(n+1)$ together with
\begin{displaymath}\log(r+2)+ s \log(n+3) + (r+1-s)\log(n+2) +
2(r+1)\log(n+1) \le
  (3r+4)\log(n+3),\end{displaymath}
we deduce
\begin{multline} \label{eq:59}
  \h(\alpha_{\bfa}) + \langle \bfh,\bfa \rangle \le
D \bigg(\wh(V)
  + \deg(V) \bigg((3r+4)\log(n+3) \\
+\sum_{\ell=1}^{s}\frac{1}{d_\ell}\bigg(h_\ell
+\log\big(\#\Supp(f_\ell)\big)+\sum_{l=1}^m
\delta_{l,\ell}\log(p_l+1)\bigg)\bigg)\bigg).
\end{multline}
Moreover, we get from \eqref{eq:60} that
\begin{displaymath}\log
  \Big(\#\Supp(E) \, 2^{\deg_{\bfy}(E)} \Big) + |\bfa| \log (n )+
  \langle \bfd,\bfa\rangle \log (n+1)
\le 3D\deg(V)\log(n+2).\end{displaymath}
The statement follows  from \eqref{eq:55b},
\eqref{eq:56},  \eqref{eq:57},  \eqref{eq:59}
and this inequality.
\end{proof}

Theorem~\ref{mt} in the introduction follows from the case $m=0$
in the previous result, noticing that for a polynomial
$f\in\Z[\bfx]$ of degree $d$ it holds
$\#(\Supp(f))\le (n+1)^{d}$.

\begin{example}\label{masser} Let $d_1,\dots, d_{n},H\ge
1$ and set
\begin{multline*}
f_1= x_1^{d_1}, f_2= x_1
\, x_n^{d_2-1} - x_2^{d_2}, \\[1mm]
\ldots, f_{n-1}= x_{n-2} \, x_n^{d_{n-1}-1} -
x_{n-1}^{d_{n-1}},
f_n= x_{n-1} \, x_n^{d_n-1} - H
\in  \Z[x_{1},\dots, x_{n}].
\end{multline*}
It is a system of polynomials
without common zeros in~$\Qbarra^{n}$.
Theorem~\ref{arithparam} implies that there is a B\'ezout identity
$ \alpha= g_1 \,f_1 + \cdots +
g_n\,f_n$ with $\deg( g_if_i) \le d_1\cdots d_n$ and
\begin{equation}
  \label{eq:63}
\log(\alpha), \h(g_i)+\h(f_i) \le d_1\cdots d_{n-1}\log(H) +
(4n+7)\log(n+2) d_1\cdots d_n.
\end{equation}
On the other hand, let $u$ be an additional variable and consider the
specialization of any such identity {at $x_{i}=\gamma_{i}$ with
$$ \gamma_1=
H^{d_2\cdots d_{n-1}} \, u^{d_2\cdots d_n-1}, \dots, \gamma_{n-1}= H
\, u^{d_n-1}, \gamma_n = 1/u .$$
We obtain   $\alpha = g_1( \gamma_{1}, \dots, \gamma_{n-1},1/u) \,  H^{d_1\cdots d_{n-1}} \, u^{d_1\cdots
  d_n-d_1}$.} From this, we deduce the lower bounds
  $$ \deg_{x_n}( g_1f_1) \ge d_1\cdots d_n
\quad, \quad \log(\alpha) \ge d_{1}\cdots
d_{n-1}\log (H).$$
Hence, the height
bound in \eqref{eq:63} is optimal up to a term of
size $O(n\log(n)\, d_1\cdots d_n)$.
\end{example}

We next analyze Example~\ref{ejemploresultante} from the point of view
of heights.

\begin{example}\label{resultantez}
  Let $V \subset \A^n_{\Q}$ be a $\Q$-variety of
 pure  dimension $r$. For $1\le j\le r+1$
and $d_{j}\ge 1$, consider again the general $n$-variate polynomial
of degree $d_j$
$$
F_j=\sum_{|\bfa|\le d_j}u_{j,\bfa}\bfx^\bfa \in
\Z[\bfU_j][x_{1},\dots, x_{n}].
$$
It follows from Example~\ref{ejemploresultante} that
there exist $\lambda\in
\Z\setminus \{0\}$ and $g_i\in \Z[\bfU][x_{1},\dots, x_{n}]$ such
that
$$\lambda\,\Res_{d_1,\dots, d_{r+1} }(\ov V)=g_1 F_1+\cdots +
g_{r+1}F_{r+1} \quad \mbox{on } V
$$
satisfying the degree bounds in \eqref{resultanteX}. For the height,
we have that $\deg_\bfx(F_j)=d_j$,
 $\h(F_{j}) =0$ and $\log(\# \Supp(F_{j})) \le
d_{j}\log(n+1)$. Furthermore, $\delta_{l,j}=\deg_{\bfU_l}(F_j)$ equals~1 if $l=j$ and $0$ otherwise. Theorem~\ref{arithparam} then implies
 \begin{displaymath}
   \h(\Res_{d_1,\dots, d_{r+1} }(\ov V)), \h( g_i) \le \bigg(\prod_{\ell=1}^{r+1
   }d_\ell\bigg)\bigg(\wh(V)+ (6r+10)\log(n+3)\deg(V)\bigg).
 \end{displaymath}
\end{example}

The following result is a partial extension of
Theorem~\ref{arithparam}  to an arbitrary number of
polynomials.
For simplicity, we state it for polynomials not depending on
parameters.

\begin{corollary} \label{arithmeticS} Let $V\subset \A^n_{\Q}$ be
  a $\Q$-variety of pure  dimension $r$ and
  $f_1,\dots, f_s\in \Z[\bfx]\setminus \Z$
  without common zeros in $V(\Qbarra)$ such that $V(f_s)$
  intersects $V$ properly.
Set $d_j=\deg(f_j)$ for $1\le
  j\le s$. Assume that $d_1\ge \cdots \ge d_{s-1}$ and that $d_s$ is arbitrary. Set also
  $h=\max_{1\le j\le s-1}\h (f_j)$ and $h_s=\h(f_s)$.
  Then there exist $\alpha\in \Z\setminus \{ 0\} $ and $g_1,\dots ,
  g_s\in \Z[\bfx] $ such that $\alpha=g_1f_1+\cdots+g_sf_s$ on $V$
  with
  \begin{align*}
    \bullet &\  \deg\left(g_if_i\right) \leq \bigg(d_s\prod_{j=1}^{\min\{s-1,r\}}d_j\bigg)\deg(V),\\
    \bullet & \ \h(\alpha), \h(g_i)+\h(f_i) \leq
    \bigg(d_s\prod_{j=1}^{\min\{s-1,r\}}d_j \bigg)\bigg( \wh(V
    )+\deg(V)
    \Bigg( \frac{h_s}{d_s} +  \sum_{\ell=1}^{\min\{s-1,r\}} \frac{h}{d_\ell}\\
    & \hspace{5.8cm} + (6r+9)\log(n+3) + 3r\log(\max\{1,s-r\}) \bigg)\Bigg).
  \end{align*}
\end{corollary}

\begin{proof} The proof follows closely the lines of that of
  Corollary~\ref{parametricS}.
If $s\leq r+1$, the result follows from the case $m=0$ in
Theorem~\ref{multiparametric}.
Hence, we only have to consider the case  $s\ge r+2$.
Set
$$
\overline{f}_j= f_j+v_{j,1}f_{r+1}+\cdots + v_{j,s-r-1}f_{s-1}
\quad \mbox{for } 1\le j\le r,  \quad  \overline{f}_{r+1}=f_s
$$
for {a} group  $\bfv=\{v_{j,i}\}_{j,i}$ of $p:=r(s-r-1)$ variables.
It is  a system of polynomials without common zeros in
$V_{{\Q(\bfv)}}(\overline{\Q(\bfv)})$.  We have that
$\deg_\bfx(\overline{f}_j)=d_{j}$, $\deg_\bfv(\overline{f}_j)=1$ and
$\h(\overline{f}_{j})\le h$ for $1\le j\le r$ while
$\deg_\bfx(\overline{f}_{r+1})=d_s$,
$\deg_\bfv(\overline{f}_{r+1})=0$ and $\h(\overline{f}_{s})=h_{s}$.
By Theorem~\ref{arithparam}, there are $\overline{\alpha}\in
\Z[\bfv]\setminus \{0\}$ and $\overline{g}_i\in \Z[\bfv,\bfx]$ such
that $ \overline{\alpha}=\overline{g}_1\overline{f}_1+\cdots
+\overline{g}_{r+1}\overline{f}_{r+1}$ on $V_{\Q(\overline{\bfv)}}$
with $\deg_\bfx(\overline{g}_i\overline{f}_i) \le
(d_s\prod_{j=1}^{r} d_j)\deg(V)$ and
\begin{multline*}
\h(\overline{\alpha}),\h(\overline{g}_i)+\h(\overline{f}_i)\le
\bigg(d_s\prod_{j=1}^{r} d_j\bigg)\bigg(\wh(V)+\deg(V)\bigg(
\frac{1}{d_s}\Big(h_s+\log(\#\Supp(f_s))\Big)\\
  +\sum_{\ell=1}^r\frac{1}{d_\ell}\Big(h+\log(\#\Supp(\overline{f}_\ell))+
  2\log(p+1)\Big) +(3r+7)\log(n+3) \bigg)\bigg).
\end{multline*}
We have that $\log(p+1)\le \log(n+2)+\log(s-r)$ and
$\#\Supp(\overline{f}_\ell)\le
  (s-r)(n+1)^{d_\ell}$ for $  1\le \ell\le r$ while $\#\Supp(\overline{f}_{r+1})\le
  (n+1)^{d_s}$. Therefore,
\begin{multline} \label{eq:64}
\h(\overline{\alpha}),\h(\overline{g}_i)+\h(\overline{f}_i) \le
\bigg(d_s\prod_{j=1}^{r}
  d_j\bigg)\bigg(\wh(V)+\deg(V)\bigg(\frac{h_s}{d_s} +
  \sum_{\ell=1}^r\frac{h}{d_\ell}  + 3r\log(s-r)\\ +(6r+8)\log(n+3) \bigg)\bigg).
\end{multline}
Set $\wt g_i= \overline{g}_i $ for $ 1\le i\le r$, $\wt g_i=
\sum_{k=1}^{r}v_{k, i-r} \overline{g}_{k}$  for $r+1\le i\le s-1$ and
$\wt g_{s}= \overline{g}_{r+1}$, then
\begin{equation}
  \label{eq:61}
\wt  g_1f_1+\cdots+\wt
  g_{s}f_s = \overline{g}_1\ov f_1+\cdots +
  \overline{g}_{r+1}\ov f_{r+1} =  \overline{\alpha}
\end{equation}
with $\h(\wt g_i)+\h(f_i)\le \max_{k} \{\h( \overline{g}_{k})+\h(\ov
f_k)\}+\log(r)$ for all $i$.  {By taking the coefficients of a
suitable monomial in $\bfv$, we extract from \eqref{eq:61} a
B\'ezout identity $\alpha=g_1f_1+\dots+g_sf_s$ on $V$. By
\eqref{eq:64}, these polynomials satisfy the stated degree and
height bounds.}
\end{proof}

We finally prove  the arithmetic
strong Nullstellensatz {presented} in the introduction.

\begin{proof}[Proof of Theorem \ref{mt_strong}]
  We use the same notations of the proof of {Theorem~\ref{strongNSSpar}} for~$k=\Q$ and {$p=0$}.  The proof
  of this corollary already gives the stated bounds for the degree of
  the $g_{i}$'s and the exponent $\mu$. Hence, it only remains to
  bound the height in the identity $\alpha g^\mu =g_1f_1+\cdots
  +g_sf_s$ on $V$.

  \medskip Corollary~\ref{arithmeticS} applied to $W=V\times \A^1$ and
  $1-y^{d_0}g,f_1,\dots,f_s$ and Lemma~\ref{lemm:7} imply
  \begin{multline} \label{eq:65}
    \h(\alpha), \h (\ov g_i)+\h (f_i) \le 2\,
   d_{0} \, D\bigg( \wh(V )+\deg(V)
    \bigg( \frac{h_0}{2d_0} + \hspace{-3mm} \sum_{\ell=1}^{\min\{s,r+1\}} \frac{h}{d_\ell}\\
     + (6r+15)\log(n+4) + 3(r+1)\log(\max\{1,s-r\}) \bigg)\bigg),
  \end{multline}
with $D:= \prod_{j=1}^{\min\{s,r+1\}} d_j$. {From this, we deduce the}
bound for the height of $\alpha$.
Let $\widehat g_{i}$ be as in
(\ref{gsombrero}). If we write
$\widehat
g_i=\sum_{\bfa,j}\alpha_{\bfa,j} \bfx^\bfa y^j$, then
  $g_i=\sum_{\bfa,j}\alpha_{\bfa,j}\bfx^\bfa g^{\max_l\{\deg_y(\widehat g_l)\}-j}$.
Using Lemma~\ref{lemma:2-17} we deduce
that \begin{align*}
\h(g_i)&\le \max_{\bfa,j}\{\h(\alpha_{\bfa,j} \bfx^\bfa
g^{\max_l\{\deg_y(\widehat g_l)\}-j}) \} + \log(\#\Supp(\widehat
g_i))\\[1mm]
&\le  \h(\widehat g_i)+ \max_l\{\deg_y(\widehat
g_l)\}(h_0+\log(n+1)d_0) + \log(\#\Supp(\widehat
g_i)).
\end{align*}
Observe that $\h(\widehat g_{i})\le \h(\ov g_{i})$ and
$\#\Supp(\widehat g_{i})\le (n+2)^{{\deg_{\bfx,y} (\widehat g_i)}}$.
Moreover,  $\deg_{y} (\widehat g_i)\le 2D\deg(V)$ and $\deg_{\bfx,y}
(\widehat g_i)\le 2d_{0}D\deg(V)$, as shown in the proof of
{Theorem \ref{strongNSSpar}}. Therefore, \nred{$\h(g_i)+\h(f_i)$}
is bounded above by
\begin{displaymath}
\h(\ov g_i)+ \h(f_i) + 2D\deg(V)
(h_0+\log(n+1)d_0) + 2d_{0}D\deg(V)\log(n+2)).
\end{displaymath}
The stated bound for $\h(g_i)+\h(f_i)$ follows from this and \eqref{eq:65}.
\end{proof}

\end{document}

%% file: politopo.pstex_t
\begin{picture}(0,0)%
\includegraphics{politopo.pstex}%
\end{picture}%
\setlength{\unitlength}{1616sp}%
\begingroup\makeatletter\ifx\SetFigFont\undefined%
\gdef\SetFigFont#1#2#3#4#5{%
  \reset@font\fontsize{#1}{#2pt}%
  \fontfamily{#3}\fontseries{#4}\fontshape{#5}%
  \selectfont}%
\fi\endgroup%
\begin{picture}(8487,5877)(2326,-6826)
\put(6391,-4111){\makebox(0,0)[lb]{\smash{{\SetFigFont{9}{10.8}{\rmdefault}{\mddefault}{\updefault}{\color[rgb]{0,0,0}$(0,0,d_1h_2)$}%
}}}}
\put(2701,-6811){\makebox(0,0)[lb]{\smash{{\SetFigFont{9}{10.8}{\rmdefault}{\mddefault}{\updefault}{\color[rgb]{0,0,0} }%
}}}}
\put(4576,-5866){\makebox(0,0)[lb]{\smash{{\SetFigFont{9}{10.8}{\rmdefault}{\mddefault}{\updefault}{\color[rgb]{0,0,0}$(0,d_1,0)$}%
}}}}
\put(6481,-1321){\makebox(0,0)[lb]{\smash{{\SetFigFont{9}{10.8}{\rmdefault}{\mddefault}{\updefault}{\color[rgb]{0,0,0}$t$}%
}}}}
\put(4321,-1771){\makebox(0,0)[lb]{\smash{{\SetFigFont{9}{10.8}{\rmdefault}{\mddefault}{\updefault}{\color[rgb]{0,0,0}$(0,0,d_2h_1)$}%
}}}}
\put(2341,-2851){\makebox(0,0)[lb]{\smash{{\SetFigFont{9}{10.8}{\rmdefault}{\mddefault}{\updefault}{\color[rgb]{0,0,0}$(0,d_1,d_2h_1)$}%
}}}}
\put(9091,-5101){\makebox(0,0)[lb]{\smash{{\SetFigFont{9}{10.8}{\rmdefault}{\mddefault}{\updefault}{\color[rgb]{0,0,0}$(d_2,0,0)$}%
}}}}
\put(9991,-4291){\makebox(0,0)[lb]{\smash{{\SetFigFont{9}{10.8}{\rmdefault}{\mddefault}{\updefault}{\color[rgb]{0,0,0}$y_1$}%
}}}}
\put(3511,-5551){\makebox(0,0)[lb]{\smash{{\SetFigFont{9}{10.8}{\rmdefault}{\mddefault}{\updefault}{\color[rgb]{0,0,0}$y_2$}%
}}}}
\put(9121,-3451){\makebox(0,0)[lb]{\smash{{\SetFigFont{9}{10.8}{\rmdefault}{\mddefault}{\updefault}{\color[rgb]{0,0,0}$(d_2,0,d_1h_2)$}%
}}}}
\end{picture}%

%% file: eliptica.pstex_t
\begin{picture}(0,0)%
\includegraphics{eliptica.pstex}%
\end{picture}%
\setlength{\unitlength}{1699sp}%
\begingroup\makeatletter\ifx\SetFigFont\undefined%
\gdef\SetFigFont#1#2#3#4#5{%
  \reset@font\fontsize{#1}{#2pt}%
  \fontfamily{#3}\fontseries{#4}\fontshape{#5}%
  \selectfont}%
\fi\endgroup%
\begin{picture}(7677,5514)(2686,-6826)
\put(9136,-3346){\makebox(0,0)[lb]{\smash{{\SetFigFont{10}{12.0}{\rmdefault}{\mddefault}{\updefault}{\color[rgb]{0,0,0}$(6,0,5)$}%
}}}}
\put(5356,-5551){\makebox(0,0)[lb]{\smash{{\SetFigFont{10}{12.0}{\rmdefault}{\mddefault}{\updefault}{\color[rgb]{0,0,0}$(0,3,0)$}%
}}}}
\put(9136,-4336){\makebox(0,0)[lb]{\smash{{\SetFigFont{10}{12.0}{\rmdefault}{\mddefault}{\updefault}{\color[rgb]{0,0,0}$(6,0,0)$}%
}}}}
\put(3601,-3391){\makebox(0,0)[lb]{\smash{{\SetFigFont{10}{12.0}{\rmdefault}{\mddefault}{\updefault}{\color[rgb]{0,0,0}$(0,3,8)$}%
}}}}
\put(6571,-2131){\makebox(0,0)[lb]{\smash{{\SetFigFont{10}{12.0}{\rmdefault}{\mddefault}{\updefault}{\color[rgb]{0,0,0}$(0,0,11)$}%
}}}}
\put(2701,-6811){\makebox(0,0)[lb]{\smash{{\SetFigFont{10}{12.0}{\rmdefault}{\mddefault}{\updefault}{\color[rgb]{0,0,0} }%
}}}}
\put(4231,-5191){\makebox(0,0)[lb]{\smash{{\SetFigFont{10}{12.0}{\rmdefault}{\mddefault}{\updefault}{\color[rgb]{0,0,0}$y_2$}%
}}}}
\put(6481,-1591){\makebox(0,0)[lb]{\smash{{\SetFigFont{10}{12.0}{\rmdefault}{\mddefault}{\updefault}{\color[rgb]{0,0,0}$t$}%
}}}}
\put(9361,-4921){\makebox(0,0)[lb]{\smash{{\SetFigFont{10}{12.0}{\rmdefault}{\mddefault}{\updefault}{\color[rgb]{0,0,0}$y_1$}%
}}}}
\end{picture}%